\documentclass[12pt]{amsart}

\usepackage{tikz}
\usepackage{bbm}
\usetikzlibrary{3d,arrows,calc,positioning,decorations.pathreplacing,matrix} 

\usepackage{calligra,mathrsfs}
\usepackage[all]{xy}
\usepackage{float, comment}
\usepackage{mathtools}
\usepackage{amsmath}
\usepackage{amsthm}
\usepackage{amssymb}
\usepackage{amsbsy}
\usepackage{amstext} 
\usepackage{amsopn}
\usepackage[mathscr]{eucal}
\usepackage{enumerate}
\usepackage{xcolor}
\usepackage{graphicx} 
\usepackage{scalerel}
\usepackage{microtype} 
\usepackage[margin=1in,marginparwidth=0.8in, marginparsep=0.1in]{geometry}
\usepackage[bookmarks=true, bookmarksopen=true, bookmarksdepth=3,bookmarksopenlevel=2, colorlinks=true, linkcolor=blue, citecolor=blue, filecolor=blue, menucolor=blue, urlcolor=blue]{hyperref}
\usepackage{tikz}
\usepackage{bbm}
\usepackage[all]{xy}
\usepackage{xspace}
\usetikzlibrary{arrows,calc,positioning,decorations.pathreplacing} 

\DeclareMathOperator{\cHom}{\mathscr{H}\text{\kern -3pt {\calligra\large om}}\,}

\numberwithin{equation}{section}
\newtheorem{Theorem}[equation]{Theorem}
\newtheorem{Proposition}[equation]{Proposition} 
\newtheorem{Lemma}[equation]{Lemma}

\theoremstyle{definition}
\newtheorem{Remark}[equation]{Remark}
\newtheorem{Example}[equation]{Example}
\newtheorem{Definition}[equation]{Definition}

\numberwithin{figure}{section}

\def\O{\mathcal{O}}

\newcommand{\A}{\mathbb{A}}
\newcommand{\C}{\mathbb{C}}

\newcommand{\bN}{\mathbb{N}}

\newcommand{\Q}{\mathbb{Q}}

\newcommand{\cF}{\mathcal{F}}
\newcommand{\cG}{\mathcal{G}}

\newcommand{\cK}{\mathcal{K}}

\newcommand{\cN}{\mathcal{N}}
\newcommand{\cO}{\mathcal{O}}

\newcommand{\catC}{{\mathscr{C}}}
\newcommand{\catD}{{\mathscr{D}}}

\newcommand{\ga}{\gamma}
\renewcommand{\vert}{fwd}
\newcommand{\horiz}{bkwd}
\newcommand{\isom}{isom}

\newcommand{\ot}{\otimes}

\newcommand{\wh}{\widehat}

\newcommand{\conv}{{\mathbin{\scalebox{1.1}{$\mspace{1.5mu}*\mspace{1.5mu}$}}}}

\newcommand{\nfp}{n\text{-}fp}

\newcommand{\id}{{id}}

\newcommand{\hR}{{\mathcal{R}}}

\newcommand{\QCoh}{\mathrm{QCoh}}
\newcommand{\Coh}{\mathrm{Coh}}
\newcommand{\IndCoh}{\mathrm{IndCoh}}

\newcommand{\Fun}{\mathrm{Fun}}

\newcommand{\PreStkk}{\mathrm{PStk}_{\kk}}
\newcommand{\PreStkkleqn}{\mathrm{PStk}_{\kk, \leq n}}
\newcommand{\PreStkkconv}{\widehat{\mathrm{PStk}}_{\kk}}

\newcommand{\Stkk}{\mathrm{Stk}_\kk}
\newcommand{\Stkkleqn}{\mathrm{Stk}_{\kk, \leq n}}
\newcommand{\Stkkleqzero}{\mathrm{Stk}_{\kk, \leq 0}}

\newcommand{\Stkkconv}{\widehat{\mathrm{Stk}}_{\kk}}

\newcommand{\Spc}{\mathscr{S}}

\newcommand{\Ind}{\mathrm{Ind}}
\newcommand{\CAlg}{\mathrm{CAlg}}
\newcommand{\CAlgk}{\mathrm{CAlg}_\kk}
\newcommand{\CAlgkleqn}{\tau_{\leq n}\mathrm{CAlg}_{\kk}}
\newcommand{\CAlgkleqinfty}{\tau_{< \infty} \mathrm{CAlg}_{\kk}}
\newcommand{\CAlgvarleqn}[1]{\tau_{\leq n}\mathrm{CAlg}_{#1}}
\newcommand{\CAlgvarleqinfty}[1]{\tau_{< \infty}\mathrm{CAlg}_{#1}}

\newcommand{\Corr}{\mathrm{Corr}}

\newcommand{\op}{\mathrm{op}}

\newcommand{\Mod}{\mathrm{Mod}}
\newcommand{\kk}{k}
\newcommand{\indGStk}{\mathrm{indGStk}}
\newcommand{\indGStkk}{\mathrm{indGStk}_{\kk}}
\newcommand{\indGStkkreas}{\mathrm{indGStk}_{\kk}^{reas}}
\newcommand{\indGStkkadtm}{\mathrm{indGStk}_{\kk}^{tm,ad}}
\newcommand{\indGStkkcoh}{\mathrm{indGStk}_{\kk}^{coh}}

\newcommand{\indGStkkprop}{\indGStk_{\kk,\,prop}}

\newcommand{\GStk}{\mathrm{GStk}}
\newcommand{\GStkk}{\mathrm{GStk}_{\kk}}
\newcommand{\GStkkplus}{\mathrm{GStk}_{\kk}^+}

\newcommand{\PrSt}{\mathcal{P}\mathrm{r}^{\mathrm{St}}}

\newcommand{\LurieCathatinfty}{\widehat{\mathrm{Cat}}_{\infty}}
\newcommand{\Cathatinfty}{\mathrm{Cat}_{\infty}}

\newcommand{\Catinfty}{\mathrm{Cat}_{\infty}}

\newcommand{\LFun}{\mathrm{LFun}}
\newcommand{\al}{\alpha}
\newcommand{\be}{\beta}

\newcommand{\PrL}{\mathcal{P}\mathrm{r}^{\mathrm{L}}}

\newcommand{\Map}{\mathrm{Map}}
\newcommand{\Spec}{\mathrm{Spec}\,}
\newcommand{\stable}{stable\xspace}

\newcommand{\clalg}[1]{H^0(#1)}

\newcommand{\whC}{\wh{\catC}}
\newcommand{\whD}{\wh{\catD}}
\newcommand{\wcC}{\wc{\catC}}
\newcommand{\wcD}{\wc{\catD}}

\newcommand{\PrStbacpl}{\mathcal{P}\mathrm{r}^{\mathrm{St,b}}_{\mathrm{acpl}}}

\newcommand{\igpropgstkprops}{Prop. 3.2}
\newcommand{\igpropftdprops}{Prop. 3.6}
\newcommand{\igsecpropafp}{Sec. 3.6}
\newcommand{\igpropfcdprops}{Prop. 3.14}
\newcommand{\igproppropprops}{Prop. 3.19}
\newcommand{\igpropgstkconvergent}{Prop. 3.25}

\newcommand{\igpropfactorthroughgeometric}{Prop. 4.6}
\newcommand{\igpropigprestermsarereas}{Prop. 4.8}
\newcommand{\igpropindPequivconditions}{Prop. 4.9}
\newcommand{\igpropindPcompprops}{Prop. 4.13}
\newcommand{\igpropindPtwoofthreeprops}{Prop. 4.14}
\newcommand{\igpropafpdefconsistent}{Prop. 4.16}
\newcommand{\igpropcohfiltcolims}{Prop. 4.19}
\newcommand{\igpropindgeomfiberprods}{Prop. 4.20}

\newcommand{\igpropindPbaseprops}{Prop. 4.23}
\newcommand{\igpropcohbasechangeunderafpclimm}{Prop. 4.24}
\newcommand{\iglemaffcgen}{Lem. 4.25}

\newcommand{\igseccohsubsec}{Sec. 5.1}
\newcommand{\igsecanticompletion}{Sec. 5.2}
\newcommand{\igpropcohonindgstks}{Prop. 5.5}
\newcommand{\igdefIndCohononeGStk}{Def. 5.10}
\newcommand{\igpropindcohonindgstks}{Prop. 5.19}
\newcommand{\igpropIndCohisIndofCoh}{Prop. 5.29}
\newcommand{\igpropuppershriekftdbasechange}{Prop. 6.2}
\newcommand{\igcoruppershriekftdbasechangeIndCoh}{Cor. 6.5}

\newcommand{\igsecsheafHom}{Sec. 7}
\newcommand{\igsecindextcohcase}{Sec. 7.3}
\newcommand{\igpropcHomftdbasechange}{Prop. 7.3}
\newcommand{\igpropreascohlaxsymmstruct}{Prop. 7.20}

\newcommand{\igpropeFXRpropsone}{Prop. 7.39}
\newcommand{\igcorcHomftdbasechangeind}{Cor. 7.41}
\newcommand{\igpropeFXRlowerstar}{Prop. 7.48}
\newcommand{\igpropindhomrelationcoherent}{Prop. 7.52}

\newcommand{\crpropcohrigidity}{Prop. 6.5}
\newcommand{\crpropcohdualformula}{Prop. 6.8}

\DeclareFontFamily{U}{mathx}{\hyphenchar\font45}
\DeclareFontShape{U}{mathx}{m}{n}{
	<5> <6> <7> <8> <9> <10>
	<10.95> <12> <14.4> <17.28> <20.74> <24.88>
	mathx10
}{}
\DeclareSymbolFont{mathx}{U}{mathx}{m}{n}
\DeclareFontSubstitution{U}{mathx}{m}{n}
\DeclareMathAccent{\widecheck}{0}{mathx}{"71}
\newcommand{\wc}{\widecheck}
\newcommand{\Sp}{\mathrm{Sp}}
\DeclareMathSymbol{\shortminus}{\mathbin}{AMSa}{"39}

\DeclareRobustCommand{\SkipTocEntry}[5]{}

\DeclareMathOperator*{\colim}{colim}

\newcommand{\congto}{\xrightarrow{\sim}}

\newcommand{\arrtip}{latex'}

\begin{document}
\title{Tamely presented morphisms and coherent pullback}

\author[Sabin Cautis]{Sabin Cautis}
\address[Sabin Cautis]{University of British Columbia \\ Vancouver BC, Canada}
\email{cautis@math.ubc.ca}

\author[Harold Williams]{Harold Williams}
\address[Harold Williams]{University of Southern California \\ Los Angeles CA, USA}
\email{hwilliams@usc.edu}

\begin{abstract}
	We study two classes of morphisms in infinite type: tamely presented morphisms and morphisms with coherent pullback. These are generalizations of finitely presented morphisms and morphisms of finite Tor-dimension, respectively. 
	The class of tamely presented schemes and stacks is restricted enough to retain the key features of finite-type schemes from the point of view of coherent sheaf theory, but wide enough to encompass many infinite-type examples of interest in geometric representation theory. 
	The condition that a diagonal has coherent pullback is a natural generalization of smoothness to the tamely presented setting, and we show such objects retain many good cohomological properties of smooth varieties. 
	Our results are motivated by the study of convolution products in the double affine Hecke category and related categories in the theory of Coulomb branches. 
\end{abstract}

\maketitle

\setcounter{tocdepth}{1}

\tableofcontents

\section{Introduction}
\thispagestyle{empty}

This paper and \cite{CWig} provide a collection of foundational results on coherent sheaf theory in infinite-dimensional derived algebraic geometry. These results are applied in \cite{CW2} to the study of Coulomb branches of 4d $\cN=2$ gauge theories, whose mathematical theory was pioneered by Braverman-Finkelberg-Nakajima. In \cite{BFN} these authors associate to a reductive group $G$ and representation $N$ an ind-scheme $\hR_{G,N}$ with an action of the jet group $G_\cO$. The relevant Coulomb branch is the affine variety whose coordinate ring is the equivariant $K$-theory $K^{G_\cO}(\hR_{G,N})$, equipped with a certain convolution product. In \cite{CW2} we study the derived category $\Coh^{G_\cO}(\hR_{G,N})$ of equivariant coherent sheaves and construct a nonstandard t-structure on it, which in turn equips $K^{G_\cO}(\hR_{G,N})$ with a canonical basis.

By construction, the category $\Coh^{G_\cO}(\hR_{G,N})$ acts on $\Coh^{G_\cO}(N_\cO)$ by integral transforms, and convolution in $\Coh^{G_\cO}(\hR_{G,N})$ is an example of the convolution of kernels. Integral transforms of coherent sheaves are a well-studied topic (e.g. \cite{Huy06,BZNP17}), and are of representation-theoretic importance through constructions of the above kind (e.g. \cite{KL87,Bez16}). However, $\hR_{G,N}$ and $N_\cO$ being of infinite type means that many standard results in the literature cannot be directly applied to $\Coh^{G_\cO}(\hR_{G,N})$. This poses technical challenges in \cite{CW2}, for example in proving that $\Coh^{G_\cO}(\hR_{G,N})$ is rigid as a monoidal category. The argument requires that certain pullbacks of coherent sheaves commute with sheaf Hom and !-pullback, but the closest results available (e.g. \cite[Lem. 0AA7]{Sta} or \cite[Prop. 7.1.6]{Gai13a}) do not apply. The motivating goal of this paper is to provide the needed infinite-type extensions of such results. 

In carrying this out we face the following complication. Let us say a morphism $f: X \to Y$ has \emph{coherent pullback} if $f^*: \QCoh(Y) \to \QCoh(X)$ takes $\Coh(Y)$ to $\Coh(X)$, where $\QCoh(-)$ denotes the unbounded derived category of quasi-coherent sheaves and $\Coh(-)$ the bounded derived category of coherent sheaves. In finite type such morphisms have many good properties, such as stability under (derived) base change. This is because in finite type they coincide with morphisms of finite Tor-dimension. In infinite type, however, morphisms with coherent pullback can be of infinite Tor-dimension, and their good properties in finite type do not fully extend --- for example, they are no longer stable under arbitrary base change (Example~\ref{ex:originintoAinfty}). In particular, we cannot expect naive extensions of the finite-type results we want to apply to $\Coh^{G_\cO}(\hR_{G,N})$, but only extensions with suitable hypotheses to avoid pathologies specific to infinite type. 

These extensions turn out to be most naturally formulated using the notion of \emph{tamely presentedness}. For example, morphisms with coherent pullback among tamely presented schemes do turn out to be stable under base change along tamely presented morphisms. Locally, we say an algebra $A$ is (strictly) tamely presented if it is the union of the finitely presented subalgebras over which it is flat. Tamely presented schemes generalize finitely presented schemes in a way which is restrictive enough to retain their good formal and sheaf-theoretic properties, but broad enough to include the examples relevant to $\hR_{G,N}$.  In the rest of the introduction we give more details first on the main technical themes developed in the paper, and then on the relationship of our work with the literature. 

\subsection{Technical overview} 
We refer to Section \ref{sec:convnot} for our detailed conventions. For now the reader may take $\kk$ to be a field of characteristic zero and $\CAlgk$ the (enhanced homotopy) category of nonpositively graded commutative dg $\kk$-algebras. 

In Section \ref{sec:affine} we study {\it strictly tamely presented algebras}. 
If $A \in \CAlgk$ is an ordinary Noetherian ring, an ordinary $A$-algebra $B$ is strictly tamely presented if it is the union of the finitely presented subalgebras over which it is flat. If $A$ and $B$ are more general, we ask that each truncation of $B$ satisfies a similar condition, and we remove the word ``strictly'' to mean these conditions are satisfied up to a suitable flat cover. Beyond basic stability properties, the key results of this section are that tamely presented $\kk$-algebras are coherent (Proposition~\ref{prop:stablecoh}), that morphisms with coherent pullback are stable under tamely presented base change (Proposition \ref{thm:reasonableaffinecohbasechange}), and that a tamely presented algebra over a Noetherian base is of finite Tor-dimension if and only if it has coherent pullback (Proposition \ref{prop:finitetor2}). 

Section \ref{sec:reasgeomstacks} studies tamely presented geometric stacks. As in \cite[Sec. 9]{LurSAG}, a geometric stack will mean a functor $X: \CAlgk \to \Spc$ which satisfies flat descent, has affine diagonal, and admits a flat cover $\Spec A \to X$ (here $\Spc$ is the $\infty$-category of spaces). We say $X$ is tamely presented if we can choose this cover so that $\Spec A$ is strictly tamely presented over both $X$ and $\Spec \kk$. Such stacks are intermediate between general geometric stacks and Artin stacks of finite type. The basic example is the quotient of a tamely presented scheme by the action of an affine group scheme (Proposition \ref{prop:tamequotients}). 

We say a morphism has \emph{\stable coherent pullback} if it has coherent pullback after base change along any tamely presented morphism. For a morphism of tamely presented schemes this is the same as simply having coherent pullback. For a morphism $f: X \to Y$ of tamely presented geometric stacks, it is equivalent (via Propositions \ref{thm:reasonableaffinecohbasechange} and \ref{prop:cohpullprops}) to asking that the base change $f'$ along any fixed tamely presented flat cover $\Spec A \to Y$ has coherent pullback. A key fact is that, under these hypotheses, $f$ can in a certain sense be approximated by morphisms of finite Tor-dimension  (Proposition~\ref{prop:finitetorapproxgstk}). 

Recall that if $X$ and $Y$ are varieties and $Y$ is smooth, then any morphism $f: X \to Y$ is of finite Tor-dimension. This follows from the fact that the diagonal $\Delta_Y: Y \to Y \times Y$ is of finite Tor-dimension, a property which characterizes smooth varieties. With this in mind, we say a tamely presented stack $Y$ is {\it weakly smooth} if $\Delta_Y$ has \stable coherent pullback (note that a variety is weakly smooth if and only if it is smooth). By Proposition \ref{prop:wprosmoothcohdiagb} this follows if~$Y$ has a pro-smooth atlas (e.g. if $Y = N_\cO$), but the notion of weak smoothness has the virtue of being intrinsic. If $X$ and $Y$ are tamely presented stacks and $Y$ is weakly smooth, any tamely presented morphism $f: X \to Y$ has coherent pullback. For this reason, many of our results can be understood as quantifying the extent to which weakly smooth schemes and stacks retain the good cohomological properties of smooth varieties. 

In Section \ref{sec:indgeom} we extend our main notions to the setting of ind-geometric stacks. Ind-geometric stacks generalize  ind-schemes, in particular the notion of dg ind-scheme introduced in \cite{GR14}, Their basic theory is developed in \cite{CWig}. As a key motivating example, the quotient $\hR_{G,N}/G_{\cO}$ is an ind-tamely presented ind-geometric stack.

In Section \ref{sec:upper!} we study the interaction between coherent pullback and $!$-pullback in the tamely presented setting. The main result (Proposition \ref{prop:up!up*genICcaseIG}) is the following generalization of \cite[Prop. 7.1.6]{Gai13a}. Suppose we have a Cartesian diagram
\begin{equation*}
	\begin{tikzpicture}
		[baseline=(current  bounding  box.center),thick,>=\arrtip]
		\node (a) at (0,0) {$X'$};
		\node (b) at (3,0) {$Y'$};
		\node (c) at (0,-1.5) {$X$};
		\node (d) at (3,-1.5) {$Y$};
		\draw[->] (a) to node[above] {$f' $} (b);
		\draw[->] (b) to node[right] {$h $} (d);
		\draw[->] (a) to node[left] {$h' $}(c);
		\draw[->] (c) to node[above] {$f $} (d);
	\end{tikzpicture}
\end{equation*}
of coherent, ind-tamely presented ind-geometric stacks in which $h$ has \stable coherent pullback and $f$ is ind-proper and almost ind-finitely presented. Then the Beck-Chevalley transformation $h'^* f^! \to f'^! h^*$ of functors $\IndCoh(Y) \to \IndCoh(X')$ is an isomorphism. 

In Section \ref{sec:sheafHom} we study the interaction between coherent pullback and sheaf Hom in the tamely presented setting. The main result (Proposition \ref{prop:eFXRupper*cohpullindgeomcase}) is the following generalization of \cite[Lem. 0AA7]{Sta}. Suppose $h: X \to Y$ is a morphism with \stable coherent pullback between coherent, ind-tamely presented ind-geometric stacks such that $X \times X$ and $Y \times Y$ are coherent. Then for any $\cF \in \Coh(Y)$ the Beck-Chevalley transformation
\begin{equation*}
	h^* \cHom(\cF, -) \to \cHom(h^*(\cF), h^*(-))
\end{equation*}
of functors $\IndCoh(Y) \to \IndCoh(X)$ is an isomorphism. 

Finally, in Section \ref{sec:kernels} we discuss integral transforms in infinite-type settings such as that of \cite{CW2}, including the adjointability results made possible by the results of Sections~\ref{sec:upper!} and~\ref{sec:sheafHom} (though we defer the complete proofs of these to \cite{CW2}). 

\subsection{Relations to existing literature}\label{sec:literature}
Let us briefly survey the literature on which we build most directly. It is a classical fact that a direct limit of Noetherian rings along flat morphisms is coherent \cite[Sec. I.2 Ex. 12]{Bou72}. The notion of tamely presented algebra adapts and relativizes this construction. Informally, tamely presented algebras and coherent rings have a relationship similar to that of finitely presented algebras and Noetherian rings. We emphasize that tamely presented algebras do not suffer from the pathologies which make working with general coherent rings difficult. They satisfy a Hilbert basis theorem and more generally they are closed under tensor products, neither of which is true of arbitrary coherent rings \cite[Sec. 7.3.13]{Gla89}. 

Tamely presentedness is also a relative of placidity. A classical scheme $X$ is placid \cite[Def. 4.2.1]{Ras14b} (or almost smooth \cite[Def. 3.2.4]{KV04}) if it can be written as an inverse limit of finite-type schemes along smooth affine morphisms (some references require surjectivity). Replacing classical with truncated (i.e. eventually coconnective) and finite-type with almost finite-type yields a derived notion of placidity considered in \cite[Sec. 4.1]{Gai14}. Further replacing smooth with flat results in a less general variant of tamely presentedness which appears in \cite[Prop. 6.36.4]{Ras20}. In particular, any placid scheme is tamely presented. Though the converse is false, all tamely presented schemes in \cite{CW2} are in fact placid (in the derived sense \cite[Sec. 4.1]{Gai14}). Thus, it would suffice for our applications to state all results in the placid setting, bypassing the notion of tamely presentedness and minimizing the number of new notions introduced. But a few reasons have led us not to do this.

First, the smoothness in the definition of placid would never be used in any of our arguments, only its implied flatness. It would be at best awkward and at worst misleading to include a hypothesis in every step of our constructions that would never be used. This is not to say that placidity is not an important notion, as it implies certain self-duality properties that tamely presentedness alone does not (e.g. \cite[Sec. 4.7]{Ras14b} or \cite[Sec. 4.4.8]{Gai14}). But these stronger properties go beyond the scope of what our results depend on. 

Second, tamely presentedness is a more intrinsic notion than placidity. That is, the definition of placidity refers to a non-canonical choice of a presentation as an inverse limit. Whereas asking that $A$ be the union of all finitely presented subalgebras $A_\al$ over which it is flat refers to no choices, only the \emph{canonical} presentation of $\Spec A$ as the inverse limit of the $\Spec A_\al$. Note that if we replace flat with smooth in this condition it does not become equivalent to placidity, but instead degenerates to the condition that $A$ is finitely presented. We find the intrinsic nature of tamely presentedness to be a compelling conceptual feature (though as a practical side effect the proofs of some basic stability properties are slightly more direct for tamely presentedness than for placidity). 

Placid geometric stacks and their topological sheaf theory are studied in detail in \cite{BKV22}. This notion is parallel to the notion of tamely presented geometric stack we consider, though some differences (e.g. choice of Grothendieck topology, presence or absence of derived structures) stem from our different sheaf-theoretic interests. Similarly, the notion of placidly stratified stack from \cite{BKV22} is parallel (though in a looser way) to the notion of ind-tamely presented ind-geometric stack. 

As a final aside, we note that formulating a well-behaved \emph{relative} notion of placidity in the derived setting (which goes beyond the scope of \cite[Sec. 4.1]{Gai14} or \cite[Sec. 6]{Ras20}) would raise some technical issues which we address here in the context of tamely presentedness. These concern subtleties of derived Noetherian approximation, see the discussion before Example \ref{ex:weirdex}. 

\addtocontents{toc}{\SkipTocEntry}
\subsection*{Acknowledgements}
We are deeply grateful to Sam Raskin, Hiro Lee Tanaka, Aaron Mazel-Gee, and Chang-Yeon Chough for taking the time to discuss numerous technical issues that arose in the preparation of this paper and its companions \cite{CWig, CW2}. S.C. was supported by NSERC Discovery Grant 2019-03961 and H. W. was supported by NSF grants DMS-1801969 and DMS-2143922. 

\section{Conventions}\label{sec:convnot}

We collect here our notational and terminological conventions. Our default references for categorical and geometric background are \cite{LurHTT,LurHA,LurSAG}, and we follow their conventions up to a few exceptions noted below. 

\begin{itemize}
	\item We use the terms category and $\infty$-category interchangeably, and say ordinary category when we specifically mean a category in the traditional sense. We write $\Map_\catC(X,Y)$ for the mapping space between $X, Y \in \catC$, and regard ordinary categories as $\infty$-categories with discrete mapping spaces. 
	
	\item We write $\CAlg(\catC)$ for the category of commutative algebra objects of a monoidal category $\catC$. We write $\CAlg$ for $\CAlg(\Sp^{cn})$, where $\Sp^{cn}$ is the category of connective spectra (this would be $\CAlg^{cn}$ in \cite{LurSAG}). 
	
	\item We fix once and for all a Noetherian base $\kk \in \CAlg$.
	
	\item We use cohomological indexing for t-structures. If $\catC$ has a t-structure $(\catC^{\leq 0}, \catC^{\geq 0})$ with heart $\catC^\heartsuit$ we write $\tau^{\leq n} : \catC \to \catC^{\leq n}$, $H^n: \catC \to \catC^\heartsuit$, etc., for the associated functors. In this notation, the condition that $\kk$ is Noetherian is the condition that $H^0(\kk)$ is an ordinary Noetherian ring and $H^n(\kk)$ is finitely generated over $H^0(\kk)$ for all $n < 0$. We use the terms left bounded and right bounded interchangeably with (cohomologically) bounded below and bounded above. 
	
	\item We write $\tau_{\leq n} \catD$ for the subcategory of $n$-truncated objects in an $\infty$-category $\catD$. In particular, $\tau_{\leq 0} \CAlg$ is the ordinary category of ordinary commutative rings. Note the distinction between subscripts and superscripts in this and the previous convention, e.g. $\tau_{\leq n}(\catC^{\leq 0})$ and $\catC^{[-n, 0]}$ refer to the same subcategory of $\catC$. 
	
	\item Given $A \in \CAlg$, we write $\Mod_A$ for the category of $A$-modules (i.e. $A$-module objects in the category of spectra). If $A$ is an ordinary ring this is the (enhanced) unbounded derived category of ordinary $A$-modules (i.e. of $\Mod_A^{\heartsuit}$).   
	
	\item An $A$-module $M$ is coherent if it is bounded and almost perfect (i.e. $\tau^{\geq n}M$ is compact in $\Mod_A^{\geq n}$ for all $n$). If $A$ is coherent (i.e. $\clalg{A}$ is a coherent ordinary ring and $H^n(A)$ is a finitely presented $\clalg{A}$-module for all $n$), then $M$ is coherent if and only if it is bounded and $H^n(M)$ is a finitely presented $\clalg{A}$-module for all $n$. We write $\Coh_A \subset \Mod_A$ for the full subcategory of coherent modules. 
	
	\item Given $A \in \CAlg$, we write $\CAlg_A := \CAlg_{A/} \cong \CAlg(\Mod_A^{\leq 0})$, and write $\CAlgvarleqinfty{A} := \cup_n \CAlgvarleqn{A}$ for the subcategory of truncated $A$-algebras (i.e. $n$-truncated for some~$n$). If $A$ is an ordinary ring containing $\Q$, then $\CAlg_A$ is equivalently the (enhanced homotopy) category of nonpositively graded commutative dg $A$-algebras, or of simplicial/animated commutative $A$-algebras.  
	
	\item  We implicitly fix two universes and associated category sizes: small and large. We write $\Cathatinfty$ for the $\infty$-category of large $\infty$-categories (in \cite{LurHTT} this would be $\LurieCathatinfty$, and $\Cathatinfty$ would be its subcategory of small $\infty$-categories). We write $\PrL \subset \Cathatinfty$ for the subcategory of presentable $\infty$-categories and left adjoints, and $\PrSt \subset \PrL$ for the further subcategory of presentable stable $\infty$-categories. 
	
	\item Given categories $\catC$ and $\catD$, we write both $\Fun(\catC, \catD)$ and $\catD^\catC$ for the category of functors from $\catC$ to $\catD$. 
	
	\item All limit or colimit diagrams are implicitly small unless otherwise stated. Thus in ``let $X \cong \colim X_\al$ be a filtered colimit'' the indexing diagram is assumed to be small. By extension, $\Ind(\catC)$ will refer to the category freely generated by $\catC$ under small filtered colimits even if $\catC$ is large (as in \cite[Def. 21.1.2.5]{LurSAG}). 
	
	\item If $\catC$ admits filtered colimits, a functor $F: \catC \to \catD$ is continuous if it preserves them. Suppose further that $\catC$, $\catD$ are presentable, stable, and equipped with t-structures that are compatible with filtered colimits, and that $F$ is exact. Then $F$ is almost continuous if its restriction to $\catC^{\geq n}$ is continuous for all $n$ (equivalently, for $n = 0$). 
	
	\item A prestack (implicitly over $\Spec \kk$) is a functor from $\CAlgk$ to the category of (possibly large) spaces. We write $\PreStkk$ for the category of prestacks, and $\PreStkkconv$, $\PreStkkleqn$ for the variants with $\CAlgkleqinfty$, $\CAlgkleqn$  in place of $\CAlgk$. We write $\Spec: \CAlgk \to \PreStkk$ for the Yoneda embedding.  
	
	\item A stack is a prestack which is a sheaf for the fpqc topology \cite[Prop. B.6.1.3]{LurSAG}. We write $\Stkk \subset \PreStkk$ for the category of stacks, and $\Stkkconv \subset \PreStkkconv$, $\Stkkleqn \subset \PreStkkleqn$ for its variants. (Note that $\CAlgkleqinfty$ does not admit arbitrary pushouts, but the use of \cite[Prop. A.3.2.1]{LurSAG} in defining the fpqc topology only requires closure under flat pushouts.)
	
	\item If $\catC$ admits finite limits, we write $\Corr(\catC)$ for the $\infty$-category of correspondences in $\catC$ (e.g. \cite[Def. 3.3]{Bar13}, \cite[Sec. 7.1.2]{GR17}). This has the same objects as $\catC$, but a morphism from $X$ to $Z$ in $\Corr(\catC)$ is a diagram $X \xleftarrow{h} Y \xrightarrow{f} Z$ in $\catC$. 
	
	\item Let $\vert$ and $\horiz$ be classes of morphisms in $\catC$ which contain all isomorphisms and are stable under composition, and under base change along each other. Suppose also that $\catC' \subset \catC$ is a full subcategory such that $Y \in \catC'$ whenever $h: Y \to X$ is in $\horiz$ and $X \in \catC'$. Then we write $\Corr(\catC')_{\vert,\horiz}$ for the 1-full subcategory of $\Corr(\catC)$ which only includes correspondences $X \xleftarrow{h} Y \xrightarrow{f} Z$ such that $h \in \horiz$, $f \in \vert$, and $X, Z \in \catC'$ (hence $Y \in \catC'$). Note that $\catC'$ need not be closed under arbitrary pullbacks. The subcategory $\Corr(\catC')_{\vert,\isom} \subset \Corr(\catC')_{\vert,\horiz}$ which only includes correspondences in which $h$ is an isomorphism is equivalent to the 1-full subcategory $\catC'_{\vert} \subset \catC'$, likewise for the subcategory where $f$ is an isomorphism and $\catC'^{\op}_{\horiz} \subset \catC'^{\op}$. 
	
	\item We presume our constructions and results remain valid if we replace $\CAlg$ with the category $\CAlg^\Delta$ of simplicial/animated commutative rings. We work with $\CAlg$ mostly to make some references easier to pinpoint.  On the other hand, any derived prestack has an underlying spectral prestack, and by definition these share the same category of quasi-coherent sheaves. Since our focus is on sheaves, it is in this sense more natural to work in the spectral setting. This distinction is also irrelevant to our intended applications, in which our base is $\C$ and we have $\CAlg_\C \cong \CAlg_\C^\Delta$. 
\end{itemize}

\section{Affine schemes}\label{sec:affine}

We begin by studying (strictly) tamely presented morphisms and coherent pullback in the setting of affine schemes. We first establish the basic stability and coherence properties  of the latter (Propositions \ref{prop:reasstabprops2} and \ref{prop:stablecoh}). We then show that morphisms with coherent pullback and strictly tamely presented target are stable under strictly tamely presented base change (Theorem \ref{thm:reasonableaffinecohbasechange}). We also show that a strictly tamely presented algebra over a Noetherian base has coherent pullback if and only if it is of finite Tor-dimension (Proposition \ref{prop:finitetor2}). 

\subsection{Strictly tame presentations}
Given $A \in \CAlgk$, an $A$-algebra $B$ is finitely $n$-presented if it is $n$-truncated and compact in $\CAlgvarleqn{A}$, and is almost of finite presentation if $\tau_{\leq n} B$ is finitely $n$-presented for all $n$. In particular, an $A$-algebra $B$ is finitely zero-presented if and only if it is an ordinary commutative ring and is finitely presented over $\tau_{\leq 0} A = \clalg{A}$ in the ordinary sense. Also recall that $B$ is flat over $A$ if $H^0(B)$ is flat over $H^0(A)$ in the ordinary sense and the natural map $H^0(B) \ot_{H^0(A)} H^n(A) \to H^n(B)$ is an isomorphism for all $n$. 

\begin{Proposition}\label{prop:reasalg}
Given $A \in \CAlgk$, the following conditions on $B \in \CAlg_A$ are equivalent.
\begin{enumerate}
	\item We can write $B$ as a filtered colimit $B \cong \colim B_\al$ of finitely $n$-presented $A$-algebras such that $B$ is flat over each $B_\al$. 
	\item For every finitely $n$-presented $A$-algebra $C$, every morphism $C \to B$ in $\CAlg_A$ factors through a flat morphism $C' \to B$ such that $C'$ is also finitely $n$-presented over~$A$.
\end{enumerate}
\end{Proposition}

\begin{Definition}\label{def:reasalg}
	Given $A \in \CAlgk$, we say $B \in \CAlg_A$ is \emph{strictly tamely $n$-presented} if it satisfies the equivalent conditions of Proposition \ref{prop:reasalg}. We call an expression $B \cong \colim B_\al$ as in condition (1) a \emph{strictly tame presentation of order $n$}. We say $B$ is \emph{strictly tamely presented} if $\tau_{\leq n}B$ is strictly tamely $n$-presented for all $n$. 
\end{Definition}

\begin{proof}[Proof of Proposition \ref{prop:reasalg}]
	That (1) implies (2) follows from compactness of $C$ in $\CAlgvarleqn{A}$. To see that (2) implies (1), let $\CAlg^{\nfp}_A$ denote the category of finitely $n$-presented $A$-algebras, and let $(\CAlg^{\nfp}_A)_{/f\text{-}B} \subset (\CAlg^{\nfp}_A)_{/B} = \CAlg^{\nfp}_A \times_{\CAlg_A} (\CAlg_A)_{/B}$ be the full subcategory of algebras over which $B$ is flat. 
	It suffices to show $(\CAlg^{\nfp}_A)_{/f\text{-}B}$ is filtered and that~$B$ is the colimit over its forgetful functor to $\CAlg_A$. 
	
	Since $\CAlg_A$ is compactly generated so is $\CAlgvarleqn{A}$ \cite[Cor. 5.5.7.4]{LurHTT}, hence $(\CAlg^{\nfp}_A)_{/B}$ is filtered and $B$ is the colimit over its forgetful functor to $\CAlgvarleqn{A}$ \cite[Cor. 5.3.5.4, Cor. 5.5.7.4]{LurHTT}. For any $B' \in \CAlg^{\nfp}_A$, we have $(\CAlg^{\nfp}_A)_{B'/} \cong \CAlg^{\nfp}_{B'}$ by \cite[Prop. 4.1.3.1]{LurSAG} and the fact that being finitely $n$-presented is equivalent to being $n$-truncated and of finite generation to order $n+1$ \cite[Rem. 4.1.1.9]{LurSAG}. It thus suffices to show $(\CAlg^{\nfp}_{B'})_{/f\text{-}B}$ is filtered for all $B' \in (\CAlg^{\nfp}_A)_{/B}$, since then $(\CAlg^{\nfp}_A)_{/f\text{-}B} \to (\CAlg^{\nfp}_A)_{/B}$ is left cofinal by \cite[Thm. 4.1.3.1, Lem. 5.3.1.18]{LurHTT}, and $(\CAlg^{\nfp}_A)_{/f\text{-}B}$ is filtered as a special case. 
	
	We must show any finite diagram $K \to (\CAlg^{\nfp}_{B'})_{/f\text{-}B}$ extends to a diagram $K^{\triangleright} \to (\CAlg^{\nfp}_{B'})_{f\text{-}B}$. Since $(\CAlg^{\nfp}_{B'})_{/B}$ is filtered, we have an extension $K^{\triangleright} \to (\CAlg^{\nfp}_{B'})_{/B}$. The image $B''$ of the cone point is finitely $n$-presented over $A$, hence by hypothesis its morphism to $B$ factors through some $C \in (\CAlg^{\nfp}_{A})_{/f\text{-}B}$. But $C$ is also finitely $n$-presented over $B'$ (again by \cite[Prop. 4.1.3.1]{LurSAG}), hence composing with $B'' \to C$ we obtain the desired extension $K^{\triangleright} \to (\CAlg^{\nfp}_{B'})_{/f\text{-}B}$. 
\end{proof}

A strictly tamely $n$-presented algebra is $n$-truncated, since $\CAlgvarleqn{A}$ is stable under filtered colimits \cite[Prop. 7.2.4.27]{LurHA}, \cite[Cor. 5.5.7.4]{LurHTT}. Definition \ref{def:reasalg} extends from algebras to morphisms in $\CAlgk$ in the obvious way. It would be more precise to say ``almost strictly tamely presented'' instead of ``strictly tamely presented'', but for simplicity we use the shorter terminology. 

In the setting of ordinary rings, a (more elementary) variant of the proof above characterizes strictly tamely zero-presented algebras as those which are the union of the finitely zero-presented subalgebras over which they are flat. If $A$ is Noetherian, \cite[Prop. 7.2.4.31]{LurHA} implies that an ordinary $A$-algebra $B$ is strictly tamely presented if and only if it is strictly tamely zero-presented. If $A$ and $B$ are ordinary rings and $B$ is placid over $A$ in the sense of \cite[Def. 16.29.1]{Ras14}, or equivalently almost smooth over $A$ in the sense of \cite[Def. 3.2.4]{KV04}, then $B$ is strictly tamely zero-presented. 

Morphisms of strictly tame presentation have the following stability properties. 

\begin{Proposition}\label{prop:reasstabprops2}
	Let $\phi: A \to B, \psi: B \to C$, and $\eta: A \to A'$ be morphisms in $\CAlgk$.
	\begin{enumerate}
		\item If $\phi$ is of strictly tame presentation (resp. is strictly tamely $n$-presented) then so is $\phi': A' \to B \ot_A A'$ (resp. $\tau_{\leq n}\phi': \tau_{\leq n} A' \to \tau_{\leq n}(B \ot_A A')$). 
		\item If $\phi$ and $\psi$ are of strictly tame presentation (resp. are strictly tamely $n$-presented) then so is $\psi \circ \phi$. 
		\item If $\phi$ is almost of finite presentation (resp. is finitely $n$-presented) then $\psi \circ \phi$ is of strictly tame presentation (resp. is strictly tamely $n$-presented) if and only if $\psi$ is. 
	\end{enumerate}
\end{Proposition}
\begin{proof}
	Note that in each case it suffices to prove the claim about strictly tamely $n$-presented morphisms, as it implies the claim about strictly tamely presented morphisms. For~(1), let $B \cong \colim B_\al$ be a strictly tame presentation of order $n$. Then $\tau_{\leq n}(B_\al \ot_A A')$ is finitely $n$-presented over $A'$ for all $\al$ \cite[Prop. 4.1.3.2]{LurSAG}. Since $\tau_{\leq n}$ is continuous and compatible with the symmetric monoidal structure on $\Mod_A^{\leq 0}$ \cite[Prop. 7.1.3.15]{LurHA}, hence on $\CAlg_A$, we have 
	$$ \tau_{\leq n}(B \ot_A A') \cong \tau_{\leq n}((\tau_{\leq n}B) \ot_{\tau_{\leq n}A} \tau_{\leq n}A') \cong \colim \tau_{\leq n}(B_\al \ot_{\tau_{\leq n}A} \tau_{\leq n}A') \cong \colim \tau_{\leq n}(B_\al \ot_A A'). $$
	Since flatness is preserved by base change and $\tau_{\leq n}$, it follows that $\tau_{\leq n}(B \ot_A A')$ is strictly tamely $n$-presented over $A'$. 
	
	For (2), let $B \cong \colim B_\al$ and $C \cong \colim C_\be$ be strictly tame presentations of order $n$ over $A$ and $B$, respectively. Given a finitely $n$-presented $A$-algebra $A$ and a morphism $D \to C$ in $\CAlg_A$, we claim the criterion of Proposition \ref{prop:reasalg} is satisfied. Note first that $D \to C$ factors through some $C_\be$ since $D$ is compact in $\CAlgvarleqn{A}$. By Noetherian approximation \cite[Cor. 4.4.1.4]{LurSAG} we have $C_\be \cong \tau_{\leq n} (B \ot_{B_\al} C_{ \al \be})$ for some~$\al$ and some finitely $n$-presented $B_\al$-algebra $C_{\al \be}$. Since $B$ is flat over $B_\al$ we in fact have $C_\be \cong B \ot_{B_\al} C_{ \al \be}$. Letting $C_{\ga \be} := B_\ga \ot_{B_\al} C_{ \al \be}$ for $\ga \geq \al$, we moreover have $C_\be \cong \tau_{\leq n} C_\be \cong \colim_{\ga \geq \al} \tau_{\leq n} C_{\ga \be}$. 
	
	Again by compactness $D \to C_\be$ factors through some $\tau_{\leq n} C_{\ga \be}$. Now $B_\ga \to \tau_{\leq n} C_{\ga \be}$ is finitely $n$-presented since $B_\al \to C_{\al \be}$ is \cite[Prop. 4.1.3.2]{LurSAG}, hence the composition $A \to B_\ga \to \tau_{\leq n} C_{\ga \be}$ is finitely $n$-presented \cite[Prop. 4.1.3.1]{LurSAG}. But $\tau_{\leq n} C_{\ga \be} \to C_\be$ is flat since $B_\ga \to B$ is and since flatness is preserved by $\tau_{\leq n}$, hence the composition $\tau_{\leq n} C_{\ga \be} \to C_\be \to C$ is flat and we are done.
	
	For (3), suppose $C \cong \colim C_\al$ is a strictly tame presentation of order $n$ over $A$. By compactness of $B$ in $\CAlgvarleqn{A}$ we have that $B \to C$ factors through some $C_\al$. By \cite[Prop. 4.1.3.1]{LurSAG} the morphism $B \to C_\al$ is finitely $n$-presented, as is the composition $B \to C_\al \to C_\be$ for all $\be \geq \al$. But then $C \cong \colim_{\be \geq \al} C_\be$ is a strictly tame presentation of order $n$ over $B$. This proves the only if direction, and the if direction follows from (2). 
\end{proof}

Note that if $B \cong \colim B_\al$ is a filtered colimit of finitely $n$-presented algebras such that the structure maps $B_\al \to B_\be$ are flat, it follows that $B$ is flat over each $B_\al$. We do not know if every strictly tamely $n$-presented algebra admits a presentation with this stronger property, but it is easy to construct presentations which do not have it. 

\begin{Example}\label{ex:infiniteA1localization}
	Given $\{a_n\}_{n \in \bN} \subset \C$, let $B_n$ be the localization of $\C[x,y]/((x - a_n)y)$ by the elements $\{x - a_m\}_{m < n}$. In other words, $\Spec B_n$ is $\A^1$ with the points $a_1, \dotsc, a_{n-1}$ removed and with another $\A^1$ intersecting at the point $a_n$. If $B_n \to B_{n+1}$ takes $x$ to $x$ and $y$ to $0$, then $B := \colim B_n$ is the localization of $\C[x]$ by the elements $\{x - a_n\}_{n \in \bN}$. In particular it is a localization of each $B_n$, hence each $B_n \to B$ is flat even though no $B_m \to B_n$ is. 
\end{Example}

On the other hand, if $B \cong \colim B_\al$ is a strictly tame presentation of order $n$ and $B$ is \emph{faithfully} flat over each $B_\al$, it follows that the structure morphisms $B_\al \to B_\be$ are also faithfully flat \cite[Lem. B.1.4.2]{LurSAG}. But many natural examples, in particular those of the following class, do not admit such presentations. 

\begin{Example}\label{ex:essfinpresisstricttamepres}
	If $A$ is an ordinary ring, recall that an ordinary $A$-algebra $B$ is essentially finitely presented if it is a localization $B \cong S^{-1}C$ of a finitely zero-presented $A$-algebra $C$. In this case $B \cong \colim S_{fin}^{-1} C$, where the colimit is over all finite subsets $S_{fin} \subset S$. This is a strictly tame presentation of order zero, hence any essentially finitely presented algebra is strictly tamely zero-presented. 
\end{Example}

Often we can write an $A$-algebra of strictly tame presentation as a filtered colimit $B \cong \colim B_\al$ such that each $B_\al$ is almost finitely presented over $A$ and $B$ is flat over each $B_\al$. However, Example \ref{ex:weirdex} below shows that not every example is of this form. Moreover, the associated class of morphisms in $\CAlgk$ is not obviously stable under composition, since almost finitely presented algebras cannot be directly controlled by Noetherian approximation the way finitely $n$-presented algebras can be. The situation is improved by suitable truncatedness hypotheses, but these are in turn not stable under base change. Thus Proposition~\ref{prop:reasstabprops2} does not extend robustly to such morphisms. 

\begin{Example}\label{ex:weirdex}
	Let $A_0 = \C[x_1, x_2, \dotsc]$, $A_n = A_0/(x_1,\dotsc,x_n)$, and $A$ the trivial square-zero extension of $A_0$ by $\oplus_{n > 0} A_n[n]$. Then~$A$ is a strictly tamely presented $\C$-algebra, but cannot be written as a filtered colimit of almost finitely presented $\C$-algebras along flat morphisms: $A_0$ has no finitely presented subalgebra from which each $A_n$ is obtained by base change, which flatness would require. 
\end{Example}

Recall that an ordinary commutative ring $A$ is coherent if every finitely generated ideal is finitely presented. More generally, $A \in \CAlgk$ is coherent if $\clalg{A}$ is coherent in the above sense and $H^n(A)$ is a finitely presented $\clalg{A}$-module for all $n$. In general this is a brittle property, and coherence of an ordinary ring $A$ does not even imply coherence of $A[x]$. But tamely presentedness over a Noetherian base implies a more robust form of coherence. 

\begin{Proposition}\label{prop:stablecoh}
	Let $A$ be a strictly tamely presented $\kk$-algebra. Then $A$ is coherent, as is any strictly tamely presented $A$-algebra. 
\end{Proposition}

\begin{proof}
	The second claim follows from the first by the stability of strictly tamely presented morphisms under composition (Proposition \ref{prop:reasstabprops2}). 
	That $H^0(A)$ is coherent is essentially \cite[Sec. I.2 Ex. 12]{Bou72}, but we repeat the argument. Let $H^0(A) \cong \colim A_\al$ be a strictly tame presentation of order zero over $\kk$. If $I \subset H^0(A)$ is a finitely generated ideal, we can write it as the image of a morphism $\phi: H^0(A)^n \to H^0(A)$ for some $n$. This is obtained by extension of scalars from some $\phi_\al: A_\al^n \to A_\al$ for some $\al$. The kernel of $\phi_\al$ is finitely generated since $\kk$ and hence $A_\al$ are Noetherian, but $\ker \phi \cong (\ker \phi_\al) \ot_{A_\al} H^0(A)$ since $H^0(A)$ is flat over $A_\al$. 
	
	Now fix $n$, let $\tau_{\leq n} A \cong \colim A_\al$ be a strictly tame presentation of order $n$, and choose some $\al$. Each $A_\al$ is Noetherian \cite[Prop. 4.2.4.1]{LurSAG}, hence $H^n(A_\al)$ is a finitely presented $\clalg{ A_\al}$-module. Since $A_\al \to \tau_{\leq n}A$ is flat $H^n(A) \cong H^n(A_\al) \ot_{\clalg{ A_\al}} \clalg{ A}$, hence $H^n(A)$ is a finitely presented $\clalg{A}$-module. The claim follows since $n$ was arbitrary. 
\end{proof}

\begin{Remark}\label{rem:stablycoh}
Recall that an ordinary commutative ring is stably coherent if any finitely generated algebra over it is coherent. The above proof almost adapts to show that $B$ is coherent if it is strictly tamely presented over an $A \in \CAlgk$ such that $A$ is coherent and $\clalg{A}$ is stably coherent. No changes are needed if $B$ is an ordinary ring, while \cite[Prop. 5.2.2.1]{LurSAG} can be leveraged if $\clalg{A}$ is of characteristic zero, or more generally if $A \to B$ arises from a morphism of animated/simplicial commutative rings (as do the terms in the needed strictly tame presentations). Plausibly these restrictions are unnecessary, the only question being whether the hypotheses on $A$ imply the free algebra $A_m := A\{x_1, \dotsc, x_m\}$ is coherent (i.e. if each $H^n(A_m)$ is finitely presented over $H^0(A_m)$). On the other hand, we do not know an $A$ satisfying these hypotheses which is not strictly tamely presented over a Noetherian ring (possibly after passing to a flat cover). 
\end{Remark}

\subsection{Coherent pullback}
Recall that an $A$-module $M$ is almost perfect if $\tau^{\geq n} M$ is compact in $\Mod_A^{\geq n}$ for all $n$ \cite[Rem. 2.7.0.5]{LurSAG}. If $M$ is almost perfect it is right bounded, and if $A$ is coherent $M$ is almost perfect if and only if it is right bounded and $H^n(M)$ is a finitely presented $\clalg{A}$-module for all $n$. If $A$ is an ordinary ring, $M$ is almost perfect if and only if it is pseudo-coherent in the sense of \cite{Ill71} (see \cite[Rem. 2.8.4.6]{LurSAG}). 

We say $M \in \Mod_A$ is coherent if it is almost perfect and (left) bounded, and denote the full subcategory of coherent modules by $\Coh_A \subset \Mod_A$. We recall the following definition from the introduction (we will use the same terminology for algebra morphisms as for the associated morphisms of affine schemes).

\begin{Definition}\label{def:cohpull}
A morphism $A \to B$ in $\CAlgk$ has coherent pullback if $M \ot_A B$ is a coherent $B$-module for every coherent $A$-module $M$. 
\end{Definition}

Equivalently, $A \to B$ has coherent pullback if and only if $M \ot_A B$ is (left) bounded for every coherent $A$-module $M$, since $M \ot_A B$ is almost perfect over $B$ if $M$ is almost perfect over $A$ \cite[Prop. 2.7.3.1]{LurSAG}. We note that morphisms with coherent pullback are called eventually coconnective morphisms in \cite[Def. 3.5.2]{Gai13a}. The following example shows that in infinite type such morphisms are generally not stable under base change. 

\begin{Example}\label{ex:originintoAinfty}
Let $i: \{0\} \to \A^\infty = \Spec \C[x_1, x_2, \dotsc]$ denote the inclusion of the origin. Any coherent sheaf on $\A^\infty$ is the pullback of a coherent sheaf on some $\A^n$ along the (flat) projection $\A^\infty \to \A^n$. Thus $i$ has coherent pullback since its composition with each such projection does, which follows since each $\A^n$ is a smooth variety.

On the other hand, consider the self-intersection $X = \{0\} \times_{\A^\infty} \{0\}$. We have $X \cong \Spec A$, where $A$ is a symmetric algebra on countably many generators in degree $\shortminus 1$. Thus $H^n (A)$ is nonzero for all $n < 0$, hence $A$ is not coherent as a module over itself (it is perfect but not bounded). In particular, the map $X \to \{0\}$ does not have coherent pullback, even though it is a base change of $i$. 
\end{Example}

However, the following result says that pathologies only arise in the above example because the map we are base changing along is ``too far'' from being finitely presented, and that they do not appear if we only consider tamely presented base change. 

\begin{Theorem}\label{thm:reasonableaffinecohbasechange}
	Consider a diagram of the following form in $\CAlgk$.
	\begin{equation}\label{eq:reasonableaffinecohbasechange1}
		\begin{tikzpicture}[baseline=(current  bounding  box.center),thick,>=\arrtip]
			\newcommand*{\hb}{3}; \newcommand*{\va}{1.5}
			\node (aa) at (0,0) {$A$};
			\node (ab) at (\hb,0) {$A'$};
			\node (ba) at (0,-\va) {$B$};
			\node (bb) at (\hb,-\va) {$B \ot_A A'$};
			\draw[->] (aa) to node[above] {$\psi $} (ab); 
			\draw[->] (aa) to node[left] {$\phi $} (ba);
			\draw[->] (ab) to node[right] {$\phi' $} (bb);
			\draw[->] (ba) to node[above,pos=.7] {$\psi' $} (bb);
		\end{tikzpicture}
	\end{equation}
Suppose that $\phi$ has coherent pullback, $\psi$ is strictly tamely presented, and $A$ is strictly tamely presented over $\kk$. Then $\phi'$ has coherent pullback. 
\end{Theorem}

The proof will use the following reformulation of \cite[Thm. 7.1]{Swa19}. 

\begin{Lemma}\label{lem:A[x]coh}
Let $A$ be an ordinary commutative ring such that $A$ and $A[x]$ are coherent. Then $\Coh_{A[x]}$ is the smallest full stable subcategory of $\Mod_{A[x]}$ which contains the essential image of $\Coh_A$ under $- \ot_A A[x]$. 
\end{Lemma}
\begin{proof}
Let $\catC \subset \Mod_{A[x]}$ be the smallest full stable subcategory containing the essential image of $\Coh_A$, or equivalently of $\Coh_A^\heartsuit$. Since $A[x]$ is coherent it suffices to show $\Coh_{A[x]}^\heartsuit \subset \catC$. 

Given $M \in \Coh_{A[x]}^\heartsuit$, choose an exact sequence $0 \to N \to F \to F' \to M \to 0$ with $F$ and $F'$ free of finite rank. Certainly $F, F' \in \catC$, so it suffices to show $N \in \catC$. But this follows from the proof of \cite[Thm. 7.1]{Swa19}, which shows that $N$ fits into an exact sequence $0 \to N' \ot_A A[x] \to N'' \ot_A A[x] \to N \to 0$ with $N', N'' \in \Coh_A^\heartsuit$. 
\end{proof}

\begin{proof}[Proof of Theorem \ref{thm:reasonableaffinecohbasechange}] 
	Set $B' := B \ot_A A'$. Consider the following diagram in $\CAlgk$, where all but the top and bottom faces are coCartesian. 
	\begin{equation}\label{eq:reasonableaffinecohbasechange2}
		\begin{tikzpicture}[baseline=(current  bounding  box.center),thick,>=\arrtip]
			\newcommand*{\ha}{1.5}; \newcommand*{\hb}{1.75}; \newcommand*{\hc}{1.5};
\newcommand*{\va}{-.9}; \newcommand*{\vb}{-.9}; \newcommand*{\vc}{-.9}; 
			\node (ab) at (\ha,0) {$A$};
			\node (ad) at (\ha+\hb+\hc,0) {$A'$};
			\node (ba) at (0,\va) {$\clalg{A}$};
			\node (bc) at (\ha+\hb,\va) {$\clalg{A'}$};
			\node (cb) at (\ha,\va+\vb) {$B$};
			\node (cd) at (\ha+\hb+\hc,\va+\vb) {$B'$};
			\node (da) at (0,\va+\vb+\vc) {$\clalg{A} \ot_A B$};
			\node (dc) at (\ha+\hb,\va+\vb+\vc) {$\clalg{A'} \ot_{A'} B'$};
			\draw[->] (ab) to node[above] {$  $} (ad);
			\draw[->] (ab) to node[above left, pos=.25] {$  $} (ba);
			\draw[->] (ab) to node[right,pos=.2] {$\phi  $} (cb);
			\draw[->] (ad) to node[below right] {$  $} (bc);
			\draw[->] (ad) to node[right] {$\phi'  $} (cd);
			\draw[->] (ba) to node[left] {$\xi $} (da);
			\draw[->] (cb) to node[above,pos=.25] {$  $} (cd);
			\draw[->] (cb) to node[above left, pos=.25] {$ $} (da);
			\draw[->] (cd) to node[below right] {$ $} (dc);
			\draw[->] (da) to node[above,pos=.75] {$ $} (dc);
			
			\draw[-,line width=6pt,draw=white] (ba) to  (bc);
			\draw[->] (ba) to node[above,pos=.75] {$ $} (bc);
			\draw[-,line width=6pt,draw=white] (bc) to  (dc);
			\draw[->] (bc) to node[right,pos=.2] {$\xi' $} (dc);
		\end{tikzpicture}
	\end{equation}
Proposition \ref{prop:stablecoh} implies $A$ and $A'$ are coherent. In particular, restriction of scalars along $A \to \clalg{A}$ preserves coherence, hence $\xi$ has coherent pullback since $\phi$ does. Moreover, it suffices to show $M \otimes_{A'} B'$ is bounded for $M \in \Coh_{A'}^\heartsuit$. But any such $M$ is obtained by restriction of scalars along $A' \to \clalg{A'}$, hence it suffices to show~$\xi'$ has coherent pullback. 

Since $\kk$ is Noetherian, $\kk \to \clalg{A}$ is strictly tamely presented since $\kk \to A$ is \cite[Prop. 7.2.4.31]{LurHA}. Similarly $\clalg{\psi}$ is strictly tamely presented: if $\clalg{A'} \cong \colim A'_\al$ is a strictly tame presentation of order zero, each $A'_\al$ is almost finitely presented over $\clalg{A}$ by \cite[Cor. 5.2.2.3]{LurSAG} (note that any polynomial ring over $\clalg{A}$ is coherent by Proposition \ref{prop:stablecoh}). 
Replacing (\ref{eq:reasonableaffinecohbasechange1}) with the front face of (\ref{eq:reasonableaffinecohbasechange2}), we may thus assume $A$ and $A'$ are ordinary commutative rings.  
	
Next suppose that $\psi$ is finitely zero-presented. For some $n$ we can factor (\ref{eq:reasonableaffinecohbasechange1}) as
\begin{equation}\label{eq:reasonableaffinecohbasechange3}
	\begin{tikzpicture}[baseline=(current  bounding  box.center),thick,>=\arrtip]
		\newcommand*{\ha}{3.7}; \newcommand*{\hb}{3.7}; \newcommand*{\va}{1.6}
		\node (aa) at (0,0) {$A$};
		\node (ab) at (\ha,0) {$A[x_1,\dotsc,x_n]$};
		\node (ac) at (\ha+\hb,0) {$A'$};
		\node (ba) at (0,-\va) {$B$};
		\node (bb) at (\ha,-\va) {$B \ot_A A[x_1,\dotsc,x_n]$};
		\node (bc) at (\ha+\hb,-\va) {$B'$,};
		\draw[->] (aa) to node[above] {$\theta $} (ab); 
		\draw[->] (ab) to node[above] {$\xi $} (ac); 
		\draw[->] (ba) to node[above] {$\theta' $} (bb); 
		\draw[->] (bb) to node[above] {$\xi' $} (bc); 
		\draw[->] (aa) to node[left] {$\phi $} (ba);
		\draw[->] (ab) to node[right] {$\phi'' $} (bb);
		\draw[->] (ac) to node[right] {$\phi' $} (bc);
	\end{tikzpicture}
\end{equation}
where $\xi$ is surjective and finitely presented.  Since $A[x_1, \dotsc, x_n]$ and $A'$ are coherent, restriction of scalars along $\xi$ preserves coherence. But restriction of scalars along $\xi'$ is conservative and t-exact, hence $\phi'$ has coherent pullback if $\phi''$ does.  Thus we may replace (\ref{eq:reasonableaffinecohbasechange1}) with the left square of (\ref{eq:reasonableaffinecohbasechange3}) and assume $A' \cong A[x_1,\dotsc,x_n]$, and by induction we may then assume $n = 1$. 

Now let $\catC \subset \Mod_{A'}$ be the full subcategory of $M$ such that $M \ot_{A'} B'$ is bounded. Since $- \ot_{A'} B'$ is exact, $\catC$ is stable. Since $\phi$ has coherent pullback and $\psi'$ is flat, $\catC$ contains the essential image of $\Coh_A$ under $- \ot_A A'$.  Thus $\Coh_{A'} \subset \catC$ by Lemma \ref{lem:A[x]coh}, hence $\phi'$ has coherent pullback.

Now suppose $A' \cong \colim A'_\al$ is a strictly tame presentation of order zero over $A$. Since $A'$ is coherent, it suffices to show that $M \ot_{A'} B'$ is bounded for any $M \in \Coh_{A'}^\heartsuit$. For each $\al$ we write the induced factorization of (\ref{eq:reasonableaffinecohbasechange1}) as
\begin{equation*}
	\begin{tikzpicture}[baseline=(current  bounding  box.center),thick,>=\arrtip]
		\newcommand*{\ha}{3}; \newcommand*{\hb}{3}; \newcommand*{\va}{1.6}
		\node (aa) at (0,0) {$A$};
		\node (ab) at (\ha,0) {$A'_\al$};
		\node (ac) at (\ha+\hb,0) {$A'$};
		\node (ba) at (0,-\va) {$B$};
		\node (bb) at (\ha,-\va) {$B'_\al$};
		\node (bc) at (\ha+\hb,-\va) {$B'$.};
		\draw[->] (aa) to node[above] {$\theta_\al $} (ab); 
		\draw[->] (ab) to node[above] {$\xi_\al $} (ac); 
		\draw[->] (ba) to node[above] {$\theta'_\al $} (bb); 
		\draw[->] (bb) to node[above] {$\xi'_\al $} (bc); 
		\draw[->] (aa) to node[left] {$\phi $} (ba);
		\draw[->] (ab) to node[right] {$\phi_\al $} (bb);
		\draw[->] (ac) to node[right] {$\phi' $} (bc);
	\end{tikzpicture}
\end{equation*}
By flatness of the $\xi_\al$ and by e.g. \cite[Cor. 4.5.1.10]{LurSAG} or \cite[Sec. C.4]{TT90}, there exists an $\al$ and $M_\al \in \Coh_{A'_\al}^\heartsuit$ such that $M \cong M_\al \ot_{A'_\al} A'$. Since $\theta_\al$ is finitely zero-presented, we have already shown that $M_\al \ot_{A'_\al} B'_\al$ is bounded. But $\xi'_\al$ is flat since $\xi_\al$ is, hence $M \ot_{A'} B' \cong M_\al \ot_{A'_\al} A' \ot_{A'} B' \cong M_\al \ot_{A'_\al} B'_\al \ot_{B'_\al} B'$ is also bounded. 
\end{proof}

\begin{Remark}
The conclusion of Theorem \ref{thm:reasonableaffinecohbasechange} holds if instead of $A$ being strictly tamely presented over $\kk$ we assume that $A$ and $A'$ are coherent and that $\clalg{A}$ is stably coherent. This is because the proof only uses this hypothesis on $A$ in order to apply Proposition \ref{prop:stablecoh}. 
\end{Remark}

Recall that a morphism $\phi: A \to B$ in $\CAlgk$ is of Tor-dimension $\leq n$ if $B \ot_A M \in \Mod_B^{\geq -n}$ for all  $M \in \Mod_A^\heartsuit$, and is of finite Tor-dimension if it is of Tor-dimension $\leq n$ for some $n$. Clearly $\phi$ has coherent pullback if it is of finite Tor-dimension. We have the following partial converse, which generalizes \cite[Lem. 3.6.3]{Gai13a}. 

\begin{Proposition}\label{prop:finitetor2}
	Suppose $\phi: A \to B$ is a morphism in $\CAlgk$ with coherent pullback. If $A$ is Noetherian and $\phi$ of strictly tame presentation, then $\phi$ is of finite Tor-dimension. 
\end{Proposition}
\begin{proof}
	Note that $\phi$ is of finite Tor-dimension if and only if its base change $\phi': \clalg{A} \to \clalg{A} \otimes_A B$ is, since every discrete $A$-module is obtained by restriction of scalars from $\clalg{A}$. Similarly $\phi$ has coherent pullback if and only if $\phi'$ does, since $A$ is Noetherian (and in particular coherent). Since $\phi$ is of strictly tame presentation so is $\phi'$ (Proposition \ref{prop:reasstabprops2}), hence we may replace $\phi$ with $\phi'$ and assume $A$ is classical. 
	
	In particular, $A$ is now coherent over itself, hence $B$ is truncated since $\phi$ has coherent pullback. Since $\phi$ is of strictly tame presentation it admits a factorization 
\begin{equation*}
	\begin{tikzpicture}
		[baseline=(current  bounding  box.center),thick,>=\arrtip]
		\newcommand*{\ha}{2.5}; \newcommand*{\hb}{2.5};
		\node (aa) at (0,0) {$A$};
		\node (ab) at (\ha,0) {$A'$};
		\node (ac) at (\ha+\hb,0) {$B$};
		\draw[->] (aa) to node[above] {$\phi_1 $} (ab);
		\draw[->] (ab) to node[above] {$\phi_2 $} (ac);
	\end{tikzpicture}
\end{equation*}
such that $\phi_2$ is flat and $\phi_1$ is finitely $n$-presented for some $n$. Since $A$ is Noetherian, $\phi_1$ is then almost finitely presented and $A'$ is also Noetherian \cite[Prop. 4.2.4.1]{LurSAG}. 

Recall that $\phi_1$ is of Tor-dimension $\leq n$ at a prime ideal $P' \subset \clalg{A'}$ if the localization $A'_{P'}$ is of Tor-dimension $\leq n$ over $A$. For any $n$ the set of such prime ideals forms a Zariski open subset $U_n$ of $|\Spec A'|$, the underlying topological space of $\Spec A'$ \cite[Lem. 6.1.5.5]{LurSAG}. Since $A'$ is Noetherian we can increase $n$ as needed so that $U_n$ is equal to the union of the $U_m$ for all $m \in \bN$. We claim that $\phi$ is of Tor-dimension $\leq n$.

It suffices to show that for any prime ideal $P \subset \clalg{B}$, $B_P$ is of Tor-dimension $\leq n$ over~$A$ \cite[Prop. 6.1.4.4]{LurSAG}. 
Write $P' \subset \clalg{A'}$ and $Q \subset \clalg{A}$ for the preimages of $P$ under $\clalg{\phi_2}$ and $\clalg{\phi}$, and write $\kappa$ for the residue field of $A$ at $Q$. The morphism $\psi: A \to B_P$ can be factored as the middle row of the following diagram. 
\begin{equation*}
	\begin{tikzpicture}
		[baseline=(current  bounding  box.center),thick,>=\arrtip]
		\newcommand*{\ha}{3}; \newcommand*{\hb}{3};
		\newcommand*{\va}{-1.5};
		\node (zb) at (\ha,-\va) {$A'$};
		\node (zc) at (\ha+\hb,-\va) {$B$};
		\node (aa) at (0,0) {$A$};
		\node (ab) at (\ha,0) {$A'_{P'}$};
		\node (ac) at (\ha+\hb,0) {$B_P$};
		\node (ba) at (0,\va) {$\kappa$};
		\node (bb) at (\ha,\va) {$A'_{P'} \otimes_A \kappa$};
		\node (bc) at (\ha+\hb,\va) {$B_{P} \otimes_A \kappa$};
		\draw[->] (aa) to node[above left] {$\phi_1 $} (zb);
		\draw[->] (zb) to node[above] {$\phi_2 $} (zc);
		\draw[->] (aa) to node[above, pos=.7] {$\psi_1 $} (ab);
		\draw[->] (ab) to node[above] {$\psi_2 $} (ac);
		\draw[->] (ba) to node[above] {$\psi'_1 $} (bb);
		\draw[->] (bb) to node[above] {$\psi'_2 $} (bc);
		\draw[->] (aa) to node[left] {$ $} (ba);
		\draw[->] (ab) to node[right] {$ $} (bb);
		\draw[->] (ac) to node[right] {$ $} (bc);
		\draw[->] (zb) to node[right] {$ $} (ab);
		\draw[->] (zc) to node[right] {$ $} (ac);
	\end{tikzpicture}
\end{equation*}
Since $\phi_2$ is flat so is $\psi_2$ \cite[Rem. 6.1.4.3]{LurSAG}, hence $\psi$ is of Tor-dimension $\leq n$ if $\psi_1$ is, or equivalently if $P' \in U_n$. 

Suppose $P' \notin U_n$. Then $P' \notin U_m$ for any $m$, hence $A'_{P'} \otimes_A \kappa$ is not truncated \cite[Lem. 6.1.5.2]{LurSAG}. But $\psi_2$ is in fact faithfully flat since $\clalg{\psi_2}: \clalg{A'_{P'}} \to \clalg{B_P}$ is a local ring homomorphism, so $B_{P} \otimes_A \kappa$ is also not truncated. 
Now note that $\psi$ has coherent pullback since $\phi$ does, given that it is the composition of $\phi$ with the flat morphism $B \to B_P$. Since $A$ is Noetherian, $A \to \clalg{A}/Q$ is almost finitely presented, hence $A \to \kappa$ is strictly tamely presented  (as in Example \ref{ex:essfinpresisstricttamepres}). Thus $\psi' := \psi'_2 \circ \psi'_1$ also has coherent pullback (Theorem \ref{thm:reasonableaffinecohbasechange}), and we have a contradiction since then $B_{P} \otimes_A \kappa$ must be truncated. 
\end{proof}

\section{Geometric stacks}\label{sec:reasgeomstacks}

We now consider tamely presented morphisms and coherent pullback in the setting of geometric stacks. Again we begin with basic stability properties (Propositions \ref{prop:reasmorprops2}, \ref{prop:wafp2of3prop}, \ref{prop:cohpullprops}). Among tamely presented geometric stacks, we show morphisms with (\stable) coherent pullback can be approximated by morphisms of finite Tor-dimension (Proposition \ref{prop:finitetorapproxgstk}). Conversely, certain pro-smoothness conditions guarantee that the diagonal of a geometric stack has \stable coherent pullback (Proposition \ref{prop:wprosmoothcohdiagb}). 

\subsection{Definitions}
We first review our basic conventions on geometric stacks, which follow \cite[Ch. 9]{LurSAG} and by extension \cite{CWig}. We do caution that this terminology varies in the literature, in particular in \cite{TV08}. Recall that a stack will mean a functor $\CAlgk \to \Spc$ satisfying fpqc descent (here $\kk$ is our fixed Noetherian base), and that the category of stacks is denoted by $\Stkk$. 

\begin{Definition}\label{def:1}
	A stack $X$ is geometric if its diagonal $X \to X \times X$ is affine and there exists faithfully flat morphism $\Spec B \to X$ in $\Stkk$. A morphism $X \to Y$ in $\Stkk$ is geometric if for any morphism $\Spec A \to Y$, the fiber product $X \times_Y \Spec A$ is geometric. We write $\GStkk \subset \Stkk$ for the full subcategory of geometric stacks. 
\end{Definition}

Note here that products are taken in $\Stkk$, hence are implicitly over $\Spec \kk$. Also note that affineness of $X \to X \times X$ implies that any morphism $\Spec B \to X$ is affine. In particular, (faithful) flatness of such a morphism is defined by asking that its base change to any affine scheme is such. More generally, a morphism $X \to Y$ in $\GStkk$ is (faithfully) flat if its composition with any faithfully flat $\Spec A \to X$ is (faithfully) flat. A faithfully flat morphism of geometric stacks will also be called a flat cover. The following is an extension of \cite[Prop. 9.3.1.2]{LurSAG}, see \cite[\igpropgstkprops]{CWig}.

\begin{Proposition}\label{prop:gstkprops}
	Geometric morphisms are stable under composition and base change in~$\Stkk$. If $f: X \to Y$ is a morphism in $\Stkk$, then $f$ is geometric if $X$ and $Y$ are, and $X$ is geometric if $f$ and $Y$ are. In particular, $\GStkk$  is closed under fiber products in $\Stkk$. 
\end{Proposition}

The following classes of geometric stacks are of special interest. 

\begin{Definition}\label{def:coherentGstack}
	A geometric stack $X$ is locally coherent (resp. locally Noetherian) if there exists a flat cover $\Spec A \to X$ such that $A$ is coherent (resp. Noetherian). It coherent if it is locally coherent and $\QCoh(X)^\heartsuit$ is compactly generated. 
\end{Definition}

A locally Noetherian geometric stack is coherent by \cite[Prop. 9.5.2.3]{LurSAG}. 

\begin{Definition}\label{def:truncgstack}
	A geometric stack $X$ is $n$-truncated if it admits a flat cover $\Spec A \to X$ such that $A$ is $n$-truncated. We say $X$ is classical if it is zero-truncated, and is truncated if it is $n$-truncated for some $n$. We denote by $\GStkkplus \subset \GStkk$ the full subcategory of truncated geometric stacks.  
\end{Definition}

Alternatively, note that the restriction functor $(-)_{\leq n}: \PreStkk \to \PreStkkleqn$ takes $\Stkk$ to $\Stkkleqn$ \cite[Prop. A.3.3.1]{LurSAG}. Write $i_{\leq n}: \Stkkleqn \to \Stkk$ for the left adjoint of this restriction and $\tau_{\leq n}: \Stkk \to \Stkk$ for their composition. Then if $X$ is geometric, $\tau_{\leq n} X$ is an $n$-truncated geometric stack called the $n$-truncation of $X$, and $X$ is $n$-truncated if and only if the natural map $\tau_{\leq n} X \to X$ is an isomorphism \cite[Cor. 9.1.6.8, Prop. 9.1.6.9]{LurSAG}. 

In particular, if $X \in \Stkkleqzero$ is an ordinary algebraic variety, then $i_{\leq 0} X$ is a zero-truncated geometric stack. The functor $i_{\leq 0}: \Stkkleqzero \to \Stkk$ embeds the category of ordinary varieties (more generally, ordinary quasi-compact, semi-separated schemes, or quasi-compact Artin stacks with affine diagonal) as a full subcategory of $\GStkkplus$, and by default we will identify these categories with their images in $\GStkkplus$. 

We note that this use of the symbol $\tau_{\leq n}$ and of the term truncation are different from their usual meaning in terms of truncatedness of mapping spaces, but in practice no ambiguity will arise (and this abuse has the feature that $\tau_{\leq n} \Spec A \cong \Spec \tau_{\leq n} A$). 

\begin{Definition}
	If $X \in \GStkkplus$, then $\cF \in \QCoh(X)$ is coherent if~$f^*(\cF)$ is a coherent $A$-module for some (equivalently, any) flat cover $\Spec A \to X$. We write $\Coh(X) \subset \QCoh(X)$ for the full subcategory of coherent sheaves. 
\end{Definition}

If $X$ is locally coherent, the standard t-structure on $\QCoh(X)$ restricts to one on $\Coh(X)$. If $X$ is coherent, it follows from \cite[Prop. 9.1.5.1]{LurSAG} that specifically $\QCoh(X)^{\heartsuit}$ is compactly generated by $\Coh(X)^\heartsuit$. If $X$ is zero-truncated but not locally coherent, our use of the term coherent sheaf corresponds to the notion of bounded pseudocoherent complex in \cite{Ill71} (see \cite[Rem. 2.8.4.6]{LurSAG}). 

\subsection{Tamely presented morphisms}\label{subsec:geomtame}
We now consider the geometric counterparts to the notion of strictly tamely presented algebra. 

\begin{Definition}\label{def:reasaffmorph}
	A morphism $X \to Y$ in $\Stkk$ is \emph{strictly tamely presented} if it is affine and for any $\Spec A \to Y$, the coordinate ring of $X \times_Y \Spec A$ is strictly tamely presented as an $A$-algebra. A geometric morphism $X \to Y$ in $\Stkk$ is \emph{tamely presented} if for any $\Spec A \to Y$, there exists a strictly tamely presented flat cover $\Spec B \to \Spec A \times_Y X$ such that $B$ is strictly tamely presented over $A$.  A geometric stack $X$ is tamely presented if it is so over $\Spec \kk$. 
\end{Definition}

Recall following \cite[Def. 17.4.1.1]{LurSAG} that $f: X \to Y$ is (locally) almost of finite presentation if, for any $n \in \bN$ and any filtered colimit $A \cong \colim A_\al$ in $\CAlgkleqn$, the canonical map
\begin{equation*}
	\colim X(A_\al) \to X(A) \times_{Y(A)} \colim Y(A_\al)
\end{equation*} 
is an isomorphism (we omit the word locally by default, as we mostly consider quasi-compact morphisms). 
We follow \cite[Def. 6.3.2.1]{LurSAG} and say a morphism $f: X \to Y$ is representable if for any $\Spec A \to Y$, the fiber product $X \times_Y \Spec A$ is a (spectral) Deligne-Mumford stack. 

\begin{Proposition}\label{prop:afpimplieswafp}
If a geometric morphism $f: X \to Y$ in $\Stkk$ is representable and almost of finite presentation, then it is tamely presented. 
\end{Proposition}
\begin{proof}
Follows from \cite[Prop. 17.4.3.1]{LurSAG}, by which for any $\Spec A \to Y$ there is an \'etale cover $\Spec B \to X \times_Y \Spec A$ such that $B$ is almost of finite presentation over~$A$. 
\end{proof}

We recall the following stability properties of almost finitely presented morphisms, then consider their generalizations to the tamely presented setting. 

\begin{Proposition}[{\cite[Rem. 17.4.1.3, Rem. 17.4.1.5]{LurSAG}}]\label{prop:afp2of3prop} 
Almost finitely presented morphisms are stable under composition and base change in $\Stkk$. 
If $f$ and $g$ are composable morphisms in $\Stkk$ such that $g \circ f$ and $g$ are almost of finite presentation, then so is $f$. 
\end{Proposition}

\begin{Proposition}\label{prop:reasmorprops2}
Tamely presented geometric morphisms are stable under composition and base change in $\Stkk$. Suppose we have a Cartesian diagram 
\begin{equation*}
	\begin{tikzpicture}
		[baseline=(current  bounding  box.center),thick,>=\arrtip]
		\node (a) at (0,0) {$X'$};
		\node (b) at (3,0) {$Y'$};
		\node (c) at (0,-1.5) {$X$};
		\node (d) at (3,-1.5) {$Y$};
		\draw[->] (a) to node[above] {$f' $} (b);
		\draw[->] (b) to node[right] {$h $} (d);
		\draw[->] (a) to node[left] {$h' $}(c);
		\draw[->] (c) to node[above] {$f $} (d);
	\end{tikzpicture}
\end{equation*}
of geometric morphisms in $\Stkk$. If $f'$ and $h$ are tamely presented and $h$ is faithfully flat, then $f$ is tamely presented. 
\end{Proposition}
\begin{proof}
Stability under base change follows by construction. To see stability under composition let $f: X \to Y$ and $g: Y \to Z$ be tamely presented. Given $\Spec C \to Z$, there exists by hypothesis a diagram
\begin{equation*}
	\begin{tikzpicture}[baseline=(current  bounding  box.center),thick,>=\arrtip]
		\newcommand*{\ha}{3}; \newcommand*{\hb}{3}; \newcommand*{\hc}{3};
		\newcommand*{\va}{-1.5}; \newcommand*{\vb}{-1.5};
		\node (aa) at (0,0) {$\Spec A$};
		\node (ab) at (\ha,0) {$X''$};
		\node (ac) at (\ha+\hb,0) {$X'$};
		\node (ad) at (\ha+\hb+\hc,0) {$X$};
		\node (bb) at (\ha,\va) {$\Spec B$};
		\node (bc) at (\ha+\hb,\va) {$Y'$};
		\node (bd) at (\ha+\hb+\hc,\va) {$Y$};
		\node (cc) at (\ha+\hb,\va+\vb) {$\Spec C$};
		\node (cd) at (\ha+\hb+\hc,\va+\vb) {$Z$};
		
		\draw[->] (aa) to node[above] {$h'' $} (ab);
		\draw[->] (ab) to node[above] {$h' $} (ac);
		\draw[->] (ac) to node[above] {$ $} (ad);
		\draw[->] (bb) to node[above] {$h $} (bc);
		\draw[->] (bc) to node[above] {$ $} (bd);
		\draw[->] (cc) to node[above] {$ $} (cd);
		
		\draw[->] (ab) to node[right] {$ $} (bb);
		\draw[->] (ac) to node[right] {$ $} (bc);
		\draw[->] (ad) to node[right] {$f $} (bd);
		\draw[->] (bc) to node[right] {$ $} (cc);
		\draw[->] (bd) to node[right] {$g $} (cd);
	\end{tikzpicture}
\end{equation*}
in which all squares are Cartesian, $h$ and $h''$ are strictly tamely presented flat covers, and $A$, $B$ are strictly tamely presented over $B$, $C$. Proposition~\ref{prop:reasstabprops2} then implies $h' \circ h''$ is a strictly tamely presented flat cover and $A$ is strictly tamely presented over~$C$. 

To prove the last claim let $\Spec B \to Y$ be any morphism. By our hypotheses on~$h$ there exists a strictly tamely presented flat cover $\Spec B' \to Y' \times_Y \Spec B$ such that $B'$ is strictly tamely presented over $B$. We then have a commutative cube
\begin{equation*}
	\begin{tikzpicture}[baseline=(current  bounding  box.center),thick,>=\arrtip]
			\newcommand*{\ha}{1.5}; \newcommand*{\hb}{1.5}; \newcommand*{\hc}{1.5};
\newcommand*{\va}{-.9}; \newcommand*{\vb}{-.9}; \newcommand*{\vc}{-.9}; 
		\node (ab) at (\ha,0) {$Z'$};
		\node (ad) at (\ha+\hb+\hc,0) {$\Spec B'$};
		\node (ba) at (0,\va) {$X'$};
		\node (bc) at (\ha+\hb,\va) {$Y'$};
		\node (cb) at (\ha,\va+\vb) {$Z$};
		\node (cd) at (\ha+\hb+\hc,\va+\vb) {$\Spec B$};
		\node (da) at (0,\va+\vb+\vc) {$X$};
		\node (dc) at (\ha+\hb,\va+\vb+\vc) {$Y$};
		\draw[->] (ab) to node[above] {$ $} (ad);
		\draw[->] (ab) to node[above] {$ $} (ba);
		\draw[->] (ab) to node[above] {$ $} (cb);
		\draw[->] (ad) to node[above] {$ $} (bc);
		\draw[->] (ad) to node[above] {$ $} (cd);
		\draw[->] (ba) to node[above] {$ $} (da);
		\draw[->] (cb) to node[above,pos=.75] {$ $} (cd);
		\draw[->] (cb) to node[above] {$ $} (da);
		\draw[->] (cd) to node[above] {$ $} (dc);
		\draw[->] (da) to node[above,pos=.6] {$f $} (dc);
		
		\draw[-,line width=6pt,draw=white] (ba) to  (bc);
		\draw[->] (ba) to node[above,pos=.75] {$f' $} (bc);
		\draw[-,line width=6pt,draw=white] (bc) to  (dc);
		\draw[->] (bc) to node[right,pos=.2] {$h $} (dc);
	\end{tikzpicture}
\end{equation*}
in which all but the left and right faces are Cartesian. Since $f'$ is tamely presented there exists a strictly tamely presented flat cover $\Spec A \to Z'$ such that~$A$ is strictly tamely presented over $B'$. By Proposition \ref{prop:reasstabprops2} (and the stability of faithful flatness under composition and base change) it now follows that $\Spec A \to Z$ is a strictly tamely presented flat cover and that $A$ is strictly tamely presented over $B$. 
\end{proof}

\begin{Proposition}\label{prop:wafp2of3prop}
	If $f$ and $g$ are composable geometric morphisms in $\Stkk$ such that $g \circ f$ is tamely presented and $g$ is almost of finite presentation, then $f$ is tamely presented. 
\end{Proposition}
\begin{proof}
Let $f: X \to Y$ and $g: Y \to Z$ be the given morphisms. For any $\phi: \Spec A \to Y$, there exists by hypothesis a strictly tamely presented flat cover $h: \Spec B \to X':= X \times_Z \Spec A$ such that $B$ is strictly tamely presented over~$A$. Consider then the following diagram of Cartesian squares,
\begin{equation*}
	\begin{tikzpicture}[baseline=(current  bounding  box.center),thick,>=\arrtip]
		\newcommand*{\ha}{3}; \newcommand*{\hb}{3}; \newcommand*{\hc}{3};
		\newcommand*{\va}{-1.5}; \newcommand*{\vb}{-1.5};
		\node (aa) at (0,0) {$\Spec C$};
		\node (ab) at (\ha,0) {$X''$};
		\node (ac) at (\ha+\hb,0) {$\Spec A$};
		
		\node (ba) at (0,\va) {$\Spec B$};
		\node (bb) at (\ha,\va) {$X'$};
		\node (bc) at (\ha+\hb,\va) {$Y'$};
		\node (bd) at (\ha+\hb+\hc,\va) {$\Spec A$};
		
		\node (cb) at (\ha,\va+\vb) {$X$};
		\node (cc) at (\ha+\hb,\va+\vb) {$Y$};
		\node (cd) at (\ha+\hb+\hc,\va+\vb) {$Z$};
		
		\draw[->] (aa) to node[above] {$h' $} (ab);
		\draw[->] (ab) to node[above] {$f'' $} (ac);
		\draw[->] (ba) to node[above] {$h $} (bb);
		\draw[->] (bb) to node[above] {$f' $} (bc);
		\draw[->] (bc) to node[above] {$g' $} (bd);
		\draw[->] (cb) to node[above] {$f $} (cc);
		\draw[->] (cc) to node[above] {$g $} (cd);
		
		\draw[->] (aa) to node[right] {$\psi' $} (ba);
		\draw[->] (ab) to node[right] {$ $} (bb);
		\draw[->] (ac) to node[right] {$\psi $} (bc);
		\draw[->] (bb) to node[right] {$ $} (cb);
		\draw[->] (bc) to node[right] {$ $} (cc);
		\draw[->] (bd) to node[right] {$g \circ \phi $} (cd);
	\end{tikzpicture}
\end{equation*}
where $\psi$ is the canonical section of $g'$. That $\psi$ and hence $\psi'$ are affine follows from $g$ being geometric. That $h'$ is a strictly tamely presented flat cover follows from $h$ being so. Since $g$ and hence $g'$ are almost of finite presentation and $g' \circ \psi$ is the identity, it follows that $\psi$ and hence $\psi'$ are almost of finite presentation (Proposition \ref{prop:afp2of3prop}). That~$C$ is strictly tamely presented over $A$ now follows from Proposition \ref{prop:reasstabprops2}. 
\end{proof}

The definition of tamely presented morphism has two obvious variants, where respectively the condition that $\Spec B \to \Spec A \times_Y X$ or $\Spec B \to \Spec A$ is strictly tamely presented is dropped. Some results we state extend to one or the other of these variants. But there are key results which do not, so for the sake of uniformity we formulate all statements in terms of tamely presented morphisms. 

Moreover, when working with global quotient stacks over a field, the former condition turns out to be automatic. Here we say a scheme is tamely presented if it is geometric (i.e. quasi-compact and semi-separated) and tamely presented as such. 

\begin{Proposition}\label{prop:tamequotients}
Suppose $\kk$ is a field and $G$ is a classical affine group scheme over $\kk$. Then the map $\Spec \kk \to BG$ is strictly tamely presented. In particular, if $X$ is a tamely presented $G$-scheme over $\kk$, then $X/G$ is a tamely presented geometric stack. 
\end{Proposition}
\begin{proof}
Since $\kk$ is a field and $G$ is affine, we can write $G$ as a filtered limit $G \cong \lim G_\al$ of algebraic groups along faithfully flat homomorphisms \cite[Thm. 8.2, Thm. 11.1]{Mil12}. Fix a map $U \cong \Spec A \to BG$ and set $P := \Spec \kk \times_{B G} U$. By hypothesis there exists a faithfully flat $U' \cong \Spec A' \to U$ such that $P' := P \times_{U} U'$ is isomorphic to $U' \times G$. 

Now set $P_\al := P \times_G G_\al$ and $P'_\al := P' \times_G G_\al \cong U' \times G_\al$. Each map $P_\al \to U$ is affine, flat, and almost finitely presented since its pullback $P \times G_\al \to P$ along the flat cover $P \to U$ is \cite[Prop. 4.1.4.3, Lem. 9.3.1.1, Lem. B.1.4.2]{LurSAG}. We have $\lim P'_\al \cong (\lim P_\al) \times_U U'$ since $P'_\al \cong P_\al \times_U U'$ for all $\al$. It then follows that the natural map $P \to \lim P_\al$ is an isomorphism since $P' \to \lim P'_\al$ is manifestly so, and since the latter is the base change of the former along a flat cover. But by definition $\Spec \kk \to BG$ is then strictly tamely presented. 

It follows that $X \to X/G$ is strictly tamely presented since it is a base change of $\Spec \kk \to B G$ \cite[Thm. 6.1.3.9(4)]{LurHTT}. Since $X$ is tamely presented we can choose a strictly tamely presented flat cover $\Spec B \to X$ such that $B$ is a strictly tamely presented $\kk$-algebra. The composition $\Spec B \to X \to X/G$ is then strictly tamely presented, hence $X/G$ is tamely presented. 
\end{proof}

By Proposition \ref{prop:stablecoh} a tamely presented geometric stack $X$ is locally coherent, though to our knowledge it is not known that $X$ must be coherent even when it is a classical Artin stack \cite{Ryd16}. On the other hand, by \cite[\iglemaffcgen]{CWig} a tamely presented geometric stack is coherent if it is admissible in the following sense (which is a variant of \cite[Def. 1.4.5(a)]{VV10}). 

\begin{Definition}\label{def:admissible}
	A geometric stack is \emph{admissible} if it admits an affine morphism to a locally Noetherian geometric stack. 
\end{Definition}

Since admissible geometric stacks are manifestly closed under products, it follows from Proposition \ref{prop:reasmorprops2} that admissible, tamely presented geometric stacks provide a class of coherent geometric stacks which are closed under products. 

\subsection{Coherent pullback}\label{sec:geomcoh} We now consider coherent pullback in the geometric setting. 

\begin{Definition}\label{def:geomcohpull}
	Let $f: X \to Y$ be a morphism in $\GStkk$ such that $Y$ is truncated. We say $f$ has \emph{coherent pullback} if $X$ is truncated and $f^*: \QCoh(Y) \to \QCoh(X)$ takes $\Coh(Y)$ to $\Coh(X)$. We say $f$ has \emph{\stable coherent pullback} if for any truncated $Y'$ and any tamely presented morphism $Y' \to Y$, the base change $f': X \times_Y Y' \to Y'$ has coherent pullback. 
\end{Definition}

\begin{Proposition}\label{prop:cohpullprops}
	Morphisms with \stable coherent pullback are stable under composition and under base change along tamely presented morphisms in $\GStk^+$. 	Let $Y$ and $Y'$ be truncated geometric stacks and
\begin{equation*}
	\begin{tikzpicture}
		[baseline=(current  bounding  box.center),thick,>=\arrtip]
		\node (a) at (0,0) {$X'$};
		\node (b) at (3,0) {$Y'$};
		\node (c) at (0,-1.5) {$X$};
		\node (d) at (3,-1.5) {$Y$};
		\draw[->] (a) to node[above] {$f' $} (b);
		\draw[->] (b) to node[right] {$h $} (d);
		\draw[->] (a) to node[left] {$h' $}(c);
		\draw[->] (c) to node[above] {$f $} (d);
	\end{tikzpicture}
\end{equation*}
a Cartesian diagram in $\GStkk$.  If $f'$ has \stable coherent pullback and $h$ is faithfully flat, then $f$ has \stable coherent pullback. 
\end{Proposition}
\begin{proof} 
	Stability under composition and tamely presented base change follow from Proposition~\ref{prop:reasmorprops2}. Suppose that $Z \to Y$ is tamely presented and $Z$ is truncated. Consider the commutative cube
	\begin{equation*}
		\begin{tikzpicture}[baseline=(current  bounding  box.center),thick,>=\arrtip]
			\newcommand*{\ha}{1.5}; \newcommand*{\hb}{1.5}; \newcommand*{\hc}{1.5};
\newcommand*{\va}{-.9}; \newcommand*{\vb}{-.9}; \newcommand*{\vc}{-.9}; 
			\node (ab) at (\ha,0) {$W'$};
			\node (ad) at (\ha+\hb+\hc,0) {$Z'$};
			\node (ba) at (0,\va) {$X'$};
			\node (bc) at (\ha+\hb,\va) {$Y'$};
			\node (cb) at (\ha,\va+\vb) {$W$};
			\node (cd) at (\ha+\hb+\hc,\va+\vb) {$Z$};
			\node (da) at (0,\va+\vb+\vc) {$X$};
			\node (dc) at (\ha+\hb,\va+\vb+\vc) {$Y$};
			\draw[->] (ab) to node[above] {$\phi' $} (ad);
			\draw[->] (ab) to node[above] {$ $} (ba);
			\draw[->] (ab) to node[above] {$ $} (cb);
			\draw[->] (ad) to node[above] {$ $} (bc);
			\draw[->] (ad) to node[above] {$ $} (cd);
			\draw[->] (ba) to node[above] {$ $} (da);
			\draw[->] (cb) to node[above,pos=.2] {$\phi $} (cd);
			\draw[->] (cb) to node[above] {$ $} (da);
			\draw[->] (cd) to node[above] {$ $} (dc);
			\draw[->] (da) to node[above,pos=.6] {$f $} (dc);
			
			\draw[-,line width=6pt,draw=white] (ba) to  (bc);
			\draw[->] (ba) to node[above,pos=.75] {$f' $} (bc);
			\draw[-,line width=6pt,draw=white] (bc) to  (dc);
			\draw[->] (bc) to node[right,pos=.2] {$h $} (dc);
		\end{tikzpicture}
	\end{equation*}
	with all faces Cartesian. By Proposition~\ref{prop:reasmorprops2} the morphism $Z' \to Y'$ is tamely presented, hence $\phi'$ has coherent pullback. But since $h$ is faithfully flat so are the other vertical morphisms, and it follows that~$\phi$ has coherent pullback. 
\end{proof}

It is convenient for the definition of \stable coherent pullback to require no hypotheses beyond truncatedness, so that for example it includes a flat morphism of arbitrary truncated geometric stacks. However, our main interest will be in the case where the stacks involved are tamely presented. In particular, in this setting it can be reformulated as the condition of ``tamely local coherent pullback''. That is, suppose $f: X \to Y$ is a morphism of tamely presented truncated geometric stacks, and that $\Spec A \to Y$ is a strictly tamely presented flat cover such with $A$ strictly tamely presented over $\kk$. Then it follows from Theorem \ref{thm:reasonableaffinecohbasechange} and Proposition \ref{prop:cohpullprops} that $f$ has \stable coherent pullback if and only if its base change to $\Spec A$ has coherent pullback. 

Another important feature of the tamely presented setting is that morphisms with \stable coherent pullback can be locally approximated by morphisms of finite Tor-dimension. A morphism $f: X \to Y$ in $\GStkk$ is of Tor-dimension~$\leq n$ if $f^*(\QCoh(Y)^{\geq 0}) \subset \QCoh(Y)^{\geq n}$, and is of finite Tor-dimension if it is of Tor-dimension~$\leq n$ for some $n$. Morphisms of finite Tor-dimension are flat local on the target, and are stable under composition and base change \cite[\igpropftdprops]{CWig}. We have the following result, where we note that an expression $A \cong \colim A_\al$ of the indicated kind always exists for some $n$. 

\begin{Proposition}\label{prop:finitetorapproxgstk}
	Let $f: X \to Y$ be a morphism in $\GStk^+$ with \stable coherent pullback, and suppose that $X$ and $Y$ are tamely presented. Choose a strictly tamely presented flat cover $\Spec A \to Y$ such that $A$ is strictly tamely presented over~$\kk$, and let $f': X' \to \Spec A$ denote the base change of $f$. If $A \cong \colim A_\al$ exhibits $A$ as a strictly tamely $n$-presented $\kk$-algebra for some $n$, then the composition  
	\begin{equation*}
		\begin{tikzpicture}
			[baseline=(current  bounding  box.center),thick,>=\arrtip]
			\newcommand*{\ha}{2.5}; \newcommand*{\hb}{2.5};
			\node (aa) at (0,0) {$X'$};
			\node (ab) at (\ha,0) {$\Spec A$};
			\node (ac) at (\ha+\hb,0) {$\Spec A_\al$};
			\draw[->] (aa) to node[above] {$f' $} (ab);
			\draw[->] (ab) to node[above] {$u_\al $} (ac);
		\end{tikzpicture}
	\end{equation*}
	is of finite Tor-dimension for all $\al$. 
\end{Proposition}
\begin{proof}
	Fix a strictly tamely presented flat cover $g: \Spec B \to X$ such that $B$ is strictly tamely presented over $\kk$. If $g': \Spec B' \to X'$ is its base change (whose source is affine since $Y$ is geometric), then $B'$ is strictly tamely presented over $B$ and thus $\kk$ by Proposition~\ref{prop:reasstabprops2}. Since $\kk$ is Noetherian, $A_\al$ is almost finitely presented over $\kk$ for any $\al$ \cite[Prop. 4.1.2.1]{LurSAG}. Proposition~\ref{prop:reasstabprops2} then further implies that $B'$ is strictly tamely presented over~$A_\al$. 
	
	It follows from the hypotheses that $f'$ has coherent pullback, hence by flatness of $g'$ and $u_\al$ the composition $u_\al \circ f' \circ g'$ also has coherent pullback. This composition then has finite Tor-dimension by Proposition \ref{prop:finitetor2}, hence so does $u_\al \circ f'$ by faithful flatness of $g'$. 
\end{proof}

\subsection{Weak smoothness}

We now consider the following generalized notion of smoothness. 

\begin{Definition}\label{def:weaklysmooth}
A morphism $f: X \to Y$ of truncated geometric stacks is \emph{weakly smooth} if it is flat, tamely presented, and the diagonal $\Delta_f: X \to X \times_Y X$ has \stable coherent pullback. A tamely presented truncated geometric stack is weakly smooth if it is so over $\Spec \kk$. 
\end{Definition}

The requirement that $f$ be tamely presented here is mostly for convenience. The following results capture how morphisms with \stable coherent pullback and infinite Tor-dimension arise  in our applications. In particular, Proposition \ref{prop:wprosmoothcohdiagb} and \cite[Prop. 1.2.1, Prop. 1.7.1]{KV04} imply the weak smoothness of the jet scheme of an affine variety which is \'etale over $\A^n$. 

\begin{Proposition}\label{prop:wprosmoothcohdiaga}
	Suppose that $X \in \GStkkplus$ is weakly smooth and that $Y \in \GStkkplus$ has \stable coherent pullback over $\kk$ (for example, $\kk$ is a field and $Y \in \GStkkplus$ is arbitrary). Then any tamely presented morphism $f: Y \to X$ has \stable coherent pullback. 
\end{Proposition}
\begin{proof}
	We can factor $f$ as $Y \to X \times Y \to X$. 
	The first factor is the base change of $\Delta_X$ along the tamely presented morphism $(\id_X \times f)$, and the second is the base change of $Y \to \Spec \kk$ along the tamely presented morphism $X \to \Spec \kk$. Thus each factor and their composition have \stable coherent pullback by hypothesis and Proposition \ref{prop:cohpullprops}. 
\end{proof}

Recall that a morphism $A \to B$ is $\CAlgk$ is fiber smooth if it flat and if $\clalg{A} \to \clalg{B}$ is smooth in the ordinary sense. In particular, if $A$ and $B$ are ordinary rings then fiber smoothness and ordinary smoothness are the same. 

\begin{Proposition}\label{prop:wprosmoothcohdiagb}
	Suppose $X \in \GStkkplus$ is tamely presented and admits a flat cover $\Spec A \to X$ such that for some $n$ there is a strictly tame presentation $A \cong \colim A_\al$ of order $n$ in which each $A_\al$ is fiber smooth over $\kk$. Then $X$ is weakly smooth. 
\end{Proposition}
\begin{proof}
	Set $U = \Spec A$ and $U_\al = \Spec A_\al$. Since $U \times U \to X \times X$ is a flat cover, it suffices to show the base change $\Delta'_X: U \times_X U \to U \times U$ of $\Delta_X$ has \stable coherent pullback (Proposition~\ref{prop:cohpullprops}). Proposition \ref{prop:reasstabprops2} implies $U \times U$ is strictly tamely presented over $\Spec \kk$ since $U$ is, hence it suffices to show $\Delta'_X$ has coherent pullback (Theorem \ref{thm:reasonableaffinecohbasechange}). Since $U \times U$ is coherent (Proposition \ref{prop:stablecoh}) it suffices to show $\Delta'^*_X(\cF)$ is coherent for any $\cF \in \Coh(U \times U)^\heartsuit$. 
	
	By flatness of the $p_\al \times p_\al: U \times U \to U_\al \times U_\al$ and by e.g. \cite[Cor. 4.5.1.10]{LurSAG} or \cite[Sec. C.4]{TT90}, $\cF \cong (p_\al \times p_\al)^*(\cF_\al)$ for some $\al$ and some $\cF_\al \in \Coh(U_\al \times U_\al)^\heartsuit$. But since $U_\al \times U_\al$ is fiber smooth over $\kk$ its diagonal is of finite Tor-dimension \cite[Lem. 11.3.5.2]{LurSAG} (c.f. \cite[Lem. 0FDP]{Sta}). Thus $(p_\al \times p_\al) \circ \Delta'_X$ is of finite Tor-dimension by the reasoning of the proof of Proposition \ref{prop:wprosmoothcohdiaga}, hence $\Delta'^*_X(\cF)$ is coherent. 
\end{proof}

\subsection{Pushforward and base change}\label{sec:geombc}

Given a morphism $f: X \to Y$ in $\GStkk$, the pushforward $f_*: \QCoh(X) \to \QCoh(Y)$ is defined as the right adjoint of $f^*$. 
For $f_*$ to be well-behaved one needs additional hypotheses on $f$. Recall that a morphism $f: X \to Y$ in $\GStkk$ is of cohomological dimension $\leq n$ if $f_*(\QCoh(X)^{\leq 0}) \subset \QCoh(Y)^{\leq n}$, and is of finite cohomological dimension if it is of cohomological dimension $\leq n$ for some $n$. Morphisms of finite cohomological dimension  in $\GStkk$ are flat local on the target are stable under composition and base change  \cite[\igpropfcdprops]{CWig}. Moreover, pushforward along a morphism of finite cohomological dimension is continuous and satisfies base change with respect to $*$-pullback. 

A key case is that of proper morphisms. A morphism $f: X \to Y$ of geometric stacks is proper if for any $\Spec A \to Y$, the fiber product $X \times_Y \Spec A$ is proper over $\Spec A$ in the sense of \cite[Def. 5.1.2.1]{LurSAG}. In particular, this requires that $X \times_Y \Spec A$ be a quasi-compact separated (spectral) algebraic space. Proper morphisms are of finite cohomological dimension and are stable under composition and base change \cite[\igproppropprops]{CWig}. Moreover, pushforward along a proper, almost finitely presented morphisms preserves coherence.

Following \cite{GR17}, we summarize the compatiblity between  pushforward and pullback via correspondence categories. Let $\Corr(\GStkkplus)_{prop,coh}$ denote the 1-full subcategory of $\Corr(\GStkk)$ which only includes correspondences $X \xleftarrow{h} Y \xrightarrow{f} Z$ such that $h$ has \stable coherent pullback, $f$ is proper and almost finitely presented, and $X$ and $Z$ are truncated (hence so is~$Y$). Then, as in \cite[\igsecpropafp]{CWig}, there exists by Propositions~\ref{prop:afpimplieswafp} and \ref{prop:cohpullprops} a functor 
\begin{equation}\label{eq:CorrCohGStkpropcoh}
	\Coh: \Corr(\GStkkplus)_{prop,coh} \to \Catinfty
\end{equation}
whose value on the above correspondence is $f_* h^*: \Coh(X) \to \Coh(Z)$. 

\section{Ind-geometric stacks}\label{sec:indgeom}

We now consider tamely presented morphisms and coherent pullback in the setting of ind-geometric stacks. After recalling the needed definitions from \cite{CWig}, we consider the stability properties of our main classes of morphisms, followed by their interaction with coherent sheaves. 

\subsection{Definitions}

Recall from \cite[\igpropgstkconvergent]{CWig} that $\GStkk$ is contained in the full subcategory $\Stkkconv \subset \Stkk$ of convergent stacks. This subcategory plays a central role in our discussion because colimits in $\Stkkconv$ are typically more natural than colimits in $\Stkk$. For example, any $\Spec A$ is the colimit of its truncations $\Spec \tau_{\leq n} A$ in $\Stkkconv$, but not in $\Stkk$ unless $A$ is itself truncated. 

\begin{Definition}\label{def:indgeomstack}
	An \emph{ind-geometric stack} is a convergent stack $X$ which admits an expression $X \cong \colim_{\al} X_\al$ as a filtered colimit in $\Stkkconv$ of truncated geometric stacks along closed immersions. We call such an expression an \emph{ind-geometric presentation} of $X$. 
\end{Definition}

We write $\indGStkk \subset \Stkkconv$ for the full subcategory of ind-geometric stacks. Ind-geometric stacks are closed under finite limits in $\Stkkconv$ \cite[\igpropindgeomfiberprods]{CWig}. A key property of an ind-geometric presentation $X \cong \colim X_\al$ is that for any truncated geometric stack $Y$, the natural map
$$\colim \Map_{\Stkk}(Y, X_\al) \to \Map_{\Stkk}(Y, X)$$
is an isomorphism \cite[\igpropfactorthroughgeometric]{CWig}. 

\begin{Definition}\label{def:reasindgeomstack}
	A \emph{reasonable presentation} is an ind-geometric presentation $X \cong \colim_{\al} X_\al$ in which the structure maps are almost finitely presented. An ind-geometric stack is \emph{reasonable} if it admits a reasonable presentation, and \emph{coherent} if it admits a reasonable presentation whose terms are coherent geometric stacks. 
\end{Definition}

We write $\indGStkkreas \subset \indGStkk$ (resp. $\indGStkkcoh \subset \indGStkk$) for the full subcategory of reasonable (resp. coherent) ind-geometric stacks.    

\begin{Definition}
	An ind-geometric stack $X$ is $n$-truncated if it admits an ind-geometric presentation $X \cong \colim X_\al$ in which each $X_\al$ is an $n$-truncated geometric stack. We say $X$ is classical if it is zero-truncated. 
\end{Definition}

A typical example of a classical ind-geometric stack is the quotient of a classical ind-scheme by an affine group scheme. 

\begin{Definition}\label{def:reassubstack}
	Let $X$ be an ind-geometric stack. 
	A \emph{truncated} (resp. \emph{reasonable}) \emph{geometric substack} of $X$ is a truncated geometric stack $X'$ equipped with a closed immersion $X' \to X$ (resp. an almost finitely presented closed immersion $X' \to X$).  
\end{Definition}

The basic example of a truncated (resp. reasonable) geometric substack is any term in an ind-geometric (resp. reasonable) presentation \cite[\igpropigprestermsarereas]{CWig}. 

\subsection{Ind-tamely presented morphisms}\label{sec:indtamemorphisms}
The notion of tamely presented morphism is extended to the ind-geometric setting via the following template, which follows parallel notions such as ind-properness. 

\begin{Proposition}\label{prop:indPequivconditions}
	Let $f: X \to Y$ be a morphism of reasonable ind-geometric stacks, and let $X \cong \colim_\al X_\al$ be a reasonable presentation. The following conditions are equivalent.
	\begin{enumerate}
		\item 
		\begin{minipage}[t]{\linewidth-\labelwidth-\labelsep}
			For every diagram
			\begin{equation}\label{eq:indPdef}
				\begin{tikzpicture}
					[baseline=(current  bounding  box.center),thick,>=\arrtip]
					\node (a) at (0,0) {$X'$};
					\node (b) at (3,0) {$X$};
					\node (c) at (0,-1.5) {$Y'$};
					\node (d) at (3,-1.5) {$Y$};
					\node (e) at (5,-1.5) {$ $};
					\draw[->] (a) to node[above] {$ $} (b);
					\draw[->] (b) to node[right] {$f $} (d);
					\draw[->] (a) to node[left] {$f' $}(c);
					\draw[->] (c) to node[above] {$ $} (d);
				\end{tikzpicture}
			\end{equation}
			in which $X' \to X$ and $Y' \to Y$ are reasonable geometric substacks, the map $f'$ is tamely presented. 
		\end{minipage}
		\vspace{2mm}
		\item 
		\begin{minipage}[t]{\linewidth-\labelwidth-\labelsep}
			For every $X_\al$ there exists a diagram
			\begin{equation}\label{eq:cancheckpresforindPdiag}
				\begin{tikzpicture}
					[baseline=(current  bounding  box.center),thick,>=\arrtip]
					\node (a) at (0,0) {$X_\al$};
					\node (b) at (3,0) {$X$};
					\node (c) at (0,-1.5) {$Y_\al$};
					\node (d) at (3,-1.5) {$Y$};
					\node (e) at (5,-1.5) {$ $};
					\draw[->] (a) to node[above] {$ $} (b);
					\draw[->] (b) to node[right] {$f $} (d);
					\draw[->] (a) to node[left] {$f_\al $}(c);
					\draw[->] (c) to node[above] {$ $} (d);
				\end{tikzpicture}
			\end{equation}
			in which $Y_\al$ is a reasonable geometric substack of $Y$ and $f_\al$ is tamely presented. 
		\end{minipage}
	\end{enumerate}	
\end{Proposition}

\begin{Definition}\label{def:indP}
	A morphism $f: X \to Y$ of reasonable ind-geometric stacks is ind-tamely presented if it satisfies the equivalent conditions of Proposition \ref{prop:indPequivconditions}. An ind-geometric stack is ind-tamely presented if it is reasonable and is ind-tamely presented over $\Spec \kk$. 
\end{Definition}

\begin{proof}[Proof of Proposition \ref{prop:indPequivconditions}]
	Same as \cite[\igpropindPequivconditions]{CWig}, but instead using Propositions \ref{prop:reasmorprops2} and~\ref{prop:wafp2of3prop}. 
\end{proof}

Note as a corollary that a morphism $f: X \to Y$ of truncated geometric stacks is tamely presented if and only if it is ind-tamely presented as a morphism of ind-geometric stacks. 

\begin{Proposition}\label{prop:indPcompprops}
	Ind-tamely presented morphisms of reasonable ind-geometric stacks are stable under composition.   
\end{Proposition}

\begin{proof} 
	Same as \cite[\igpropindPcompprops]{CWig}, but using Proposition \ref{prop:reasmorprops2}. 
\end{proof}

  \begin{Proposition}\label{prop:tamefiltcolims}
Ind-tamely presented ind-geometric stacks are closed under filtered colimits along almost ind-finitely presented ind-closed immersions  in $\Stkkconv$. 
\end{Proposition}
\begin{proof} 
	Same as \cite[\igpropcohfiltcolims]{CWig}. 
\end{proof}

As in the geometric case, the notions above are usefully combined with the following one. 

\begin{Definition}\label{def:admissibleind}
An ind-geometric stack is admissible if it admits an ind-geometric presentation whose terms are admissible geometric stacks. 
\end{Definition}

We write $\indGStkkadtm \subset \indGStkk$ for the full subcategory of admissible, ind-tamely presented ind-geometric stacks. It follows from \cite[\igpropigprestermsarereas]{CWig} that if $X$ is admissible then any reasonable geometric substack of $X$ is admissible. In particular, if $X \cong \colim X_\al$ is a reasonable presentation, then each $X_\al$ is coherent hence $X$ itself is coherent. It follows from the corresponding fact in the geometric case that $\indGStkkadtm$ is closed under products, whereas the larger category $\indGStkkcoh$ is not. 

\subsection{Coherent pullback} 
Recall that a morphism $f: X \to Y$ of stacks is geometric if for any morphism $Y' \to Y$ with $Y'$ an affine scheme, $Y' \times_Y X$ is a geometric stack.  Proposition~\ref{prop:gstkprops} implies more generally that $Y' \times_Y X$ is a geometric stack whenever $Y'$ is. 

\begin{Definition}\label{def:indgeomcohpull}
	Let $f: X \to Y$ be a morphism of reasonable ind-geometric stacks such that $Y$ is reasonable. We say $f$ has \emph{\stable coherent pullback} if it is geometric and for every truncated geometric stack $Y'$ and every ind-tamely presented morphism $Y' \to Y$, the base change $f': Y' \times_Y X \to Y'$ has coherent pullback. 
\end{Definition}

Following the conventions of \cite{CWig}, we say a morphism $f: X \to Y$ of ind-geometric stacks is of finite Tor-dimension if it is geometric and its base change to any geometric stack is of finite Tor-dimension in the geometric sense. It follows immediately that if $Y$ is reasonable, then $f$ also has \stable coherent pullback. 

\begin{Proposition}\label{prop:cancheckpresforcohpull}
	Let $f: X \to Y$ be a geometric morphism of ind-geometric stacks and let $Y \cong \colim_\al Y_\al$ be a reasonable presentation. Then $f$ has \stable coherent pullback if and only if its base change to every $Y_\al$ has \stable coherent pullback. 
\end{Proposition}
\begin{proof}
	The if direction follows from the definitions. Each $Y_\al \to Y$ is ind-tamely presented by \cite[\igpropafpdefconsistent]{CWig} and Proposition \ref{prop:afpimplieswafp}, hence the only if direction follows from the stability of ind-tamely presented morphisms under composition (Proposition \ref{prop:indPcompprops}). 
\end{proof}

\begin{Proposition}\label{prop:indcohpullcompprop}
	Morphisms with \stable coherent pullback are stable under composition in $\indGStkkreas$.
\end{Proposition}
\begin{proof}
	Let $X$, $Y$, and $Z$ be reasonable ind-geometric stacks, $f: X \to Y$ and $g: Y \to Z$ morphisms with \stable coherent pullback, and $Z \cong \colim Z_\al$ a reasonable presentation. Define  $g_\al: Y_\al \to Z_\al$ and $f_\al: X_\al \to Y_\al$ by base change. Then each $Y_\al$ is a truncated geometric stack since $g_\al$ has coherent pullback, the maps $Y_\al \to Y_\be$ are almost finitely presented closed immersions by base change, and $Y \cong \colim Y_\al$ by left exactness of filtered colimits in $\Stkkconv$. In particular each $f_\al$ has \stable coherent pullback, hence each $g_\al \circ f_\al$ does by Proposition~\ref{prop:cohpullprops}, hence $g \circ f$ does by Proposition \ref{prop:cancheckpresforcohpull}. 
\end{proof}

\subsection{Fiber products}
We now consider the base change properties of the above classes of morphisms. Recall that reasonable ind-geometric stacks are not closed under arbitrary fiber products, though general ind-geometric stacks are. 

\begin{Proposition}\label{prop:indPbaseprops}
	If $Y$ is a reasonable ind-geometric stack and $h: X \to Y$ a morphism with \stable coherent pullback, then $X$ is reasonable. In particular, let the following be a Cartesian diagram of ind-geometric stacks. 
	\begin{equation*}
		\begin{tikzpicture}
			[baseline=(current  bounding  box.center),thick,>=\arrtip]
			\node (a) at (0,0) {$X'$};
			\node (b) at (3,0) {$Y'$};
			\node (c) at (0,-1.5) {$X$};
			\node (d) at (3,-1.5) {$Y$};
			\draw[->] (a) to node[above] {$f' $} (b);
			\draw[->] (b) to node[right] {$h $} (d);
			\draw[->] (a) to node[left] {$h' $}(c);
			\draw[->] (c) to node[above] {$f $} (d);
		\end{tikzpicture}
	\end{equation*}
	If $X$, $Y$, and $Y'$ are reasonable, $h$ has \stable coherent pullback, and $f$ is ind-tamely presented, then $X'$ is reasonable. If $Y'$ is ind-tamely presented, so is $X'$. 
\end{Proposition}
\begin{proof}
Same as \cite[\igpropindPbaseprops]{CWig}, but using Proposition \ref{prop:indPcompprops}. 
\end{proof}

Recall also that coherent ind-geometric stacks, and even coherent affine schemes, are not closed under arbitrary fiber products. However, we do have the following enhancement of Proposition \ref{prop:indPbaseprops} in the coherent case. The hypotheses apply in particular when $X$, $Y$, and $Y'$ are admissible and ind-tamely presented, in which case so is $X'$. 

\begin{Proposition}\label{prop:haffimpliesX'coh}
	Let the following be a Cartesian diagram of ind-geometric stacks. 
	\begin{equation*}
		\begin{tikzpicture}
			[baseline=(current  bounding  box.center),thick,>=\arrtip]
			\node (a) at (0,0) {$X'$};
			\node (b) at (3,0) {$Y'$};
			\node (c) at (0,-1.5) {$X$};
			\node (d) at (3,-1.5) {$Y$};
			\draw[->] (a) to node[above] {$f' $} (b);
			\draw[->] (b) to node[right] {$h $} (d);
			\draw[->] (a) to node[left] {$h' $}(c);
			\draw[->] (c) to node[above] {$f $} (d);
		\end{tikzpicture}
	\end{equation*}
	Suppose that $X$ is coherent, $Y$ is reasonable, and $Y'$ is ind-tamely presented. 
	Suppose also that $f$ is ind-tamely presented and that $h$ is affine and has \stable coherent pullback. Then $X'$ is coherent. 
\end{Proposition}
\begin{proof}
	If $X \cong \colim X_\al$ is a reasonable presentation, then as in the proof of Proposition \ref{prop:indPbaseprops} $X' \cong \colim X'_\al$ is a reasonable presentation, where $X'_\al := X_\al \times_Y Y'$. Since $Y'$ and $f'$ are ind-tamely presented it follows that $X'$ is as well (Proposition \ref{prop:indPcompprops}), hence each $X'_\al$ is locally coherent (Proposition \ref{prop:stablecoh}). But since $\QCoh(X_\al)^\heartsuit$ is compactly generated by hypothesis, so is $\QCoh(X'_\al)^\heartsuit$ by \cite[\iglemaffcgen]{CWig} and the fact that $h'$ is affine since $h$ is. 
\end{proof}

\subsection{Coherent sheaves}\label{sec:indseccohsubsec}
We now make explicit the functoriality of coherent sheaves under \stable coherent pullback among reasonable ind-geometric stacks. Recall that if $X \cong \colim X_\al$ is a reasonable presentation, the category $\Coh(X)$ is computed by the formula
\begin{equation}\label{eq:cohviaindgeopres}
	\Coh(X) \cong \colim \Coh(X_\al)
\end{equation} 
in $\Catinfty$. 
The functoriality of this category under \stable coherent pullback may be described by extending the construction of \cite[\igseccohsubsec]{CWig}

Recall from (\ref{eq:CorrCohGStkpropcoh}) that coherent pullback of coherent sheaves on truncated geometric stacks can be extended to a functor 
\begin{equation}\label{eq:CorrCohGStkpropcoh2}
	\Coh: \Corr(\GStkkplus)_{prop,coh} \to \Catinfty.
\end{equation}
We now let $\Corr(\indGStkkreas)_{prop,coh}$ denote the 1-full subcategory of $\Corr(\indGStkk)$ which only includes correspondences $X \xleftarrow{h} Y \xrightarrow{f} Z$ such that $h$ has \stable coherent pullback, $f$ is ind-proper and almost ind-finitely presented, and $X$ and $Z$ are reasonable (hence so is $Y$ by Proposition \ref{prop:indPbaseprops}). These are indeed stable under composition of correspondences by Propositions \ref{prop:indPcompprops}, \ref{prop:indcohpullcompprop}, and \ref{prop:indPbaseprops}, and we note that $\Corr(\GStkkplus)_{prop,coh}$ is a full subcategory of $\Corr(\indGStkk)_{prop,coh}$. 

\begin{Definition}\label{def:cohonindgstks}
	We define a functor
	\begin{equation}\label{eq:cohcorrindgstkfunctorcoh}
		\Coh: \Corr(\indGStkkreas)_{prop,coh} \to \Catinfty
	\end{equation}
	by left Kan extending (\ref{eq:CorrCohGStkpropcoh2}) along $\Corr(\GStk^+)_{prop,coh} \subset \Corr(\indGStkkreas)_{prop,coh}$.   
\end{Definition}

This Kan extension exists by \cite[Cor. 4.3.2.16, Cor. 4.2.4.8]{LurHTT}. Taking either $h$ or $f$ to be the identity, the values of (\ref{eq:cohcorrindgstkfunctorcoh}) on the correspondence $X \xleftarrow{h} Y \xrightarrow{f} Z$ define functors $h^*: \Coh(X) \to \Coh(Y)$ and $f_*: \Coh(Y) \to \Coh(Z)$. The following result is proved the same way as \cite[\igpropcohonindgstks]{CWig}, and ensures the formula (\ref{eq:cohviaindgeopres}) as well as the consistency of the definitions given here with those of \cite[\igseccohsubsec]{CWig} when $h$ is of finite Tor-dimension. 

\begin{Proposition}\label{prop:cohonindgstks}
	The restriction of $\Coh$ to $\indGStkkprop$ preserves filtered colimits along almost ind-finitely presented ind-closed immersions. 
\end{Proposition}

\section{Coherent pullback and $!$-pullback}\label{sec:upper!}

In the previous section we considered the pushforward $f_*: \Coh(X) \to \Coh(Y)$ along an ind-proper, almost ind-finitely presented morphism $f: X \to Y$. This extends to a pushforward of ind-coherent sheaves, and this extension has a right adjoint $f^!: \IndCoh(Y) \to \IndCoh(X)$. The main goal of this section is to show this $!$-pullback functor is compatible with coherent pullback in the natural sense (Proposition \ref{prop:up!up*genICcaseIG}). The most subtle technical issue in this section (and the next) is that, while the functor $h^*: \IndCoh(Y) \to \IndCoh(Y')$ associated to a morphism with \stable coherent pullback will not be bounded below in general, only the bounded below part of $\IndCoh$ satisfies flat descent. 

\subsection{Coherent pullback and ind-coherent sheaves} 
If $X$ is a geometric stack, $\IndCoh(X)$ is characterized by the existence of a t-structure and a colimit-preserving t-exact functor $\Psi_X: \IndCoh(X) \to \QCoh(X)$ satisfying a universal property \cite[\igdefIndCohononeGStk]{CWig}. For our purposes, this is summarized by the statement that for any other geometric stack $Y$, composition with $\Psi_X$ and $\Psi_Y$ induces a pair of equivalences 
	\begin{equation}\label{eq:ICunivprop}
	\begin{tikzpicture}
		[baseline=(current  bounding  box.center),thick,>=\arrtip]
		\node (a) at (0,-1.5) {$\LFun^b(\IndCoh(X), \IndCoh(Y))$};
		\node (b) at (3.0,0) {$\LFun^b(\IndCoh(X), \QCoh(Y))$};
		\node (c) at (2*3.0,-1.5) {$\LFun^b(\QCoh(X), \QCoh(Y))$};
		\draw[<-] (b) to node[above, sloped, pos=.6] {$\sim $} (a);
		\draw[<-] (b) to node[above, sloped, pos=.6] {$\sim $} (c);
	\end{tikzpicture}
\end{equation}
Here $\LFun^b(-,-)$ denotes the category of bounded, colimit-preserving functors. When $X$ is coherent, $\IndCoh(X)$ is compactly generated by $\Coh(X)$. This is the main case of interest, and by abuse we let it guide the general terminology. 

From (\ref{eq:ICunivprop}) one sees that $\IndCoh(-)$ is functorial in $\GStkk$ under finite cohomological dimension pushforward and finite Tor-dimension pullback. These functorialities extend to the ind-geometric setting by the construction reviewed for coherent sheaves in Section \ref{sec:indseccohsubsec}. This leads to a functor
\begin{equation}\label{eq:IndCohindGStk}
	\IndCoh: \Corr(\indGStkk)_{fcd;ftd} \to \PrL,
\end{equation}
where $\Corr(\indGStkk)_{fcd;ftd}$ denotes the 1-full subcategory of $\Corr(\indGStkk)$ which only includes correspondences $X \xleftarrow{h} Y \xrightarrow{f} Z$ such that $h$ is of finite Tor-dimension and $f$ is of ind-finite cohomological dimension. Taking either $h$ or $f$ to be the identity, the values of (\ref{eq:IndCohindGStk}) on the correspondence $X \xleftarrow{h} Y \xrightarrow{f} Z$ define functors $h^*: \IndCoh(X) \to \IndCoh(Y)$ and $f_*: \IndCoh(Y) \to \IndCoh(Z)$. 

It remains true in the ind-geometric setting that $\IndCoh(X)$ is compactly generated by $\Coh(X)$ when $X$ is coherent \cite[\igpropIndCohisIndofCoh]{CWig}. This allows us to extend the basic functorialities of ind-coherent sheaves to include certain unbounded functors. 

\begin{Definition}\label{def:cohcohpul}
	Let $X$ and $Y$ be reasonable ind-geometric stacks such that $Y$ is coherent, and let $f: X \to Y$ be a morphism with coherent pullback. We write $h^*: \IndCoh(Y) \to \IndCoh(X)$ for the unique continuous functor whose restriction to $\Coh(Y)$ factors through the functor $h^*: \Coh(Y) \to \Coh(X)$. 
\end{Definition}

When $h$ is of finite Tor-dimension this is indeed consistent with (\ref{eq:IndCohindGStk}), since the previously defined $h^*$ is continuous and preserves coherence. Note also that $h^*$ as defined above is still right t-exact, since $\IndCoh(Y)^{\leq 0}$ is compactly generated by $\Coh(Y)^{\leq 0}$. 

To describe the pushforward counterpart of Definition \ref{def:cohcohpul}, first note that if $f: X \to Y$ is any morphism of ind-geometric stacks such that $X$ is reasonable, there is a canonical functor $f_*: \Coh(X) \to \IndCoh(Y)$ defined as follows. Write $\IndCoh^+_{naive}: \indGStkk \to \Cathatinfty$ for the left Kan extension of the evident functor $\IndCoh^+: \GStkkplus \to \Cathatinfty$. Explicitly, $\IndCoh^+_{naive}(X)$ is the full subcategory of $\cF \in\IndCoh^+(X)$ which are pushed forward from some truncated geometric substack of $X$. By construction we have a functor $f_*: \IndCoh(X)^+_{naive} \to \IndCoh(Y)^+_{naive}$, while by the universal property of left Kan extensions we have canonical functors $\Coh(X) \to \IndCoh(X)^+_{naive}$ and $\IndCoh(Y)^+_{naive} \to \IndCoh(Y)^+$, and we let $f_*: \Coh(X) \to \IndCoh(Y)$ be the composition of these. 

\begin{Definition}
	Let $f: X \to Y$ be a morphism of ind-geometric stacks, and suppose that $X$ is coherent. Then we write $f_*: \IndCoh(X) \to \IndCoh(Y)$ for the unique continuous functor whose restriction to $\Coh(X)$ is as described above. 
\end{Definition}

Note that $f_*$ as defined above is still left t-exact, since $\IndCoh(X)^{\geq 0}$ is compactly generated by $\Coh(X)^{\geq 0}$ and the t-structure on $\IndCoh(Y)$ is compatible with filtered colimits. Suppose that $g: Y \to Z$ is another morphism of ind-geometric stacks, and that either $Y$ is coherent or $g$ is of ind-finite cohomological dimension. Then we have an isomorphism $g_* f_* \cong (g \circ f)_*$ of functors $\IndCoh(X) \to \IndCoh(Z)$, since both are continuous and their restrictions to $\Coh(X)$ are isomorphic by construction. 

\begin{Proposition}\label{prop:coherentup*low*adj}
Let $Y'$ and $Y$ be coherent ind-geometric stacks and $h: Y' \to Y$ a morphism with \stable coherent pullback. Then $h_*: \IndCoh(Y') \to \IndCoh(Y)$ is right adjoint to $h^*: \IndCoh(Y) \to \IndCoh(Y')$. In particular, suppose $h$ sits in a Cartesian diagram of the following form, where $h'$ also has \stable coherent pullback and where $X'$ and $Y'$ are coherent.
\begin{equation}\label{eq:cohpushpull}
	\begin{tikzpicture}
		[baseline=(current  bounding  box.center),thick,>=\arrtip]
		\node (a) at (0,0) {$X'$};
		\node (b) at (3,0) {$Y'$};
		\node (c) at (0,-1.5) {$X$};
		\node (d) at (3,-1.5) {$Y$};
		\draw[->] (a) to node[above] {$f' $} (b);
		\draw[->] (b) to node[right] {$h $} (d);
		\draw[->] (a) to node[left] {$h' $}(c);
		\draw[->] (c) to node[above] {$f $} (d);
	\end{tikzpicture}
\end{equation}
Then the isomorphism $h_* f'_* \cong f_* h'_*$ of functors $\IndCoh(X') \to \IndCoh(Y)$ gives rise to a Beck-Chevalley transformation $h^* f_* \to f'_* h'^*$ of functors $\IndCoh(X) \to \IndCoh(Y')$, which is an isomorphism if $f$ is ind-proper and almost ind-finitely presented. 
\end{Proposition}

The proof uses the following fact, where as in \cite[\igsecanticompletion]{CWig} we write $\PrStbacpl$ for the $\infty$-category whose objects are presentable stable $\infty$-categories equipped with t-structures that are right complete, left anticomplete, and compatible with filtered colimits, and whose morphisms are bounded, colimit-preserving functors. 

\begin{Lemma}\label{lem:adjointsandPsi}
	Let $\wh{\catC}, \wh{\catD}$ be the left completions of $\wc{\catC}, \wc{\catD} \in \PrStbacpl$, and let $\Psi_\catC: \wcC \to \whC$ and $\Psi_\catD: \wcD \to \whD$ be the canonical functors. Let $\wc{F}: \wc{\catC} \to \wc{\catD}$, $\wh{F}: \wh{\catC} \to \wh{\catD}$ be  colimit-preserving functors such that $\wh{F}$ is right bounded and $\Psi_\catD \wc{F} \cong \wh{F} \Psi_\catC$. Then the Beck-Chevalley map $\Psi_\catC \wc{F}^R(X) \to \wh{F}^R \Psi_\catD(X)$ is an isomorphism for all $X \in \wcD^+$. 
\end{Lemma}
\begin{proof}
	By definition the Beck-Chevalley map is the composition
	\begin{equation}\label{eq:adjointsandPsi}
		\Psi_\catC \wc{F}^R(X) \to \Psi_\catC \wc{F}^R \Psi_\catD^R \Psi_\catD(X) \cong \Psi_\catC \Psi_\catC^R \wh{F}^R \Psi_\catD(X) \to \wh{F}^R \Psi_\catD(X)
	\end{equation}
	of unit and counit maps. Since $\Psi_\catC$ is t-exact and restricts to an equivalence $\wcC^+ \congto \whC^+$, its right adjoint $\Psi_\catC^R$ is left t-exact and restricts to the inverse equivalence $\whC^+ \congto \wcC^+$, likewise for $\Psi_\catD^R$. In particular, the first map in (\ref{eq:adjointsandPsi}) is an isomorphism since $X \in \wcD^+$. But $\wh{F}^R$ is left bounded since $\wh{F}$ is right bounded, hence $\wh{F}^R \Psi_\catD(X) \in \whC^+$, hence the last map in (\ref{eq:adjointsandPsi}) is also an isomorphism. 
\end{proof}

\begin{proof}[Proof of Proposition \ref{prop:coherentup*low*adj}]
	First suppose $Y'$ and $Y$ are truncated and geometric. Since $\IndCoh(Y')$ is compactly generated and $h^*_{IC}$ preserves compactness, the right adjoint $h^{*R}_{IC}$ is continuous. The restrictions of $\Psi_{Y} h_{IC*}$ and $h_{QC*} \Psi_{Y'}$ to $\Coh(Y')$ are isomorphic by definition, while the restrictions of $h_{QC*} \Psi_{Y'}$ and $\Psi_Y h^{*R}_{IC}$ are isomorphic by Lemma \ref{lem:adjointsandPsi}. Since $\Psi_Y$ is conservative on $\IndCoh(Y)^+$, the continuity and left t-exactness of $h_{IC*}$ and $h^{*R}_{IC}$ then imply they are isomorphic. When $X$ and $X'$ are truncated and geometric, the final claim follows immediately since $h^* f_*$ and $f'_* h'^*$ are continuous and the transformation restricts to an isomorphism of functors $\Coh(X) \to \Coh(Y')$. 
	
	Now let $Y \cong \colim Y_\al$ be a reasonable presentation with index category $A$, let $h_\al: Y'_\al \to Y_\al$ be the base change of $h$, and let $i_{\al\be}' : Y'_\al \to Y'_\be$ the base change of $i_{\al\be}: Y_\al \to Y_\be$. Using (\ref{eq:cohcorrindgstkfunctorcoh}) and ind-extending, the categories $\IndCoh(-)$ and the functors $i_{\al\be*}$, $i'_{\al\be*}$, $h^*$, and $h_\al^*$ form a diagram $A^\triangleright \times \Delta^1 \to \PrL$, and by \cite[Cor. 5.1.2.3]{LurHTT}, \cite[Lem. 5.3.2.9]{LurHA} this induces a colimit diagram $A^\triangleright \to (\PrL)^{\Delta^1}$. Passing to right adjoints, and using the previous paragraph, the functors $i^!_{\al\be}$, $i^!_{\al\be}$, $h^{*R}$, and $h_{\al*}$ form a diagram $(A^\triangleright \times \Delta^1)^{\op} \to \PrL$. By \cite[Prop. 4.5.7.19]{LurHA} the functors $i_{\al\be*}$, $i_{\al\be*}$, $h^{*R}$, and $h_{\al_*}$ form a diagram $A^\triangleright \times \Delta^1 \to \PrL$ which induces a colimit diagram $A^\triangleright \to (\PrL)^{\Delta^1}$. On the other hand, unwinding the definition of $h_*$ we see that it is given by the same colimit in $(\PrL)^{\Delta^1}$, and the claim follows. The general case of the final claim now follows as in the geometric case. 
\end{proof}

We record the following extension of Lemma \ref{lem:adjointsandPsi} for use in the next section. 

\begin{Lemma}\label{lem:BCandPsi}
	Suppose $\wc{\catC}, \wc{\catD}, \wc{\catC}', \wc{\catD}' \in \PrStbacpl$ and that we have a diagram
	\begin{equation*}
		\begin{tikzpicture}[baseline=(current  bounding  box.center),thick,>=\arrtip]
			\newcommand*{\ha}{1.5}; \newcommand*{\hb}{1.5}; \newcommand*{\hc}{1.5};
\newcommand*{\va}{-.9}; \newcommand*{\vb}{-.9}; \newcommand*{\vc}{-.9}; 
			\node (ab) at (\ha,0) {$\wc{\catC}'$};
			\node (ad) at (\ha+\hb+\hc,0) {$\wc{\catD}'$};
			\node (ba) at (0,\va) {$\wh{\catC}'$};
			\node (bc) at (\ha+\hb,\va) {$\wh{\catD}'$};
			\node (cb) at (\ha,\va+\vb) {$\wc{\catC}$};
			\node (cd) at (\ha+\hb+\hc,\va+\vb) {$\wc{\catD}$};
			\node (da) at (0,\va+\vb+\vc) {$\wh{\catC}$};
			\node (dc) at (\ha+\hb,\va+\vb+\vc) {$\wh{\catD}$};
			\draw[->] (ab) to node[above] {$\wc{F}' $} (ad);
			\draw[->] (ab) to node[above left, pos=.25] {$ $} (ba);
			\draw[->] (ab) to node[left,pos=.8] {$\wc{G}' $} (cb);
			\draw[->] (ad) to node[below right] {$ $} (bc);
			\draw[->] (ad) to node[right] {$\wc{G} $} (cd);
			\draw[->] (ba) to node[left] {$\wh{F}' $} (da);
			\draw[->] (cb) to node[above,pos=.25] {$\wc{F} $} (cd);
			\draw[->] (cb) to node[above left, pos=.25] {$ $} (da);
			\draw[->] (cd) to node[below right] {$ $} (dc);
			\draw[->] (da) to node[above,pos=.75] {$\wh{F} $} (dc);
			
			\draw[-,line width=6pt,draw=white] (ba) to  (bc);
			\draw[->] (ba) to node[above,pos=.75] {$\wh{F}' $} (bc);
			\draw[-,line width=6pt,draw=white] (bc) to  (dc);
			\draw[->] (bc) to node[right,pos=.2] {$\wh{G} $} (dc);
		\end{tikzpicture}
	\end{equation*}
	in $\PrL$ such that rightward arrows are right bounded and the forward arrows are the completion functors $\Psi_{\catC}$, $\Psi_{\catD}$, $\Psi_{\catC'}$, and $\Psi_{\catD'}$. Suppose $X \in \wc{\catD}'^+$ is such that $\wc{G}(X) \in \wc{\catD}^+$. Then if one of the Beck-Chevalley maps $\wh{G}' \wh{F}'^R \Psi_{\catD'}(X) \to \wh{F}^R \wh{G} \Psi_{\catD'}(X)$ and $\Psi_{\catC} \wc{G}' \wc{F}'^R(X) \to \Psi_{\catC}  \wc{F}^R \wc{G}(X)$ is an isomorphism, so is the other. 
\end{Lemma}
\begin{proof}
	The given morphisms are the top left and bottom right arrows in the following diagram.
	\begin{equation*}
		\begin{tikzpicture}
			[baseline=(current  bounding  box.center),thick,>=\arrtip]
			\newcommand*{\ha}{3.5}; \newcommand*{\hb}{3.5};
			\newcommand*{\va}{-1.5};
			\node (aa) at (0,0) {$\Psi_{\catC} \wc{G}' \wc{F}'^R(X)$};
			\node (ab) at (\ha,0) {$\Psi_{\catC}  \wc{F}^R \wc{G}(X)$};
			\node (ac) at (\ha+\hb,0) {$  \wh{F}^R \Psi_{\catD} \wc{G}(X)$};
			\node (ba) at (0,\va) {$\wh{G}'  \Psi_{\catC'} \wc{F}'^R(X)$};
			\node (bb) at (\ha,\va) {$\wh{G}' \wh{F}'^R \Psi_{\catD'}(X)$};
			\node (bc) at (\ha+\hb,\va) {$\wh{F}^R \wh{G} \Psi_{\catD'}(X)$};
			\draw[->] (aa) to node[above] {$ $} (ab);
			\draw[->] (ab) to node[above] {$  $} (ac);
			\draw[->] (ba) to node[above] {$ $} (bb);
			\draw[->] (bb) to node[above] {$ $} (bc);
			\draw[->] (aa) to node[below,rotate=90] {$\sim $} (ba);
			\draw[->] (ac) to node[below,rotate=90] {$\sim $} (bc);
		\end{tikzpicture} 
	\end{equation*}
	The claim follows since the top right and bottom left arrows are isomorphisms by Lemma~\ref{lem:adjointsandPsi} and our hypotheses on $X$. 
\end{proof}

\subsection{Coherent pullback and $!$-pullback}\label{sec:upper!upper*}
We now turn to the compatibility between coherent pullback and $!$-pullback. We begin by considering quasi-coherent sheaves on geometric stacks, extending from there to the case of ind-coherent sheaves and then ind-geometric stacks. In each case we begin with a base change isomorphism $h^* f_* \cong h'^* f'_*$, which leads to an associated Beck-Chevalley map $h'^* f^! \to f'^! h^*$. As with other functors, we write $f_{QC}^!$ and $f_{IC}^!$ when the meaning of $f^!$ is not otherwise clear from context. 

\begin{Proposition}\label{prop:up!up*gencohpullcaseQCoh}
Let the following be a Cartesian diagram of truncated, tamely presented geometric stacks.
\begin{equation*}
	\begin{tikzpicture}
		[baseline=(current  bounding  box.center),thick,>=\arrtip]
		\node (a) at (0,0) {$X'$};
		\node (b) at (3,0) {$Y'$};
		\node (c) at (0,-1.5) {$X$};
		\node (d) at (3,-1.5) {$Y$};
		\draw[->] (a) to node[above] {$f' $} (b);
		\draw[->] (b) to node[right] {$h $} (d);
		\draw[->] (a) to node[left] {$h' $}(c);
		\draw[->] (c) to node[above] {$f $} (d);
	\end{tikzpicture}
\end{equation*}
If $h$ has \stable coherent pullback and $f$ is proper and almost finitely presented, then the Beck-Chevalley map $h'^* f^!(\cF) \to f'^! h^*(\cF)$ in $\QCoh(X')$ is an isomorphism for all $\cF \in \Coh(Y)$. 
\end{Proposition}

\begin{proof}
	Let $\phi: U \cong \Spec A \to Y$ be a strictly tamely presented flat cover such that $A$ is strictly tamely presented over $\kk$. We obtain a diagram 
	\begin{equation}\label{eq:up!compcube}
		\begin{tikzpicture}[baseline=(current  bounding  box.center),thick,>=\arrtip]
			\newcommand*{\ha}{1.5}; \newcommand*{\hb}{1.5}; \newcommand*{\hc}{1.5};
\newcommand*{\va}{-.9}; \newcommand*{\vb}{-.9}; \newcommand*{\vc}{-.9}; 
			\node (ab) at (\ha,0) {$Z'$};
			\node (ad) at (\ha+\hb+\hc,0) {$U'$};
			\node (ba) at (0,\va) {$X'$};
			\node (bc) at (\ha+\hb,\va) {$Y'$};
			\node (cb) at (\ha,\va+\vb) {$Z$};
			\node (cd) at (\ha+\hb+\hc,\va+\vb) {$U$};
			\node (da) at (0,\va+\vb+\vc) {$X$};
			\node (dc) at (\ha+\hb,\va+\vb+\vc) {$Y$};
			\draw[->] (ab) to node[above] {$g' $} (ad);
			\draw[->] (ab) to node[above left, pos=.25] {$\psi' $} (ba);
			\draw[->] (ab) to node[right,pos=.2] {$\xi' $} (cb);
			\draw[->] (ad) to node[below right] {$\psi $} (bc);
			\draw[->] (ad) to node[right] {$\xi $} (cd);
			\draw[->] (ba) to node[left] {$ $} (da);
			\draw[->] (cb) to node[above,pos=.25] {$g $} (cd);
			\draw[->] (cb) to node[above left, pos=.25] {$\phi' $} (da);
			\draw[->] (cd) to node[below right] {$\phi $} (dc);
			\draw[->] (da) to node[above,pos=.75] {$ $} (dc);
			
			\draw[-,line width=6pt,draw=white] (ba) to  (bc);
			\draw[->] (ba) to node[above,pos=.75] {$ $} (bc);
			\draw[-,line width=6pt,draw=white] (bc) to  (dc);
			\draw[->] (bc) to node[right,pos=.2] {$ $} (dc);
		\end{tikzpicture}
	\end{equation}
	with Cartesian faces. Note that $\xi$ has \stable coherent pullback by Proposition \ref{prop:cohpullprops}. Since $\psi'^*$ is conservative, it suffices to show the top left arrow in
	\begin{equation}\label{eq:cHomcohpullbasechange1}
		\begin{tikzpicture}
			[baseline=(current  bounding  box.center),thick,>=\arrtip]
			\newcommand*{\ha}{3.5}; \newcommand*{\hb}{3.5};
			\newcommand*{\va}{-1.5};
			\node (aa) at (0,0) {$\psi'^* h'^* f^!(\cF)$};
			\node (ab) at (\ha,0) {$\psi'^* f'^! h^*(\cF)$};
			\node (ac) at (\ha+\hb,0) {$g'^!\psi^* h^*(\cF)$};
			\node (ba) at (0,\va) {$\xi'^* \phi'^*f^!(\cF)$};
			\node (bb) at (\ha,\va) {$\xi'^* g^! \phi^*(\cF)$};
			\node (bc) at (\ha+\hb,\va) {$g'^! \xi^* \phi^*(\cF)$};
			\draw[->] (aa) to node[above] {$ $} (ab);
			\draw[->] (ab) to node[above] {$  $} (ac);
			\draw[->] (ba) to node[above] {$ $} (bb);
			\draw[->] (bb) to node[above] {$ $} (bc);
			\draw[->] (aa) to node[below,rotate=90] {$\sim $} (ba);
			\draw[->] (ac) to node[below,rotate=90] {$\sim $} (bc);
		\end{tikzpicture} 
	\end{equation}
	is an isomorphism. The bottom left and top right arrows are isomorphisms by \cite[\igpropuppershriekftdbasechange]{CWig}, so it suffices to show the bottom right is. 
	
	By hypothesis $A$ is truncated, so for some $n$ there exists a strictly tame presentation $A \cong \colim A_\al$ of order $n$. By Noetherian approximation \cite[Prop. 4.2.1.5, Thm. 4.4.2.2, Prop. 5.5.4.1]{LurSAG} and flatness of the $u_\al: U \to U_\al = \Spec A_\al$, we can extend the back face of (\ref{eq:up!compcube}) to a diagram
\begin{equation}\label{eq:cHomcohpullbasechange2}
	\begin{tikzpicture}
		[baseline=(current  bounding  box.center),thick,>=\arrtip]
		\newcommand*{\ha}{3};
		\newcommand*{\va}{-1.5};
		\newcommand*{\vb}{-1.5};
		\node (aa) at (0,0) {$Z'$};
		\node (ab) at (\ha,0) {$U'$};
		\node (ba) at (0,\va) {$Z$};
		\node (bb) at (\ha,\va) {$U$};
		\node (ca) at (0,\va+\vb) {$Z_\al$};
		\node (cb) at (\ha,\va+\vb) {$U_\al$};
		\draw[->] (aa) to node[above] {$g' $} (ab);
		\draw[->] (ba) to node[above] {$g $} (bb);
		\draw[->] (aa) to node[left] {$\xi' $} (ba);
		\draw[->] (ab) to node[right] {$\xi $} (bb);
		
		\draw[->] (ca) to node[above] {$g_\al $} (cb);
		\draw[->] (ba) to node[left] {$u'_\al $} (ca);
		\draw[->] (bb) to node[right] {$u_\al $} (cb);
	\end{tikzpicture}
\end{equation} 
of Cartesian squares in which $g_\al$ is proper and almost finitely presented. Increasing $\al$ if needed, we may assume by \cite[Cor. 4.5.1.10]{LurSAG}, flatness of the $u_\al$, and coherence of the $A_\al$ that there exists $\cF_\al \in \Coh(U_\al)$ such that $\phi^*(\cF) \cong u^*_\al(\cF_\al)$.

The bottom right arrow of (\ref{eq:cHomcohpullbasechange1}) is then also the second factor of
\begin{equation*}
	\begin{tikzpicture}
		[baseline=(current  bounding  box.center),thick,>=\arrtip]
		\newcommand*{\ha}{3.0}; \newcommand*{\hb}{3.0};
		\node (aa) at (0,0) {$\xi'^* u'^*_\al g^!_\al(\cF_\al)$};
		\node (ab) at (\ha,0) {$\xi'^* g^! u^*_\al(\cF_\al)$};
		\node (ac) at (\ha+\hb,0) {$ g'^!\xi^*u^*_\al(\cF_\al).$};
		\draw[->] (aa) to node[above] {$ $} (ab);
		\draw[->] (ab) to node[above] {$  $} (ac);
	\end{tikzpicture} 
\end{equation*}
But $u_\al\circ \xi$ is of finite Tor-dimension by Proposition \ref{prop:finitetorapproxgstk}, so the composition and first factor are isomorphisms by \cite[Prop. 6.4.1.4]{LurSAG}, and thus so is the second. 
\end{proof}

Next we consider the ind-coherent analogue of Proposition \ref{prop:up!up*gencohpullcaseQCoh}. In this setting we can strengthen the result by further assuming the stacks involved are coherent. We caution that in the following statement the coherence of $X'$ does not follow from the coherence of $X$, $Y$, and~$Y'$, but rather will follows from other hypotheses in practice. 

\begin{Proposition}\label{prop:up!up*gencohpullcase}
	Let the following be a Cartesian diagram of truncated, coherent, tamely presented geometric stacks.
	\begin{equation*}
		\begin{tikzpicture}
			[baseline=(current  bounding  box.center),thick,>=\arrtip]
			\node (a) at (0,0) {$X'$};
			\node (b) at (3,0) {$Y'$};
			\node (c) at (0,-1.5) {$X$};
			\node (d) at (3,-1.5) {$Y$};
			\draw[->] (a) to node[above] {$f' $} (b);
			\draw[->] (b) to node[right] {$h $} (d);
			\draw[->] (a) to node[left] {$h' $}(c);
			\draw[->] (c) to node[above] {$f $} (d);
		\end{tikzpicture}
	\end{equation*}
	Suppose that $h$ is of finite cohomological dimension and has \stable coherent pullback, and that $f$ is proper and almost finitely presented. Then the Beck-Chevalley map $h'^* f^!(\cF) \to f'^! h^*(\cF)$ in $\IndCoh(X')$ is an isomorphism for all $\cF \in \IndCoh(Y)$. 
\end{Proposition}

We will use the following standard result, see \cite[Cor. 5.1.5]{BKV22}. 

\begin{Lemma}\label{lem:colimpres}
	Let $\catC \cong \colim \catC_\al$ be the colimit of a diagram $A \to \PrL$, with $F_\al: \catC_\al \to \catC$ the canonical functors and $G_\al: \catC \to \catC_\al$ their right adjoints. 
	Then for any $X \in \catC$, the objects $X_\al \cong \colim F_\al G_\al (X)$ assemble into a diagram $A \to \catC$ whose colimit is $X$. 
\end{Lemma}

\begin{proof}[Proof of Proposition \ref{prop:up!up*gencohpullcase}]
First note $f_*$ and $f'_*$ preserve compact objects and have compactly generated source \cite[\igpropIndCohisIndofCoh]{CWig}, hence $f^!$ and $f'^!$ are continuous \cite[Prop. 5.5.7.2]{LurHTT}. By compact generation of $\IndCoh(Y)$ it then suffices to assume $\cF \in \Coh(Y)$. 

Since $h$ and $h'$ are of finite cohomological dimension, the $\Psi_{(-)}$ functors and the quasi-coherent and ind-coherent versions of $f_*$, $f'_*$, $h_*$, and $h'_*$ form the edges of a diagram $(\Delta^1)^3 \to \Cathatinfty$. By Proposition \ref{prop:coherentup*low*adj} we may pass to adjoints to obtain a corresponding diagram involving $h^*$ and~$h'^*$.  Proposition \ref{prop:up!up*gencohpullcaseQCoh} and Lemma \ref{lem:BCandPsi} imply that $\Psi_{X'} h'^* f^!(\cF) \to \Psi_{X'}f'^! h^*(\cF)$ is an isomorphism in $\QCoh(X')$. Since $f'^! h^*(\cF)$ is left bounded and $\Psi_{X'}$ is conservative on $\IndCoh(X')^+$, it suffices to show $h'^* f^!(\cF)$ is left bounded. 

We first claim that $\tau^{\leq n} f^!(\cF)$ is coherent for all $n$. As in the proof of Proposition \ref{prop:up!up*gencohpullcaseQCoh}, we fix diagrams (\ref{eq:up!compcube}), (\ref{eq:cHomcohpullbasechange2}) and a sheaf $\cF_\al \in \Coh(U_\al)$ such that $\phi^*(\cF) \cong u^*_\al(\cF_\al)$. It follows from \cite[\igcoruppershriekftdbasechangeIndCoh]{CWig} and \cite[Prop. 6.4.1.4]{LurSAG} that $\phi'^*  f^!(\cF) \cong u'^*_\al g_\al^! (\cF_\al)$. By flatness of $\phi'$ and $u'_\al$ we then have
\begin{equation}\label{eq:up!up*gencohpullcase1}
\phi'^* \tau^{\leq n} f^!(\cF) \cong \tau^{\leq n} \phi'^*  f^!(\cF) \cong \tau^{\leq n} u'^*_\al g_\al^! (\cF_\al) \cong u'^*_\al \tau^{\leq n}g_\al^! (\cF_\al). 	
\end{equation}
Thus $\tau^{\leq n} f^!(\cF)$ is coherent if $\tau^{\leq n}g_\al^! (\cF_\al)$ is, since $\phi'$ is faithfully flat and $\tau^{\leq n} f^!(\cF)$ is left bounded. But $\tau^{\leq n}g_\al^! (\cF_\al)$ is coherent for all $n$ by \cite[Prop. 6.4.3.4]{LurSAG}. 

Since the standard t-structure is right complete we have $f^!(\cF) \cong \colim \tau^{\leq n} f^!(\cF)$ by Lemma~\ref{lem:colimpres}. Since it is compatible with filtered colimits, and since $h'^*$ is continuous, $h'^* f^!(\cF)$ is then bounded below if the sheaves $h'^* \tau^{\leq n} f^!(\cF)$ are uniformly bounded below. Since $\psi'$ is faithfully flat, hence $\psi'^*$ conservative on $\IndCoh(X')^+$, it suffices to show this for the sheaves $\psi'^* h'^* \tau^{\leq n} f^!(\cF)$. But using (\ref{eq:up!up*gencohpullcase1}) we have
$$\psi'^* h'^* \tau^{\leq n} f^!(\cF) \cong \xi'^* \phi'^* \tau^{\leq n} f^!(\cF) \cong \xi'^* u'^*_\al \tau^{\leq n}g_\al^! (\cF_\al),$$ 
and the right hand side is uniformly left bounded since Proposition~\ref{prop:finitetorapproxgstk} implies $u_\al \circ \xi$ is of finite Tor-dimension, hence so is $u'_\al \circ \xi'$. 
\end{proof}

We now have the following extension to the ind-geometric setting. 

\begin{Proposition}\label{prop:up!up*genICcaseIG}
	Let the following be a Cartesian diagram of ind-geometric stacks. 
\begin{equation*}
	\begin{tikzpicture}
		[baseline=(current  bounding  box.center),thick,>=\arrtip]
		\node (a) at (0,0) {$X'$};
		\node (b) at (3,0) {$Y'$};
		\node (c) at (0,-1.5) {$X$};
		\node (d) at (3,-1.5) {$Y$};
		\draw[->] (a) to node[above] {$f' $} (b);
		\draw[->] (b) to node[right] {$h $} (d);
		\draw[->] (a) to node[left] {$h' $}(c);
		\draw[->] (c) to node[above] {$f $} (d);
	\end{tikzpicture}
\end{equation*}
Suppose that all stacks in the diagram are coherent and ind-tamely presented, that $h$ is of ind-finite cohomological dimension and has \stable coherent pullback, and that $f$ is ind-proper and almost ind-finitely presented. Then for any $\cF \in \IndCoh(Y)$ the Beck-Chevalley map $h'^* f^!(\cF) \to f'^! h^*(\cF)$ is an isomorphism. 
\end{Proposition}

\begin{proof}
	Suppose first that $f$ is the inclusion of a reasonable geometric substack, which we may assume is a term in a reasonable presentation $Y \cong \colim Y_\al$ with index category $A$. Write $h_\al: Y'_\al \to Y_\al$ and $i'_{\al\be}: Y'_\al \to Y'_\be$ for the base changes of $h$ and $i_{\al\be}$. By construction, the functors $h^*_\al$, $i_{\al\be*}$, and $i'_{\al\be*}$ form a diagram $A \times \Delta^1 \to \PrL$. Note that since $h$ has \stable coherent pullback each $Y'_\al$ is a reasonable geometric substack of $Y'$, and in particular is coherent and tamely presented \cite[\igpropcohbasechangeunderafpclimm]{CWig}. Proposition \ref{prop:up!up*gencohpullcase} implies $h^*_\al i^!_{\al\be} \to i'^!_{\al\be} h^*_\be$ is an isomorphism for all $\al \leq \be$. Since  $\IndCoh(X) \cong \colim \IndCoh(X_\al)$ and  $\IndCoh(Y) \cong \colim \IndCoh(Y_\al)$ in $\PrL$ \cite[\igpropindcohonindgstks]{CWig}, the claim follows by \cite[Prop. 4.7.5.19]{LurHA}. 
	
	In the general case, let $X \cong \colim X_\al$ be a reasonable presentation. For any $\al$ we can factor $f \circ i_\al$ through some reasonable geometric substack $j_\al: Y_\al \to Y$, and have a diagram
	\begin{equation*}
		\begin{tikzpicture}[baseline=(current  bounding  box.center),thick,>=\arrtip]
			\newcommand*{\ha}{1.5}; \newcommand*{\hb}{1.5}; \newcommand*{\hc}{1.5};
			\newcommand*{\va}{-.9}; \newcommand*{\vb}{-.9}; \newcommand*{\vc}{-.9}; 
			\node (ab) at (\ha,0) {$X'_\al$};
			\node (ad) at (\ha+\hb+\hc,0) {$Y'_\al$};
			\node (ba) at (0,\va) {$X'$};
			\node (bc) at (\ha+\hb,\va) {$Y'$};
			\node (cb) at (\ha,\va+\vb) {$X_\al$};
			\node (cd) at (\ha+\hb+\hc,\va+\vb) {$Y_\al$};
			\node (da) at (0,\va+\vb+\vc) {$X$};
			\node (dc) at (\ha+\hb,\va+\vb+\vc) {$Y$};
			\draw[->] (ab) to node[above] {$f'_\al $} (ad);
			\draw[->] (ab) to node[above left, pos=.25] {$i'_\al $} (ba);
			\draw[->] (ab) to node[right,pos=.2] {$h'_\al $} (cb);
			\draw[->] (ad) to node[below right] {$j'_\al $} (bc);
			\draw[->] (ad) to node[right] {$h_\al $} (cd);
			\draw[->] (ba) to node[left] {$h' $} (da);
			\draw[->] (cb) to node[above,pos=.25] {$f_\al $} (cd);
			\draw[->] (cb) to node[above left, pos=.25] {$i_\al $} (da);
			\draw[->] (cd) to node[below right] {$j_\al $} (dc);
			\draw[->] (da) to node[above,pos=.75] {$f $} (dc);
			
			\draw[-,line width=6pt,draw=white] (ba) to  (bc);
			\draw[->] (ba) to node[above,pos=.75] {$ $} (bc);
			\draw[-,line width=6pt,draw=white] (bc) to  (dc);
			\draw[->] (bc) to node[right,pos=.2] {$h $} (dc);
		\end{tikzpicture}
	\end{equation*} 
	with all faces but the top and bottom Cartesian. We then have a diagram
	\begin{equation*}
		\begin{tikzpicture}
			[baseline=(current  bounding  box.center),thick,>=\arrtip]
			\newcommand*{\ha}{3.5}; \newcommand*{\hb}{3.5};
			\newcommand*{\va}{-1.5};
			\node (aa) at (0,0) {$h'^*_\al i^!_\al f^!(\cF)$};
			\node (ab) at (\ha,0) {$i'^!_\al h'^* f^!(\cF)$};
			\node (ac) at (\ha+\hb,0) {$i'^!_\al f'^! h^*(\cF)$};
			\node (ba) at (0,\va) {$h'^*_\al f^!_\al j^!_\al(\cF)$};
			\node (bb) at (\ha,\va) {$ f'^!_\al h^*_\al j^!_\al(\cF)$};
			\node (bc) at (\ha+\hb,\va) {$f'^!_\al j'^!_\al h^* (\cF)$};
			\draw[->] (aa) to node[above] {$ $} (ab);
			\draw[->] (ab) to node[above] {$  $} (ac);
			\draw[->] (ba) to node[above] {$ $} (bb);
			\draw[->] (bb) to node[above] {$ $} (bc);
			\draw[->] (aa) to node[below,rotate=90] {$\sim $} (ba);
			\draw[->] (ac) to node[below,rotate=90] {$\sim $} (bc);
		\end{tikzpicture} 
	\end{equation*}
in $\IndCoh(X'_\al)$. Since the functors $i'^!_\al$ determine an isomorphism $\IndCoh(Y) \cong \lim \IndCoh(Y_\al)$ in $\Cathatinfty$ \cite[\igpropindcohonindgstks]{CWig}, it suffices to show the top right arrow is an isomorphism for all $\al$.  But, given that $X'_\al$ and $Y'_\al$ are coherent \cite[\igpropcohbasechangeunderafpclimm]{CWig}, the top left and bottom right arrows are isomorphisms by the previous paragraph, and the bottom left is by Proposition \ref{prop:up!up*gencohpullcase}. 
\end{proof}

In the simpler case where $h$ is of finite Tor-dimension, Propositions \ref{prop:up!up*gencohpullcaseQCoh} and \ref{prop:up!up*gencohpullcase} follow from \cite[\igpropuppershriekftdbasechange, \igcoruppershriekftdbasechangeIndCoh]{CWig}. In Proposition~\ref{prop:up!up*genICcaseIG}, the hypothesis that $h$ is of finite Tor-dimension lets use drop the coherence hypotheses, provided we impose suitable boundedness conditions on $\cF$. To state this precisely, let $\IndCoh(Y)^{+}_{lim} \subset  \IndCoh(Y)$ denote the full subcategory of $\cF$ such that $i^!(\cF) \in \IndCoh(Y'')^+$ for every truncated geometric substack $i: Y'' \to Y$. Then one can extend the proof of Proposition \ref{prop:up!up*genICcaseIG} to show the following. 

\begin{Proposition}\label{prop:up!up*genICcaseIGftd}
	Let the following be a Cartesian diagram of ind-geometric stacks. 
	\begin{equation*}
		\begin{tikzpicture}
			[baseline=(current  bounding  box.center),thick,>=\arrtip]
			\node (a) at (0,0) {$X'$};
			\node (b) at (3,0) {$Y'$};
			\node (c) at (0,-1.5) {$X$};
			\node (d) at (3,-1.5) {$Y$};
			\draw[->] (a) to node[above] {$f' $} (b);
			\draw[->] (b) to node[right] {$h $} (d);
			\draw[->] (a) to node[left] {$h' $}(c);
			\draw[->] (c) to node[above] {$f $} (d);
		\end{tikzpicture}
	\end{equation*}
	Suppose that all objects in the diagram are reasonable, that $h$ is of finite Tor-dimension, and that $f$ is ind-proper and almost ind-finitely presented. Then for any $\cF \in \IndCoh(Y)^+_{lim}$ the Beck-Chevalley map $h'^* f^!(\cF) \to f'^! h^*(\cF)$ is an isomorphism. 
\end{Proposition}

\section{Coherent pullback and sheaf Hom}\label{sec:sheafHom}

Given a geometric stack $Y$, sheaf Hom from $\cF \in \QCoh(Y)$ is defined by the adjunction
\begin{equation*}
	- \otimes \cF : \QCoh(Y) \leftrightarrows \QCoh(Y) : \cHom(\cF, -).
\end{equation*}
If $Y$ is a reasonable ind-geometric stack and $\cF \in \Coh(Y)$, there is still a natural functor $\cHom(\cF, -): \IndCoh(Y) \to \IndCoh(Y)$, despite the absence of a tensor product on $\IndCoh(Y)$ in general. This is because ind-coherent sheaves do still admit external tensor products, leading to an adjunction 
\begin{equation*}
	- \boxtimes \cF : \IndCoh(X) \leftrightarrows \IndCoh(X \times Y) : (- \boxtimes \cF)^R
\end{equation*}
for any ind-geometric $X$. We denote these functors by $e_{\cF,X}$ and $e_{\cF,X}^R$ to make explicit their dependence on $X$. If $X$ and $Y$ are geometric, the quasi-coherent version of $e_{\cF,X}^R$ satisfies $\cHom(\cF, -) \cong e_{\cF,Y}^R \Delta_{Y*}$, and we take this as the definition of sheaf Hom in the ind-coherent, ind-geometric setting. 

A morphism $f: X \to Y$ of geometric stacks induces an isomorphism $f^*(-\otimes \cF) \cong - \otimes f^*(\cF)$ of functors $\QCoh(Y) \to \QCoh(X)$, hence a Beck-Chevalley map
\begin{equation}\label{eq:hompullback2}
	f^* \cHom (\cF,\cG) \rightarrow \cHom(f^* (\cF), f^*(\cG))
\end{equation}
for any $\cG \in \QCoh(Y)$. This is not an isomorphism in general, but is when $f$ is of finite Tor-dimension under certain hypotheses on $\cF$ and $\cG$. The basic goal of this section is to generalize this and related results, in particular allowing $f$ to have coherent pullback, $X$ and~$Y$ to be ind-geometric, and $\cF$ and $\cG$ to be ind-coherent (Proposition \ref{prop:eFXRupper*cohpullgeomcase}). A basic technical theme is the close analogy between $e_{\cF,X}^R$ and ind-proper $!$-pullback, with many proofs in this section being parallel to corresponding proofs in Section~\ref{sec:upper!}. 

\subsection{Ind-geometric external products} \label{sec:indextcohcase}

We first review and extend the construction of external products on ind-geometric stacks from \cite[\igsecsheafHom]{CWig}. Since it is necessary in order for these constructions to be well-defined, \emph{we assume that $\kk$ is an ordinary (Noetherian) ring of finite global dimension for the rest of the paper}. 

Following \cite[Sec. 9]{GR17}, external products are most fully encoded as lax symmetric monoidal structures on sheaf theories. The proof of  \cite[\igpropreascohlaxsymmstruct]{CWig} extends directly to show that the external product of quasi-coherent sheaves induces a lax symmetric monoidal structure on the functor
\begin{equation*}
	\Coh: \Corr(\GStkkplus)_{prop,coh} \to \Catinfty
\end{equation*}
considered in Section \ref{sec:geombc}, and that this extends to a lax symmetric monoidal structure on its left Kan extension 
	\begin{equation}\label{eq:cohcorrindgstkfunctorcoh2}
	\Coh: \Corr(\indGStkkreas)_{prop,coh} \to \Catinfty
\end{equation}
considered in Section \ref{sec:indseccohsubsec}. 

Explicitly, suppose $X \cong \colim X_\al$ and $Z \cong \colim Z_\be$ are reasonable presentations. Then the external product
\begin{equation}\label{eq:cohext}
	- \boxtimes -: \Coh(X) \times \Coh(Z) \to \Coh(X \times Z)	
\end{equation}
is characterized at the level of objects by the existence of a diagram
\begin{equation*}
	\begin{tikzpicture}
		[baseline=(current  bounding  box.center),thick,>=\arrtip]
		\node (a) at (0,0) {$\Coh(X_\al) \times \Coh(Z_\be)$};
		\node (b) at (5.0,0) {$\Coh(X_\al \times Z_\be)$};
		\node (c) at (0,-1.5) {$\Coh(X) \times \Coh(Z) $};
		\node (d) at (5.0,-1.5) {$\Coh(X \times Z)$};
		\draw[->] (a) to node[above] {$- \boxtimes - $} (b);
		\draw[->] (b) to node[right] {$(i_{\al} \times i_\be)_* $} (d);
		\draw[->] (a) to node[left] {$i_{\al*} \times i_{\be*} $}(c);
		\draw[->] (c) to node[above] {$- \boxtimes - $} (d);
	\end{tikzpicture}
\end{equation*}
for all $\al$, $\be$. 

The above construction also induces an external product of ind-coherent sheaves on coherent ind-geometric stacks. If $X$ and $Z$ are coherent, then as in \cite[\igsecindextcohcase]{CWig} there is a unique extension of (\ref{eq:cohext}) to a continuous functor
\begin{equation}\label{eq:cohextcohcase}
	- \boxtimes -: \IndCoh(X) \otimes \IndCoh(Z) \to \IndCoh(X \times Z),
\end{equation}
where the left-hand term refers to the tensor product in $\PrL$. Fixing $\cF \in \Coh(Z)$, we write $e_{\cF,X}: \IndCoh(X) \to \IndCoh(X \times Z)$ for the functor induced by (\ref{eq:cohextcohcase}), and $e_{\cF,X}^R$ for its right adjoint. If $X$ and $W$ are coherent and $h: X \to Y$ and $g: W \to Z$ have \stable coherent pullback, then we have isomorphisms
\begin{equation}\label{eq:eFXup*compat}
	e_{g^*(\cF),X} h^* \cong (h \times g)^* e_{\cF,Y}, \hspace{6mm}  h_* e_{g^*(\cF),X}^R \cong e_{\cF,Y}^R  (h \times g)_*
\end{equation}
the left being given by construction and the right by adjunction (per Proposition \ref{prop:coherentup*low*adj}).  Likewise, if $f: X \to Y$ and $g: W \to Z$ are ind-proper and almost ind-finitely presented, and if  $\cF' \in \Coh(W)$, then we have isomorphisms 
\begin{equation}\label{eq:eFXRupper!com}
e_{g_*(\cF'),Y} f_* \cong (f \times g)_* e_{\cF',X},\hspace{7mm}  f^! e_{g_*(\cF'),Y}^R \cong e_{\cF',X}^R (f \times g)^!.
\end{equation}

The extension (\ref{eq:cohextcohcase}) does not quite fit into the framework of lax symmetric monoidal structures, since coherent ind-geometric stacks are not closed under products. However, as explained in Section \ref{sec:indtamemorphisms}, this failure is resolved if we restrict to the further subcategory $\indGStkkadtm \subset \indGStkkcoh$ of admissible, ind-tamely presented ind-geometric stacks. Thus, as in \cite[\igsecindextcohcase]{CWig}, restriction and ind-extension of (\ref{eq:cohcorrindgstkfunctorcoh2}) now yields a lax symmetric monoidal functor
\begin{equation}
	\IndCoh: \Corr(\indGStkkadtm)_{prop,coh} \to \PrL.
\end{equation}

\subsection{Coherent pullback and sheaf Hom}

We now turn to the compatibility between coherent pullback and sheaf Hom. We begin by considering quasi-coherent sheaves on geometric stacks, extending from there to the case of ind-coherent sheaves and then ind-geometric stacks.  

\begin{Proposition}\label{prop:cHomcohpullbasechange}
	Let $X$ and $Y$ be tamely presented truncated geometric stacks, $h: X \to Y$ a morphism with \stable coherent pullback, and $\cF \in \Coh(Y)$. Then the Beck-Chevalley map $h^* \cHom(\cF, \cG) \to \cHom(h^*(\cF), h^*(\cG))$ in $\QCoh(X)$ is an isomorphism for all $\cG \in \Coh(Y)$. 
\end{Proposition}

\begin{proof}
	Let $\phi: U \cong \Spec A \to Y$ be a strictly tamely presented flat cover, and let
	\begin{equation}\label{eq:cHomcohpullbasechange0}
		\begin{tikzpicture}
			[baseline=(current  bounding  box.center),thick,>=\arrtip]
			\newcommand*{\ha}{3};
			\newcommand*{\va}{-1.5};
			\node (aa) at (0,0) {$W$};
			\node (ab) at (\ha,0) {$U$};
			\node (ba) at (0,\va) {$X$};
			\node (bb) at (\ha,\va) {$Y$.};
			\draw[->] (aa) to node[above] {$h'$} (ab);
			\draw[->] (ba) to node[above] {$h $} (bb);
			\draw[->] (aa) to node[left] {$\phi' $} (ba);
			\draw[->] (ab) to node[right] {$\phi $} (bb);
		\end{tikzpicture}
	\end{equation} 
	be a Cartesian diagram. Since $\phi'^*$ is conservative it suffices to show the top left arrow in 
	\begin{equation}\label{eq:cHomcohpullbasechange3}
		\begin{tikzpicture}
			[baseline=(current  bounding  box.center),thick,>=\arrtip]
			\newcommand*{\ha}{4.5}; \newcommand*{\hb}{5.0};
			\newcommand*{\va}{-1.5};
			\node (aa) at (0,0) {$\phi'^* h^* \cHom(\cF,\cG)$};
			\node (ab) at (\ha,0) {$\phi'^*  \cHom(h^*(\cF),h^*(\cG))$};
			\node (ac) at (\ha+\hb,0) {$  \cHom(\phi'^*h^*(\cF),\phi'^*h^*(\cG))$};
			\node (ba) at (0,\va) {$h'^* \phi^* \cHom(\cF,\cG)$};
			\node (bb) at (\ha,\va) {$h'^*  \cHom(\phi^*(\cF),\phi^*(\cG))$};
			\node (bc) at (\ha+\hb,\va) {$\cHom(h'^*\phi^*(\cF),h'^*\phi^*(\cG))$};
			\draw[->] (aa) to node[above] {$ $} (ab);
			\draw[->] (ab) to node[above] {$  $} (ac);
			\draw[->] (ba) to node[above] {$ $} (bb);
			\draw[->] (bb) to node[above] {$ $} (bc);
			\draw[->] (aa) to node[below,rotate=90] {$\sim $} (ba);
			\draw[->] (ac) to node[below,rotate=90] {$\sim $} (bc);
		\end{tikzpicture} 
	\end{equation}
	is an isomorphism. The bottom left and top right arrows are isomorphisms by \cite[\igpropcHomftdbasechange]{CWig}, so it suffices to show the bottom right is. 
	
	By hypothesis $A$ is truncated, so for some $n$ there exists a strictly tame presentation $A \cong \colim A_\al$ of order $n$. For each $\al$ write $u_\al: U \to U_\al \cong \Spec A_\al$ for the associated morphism. Since $\cF$ and $\cG$ are coherent, for some $\al$ there exist $\cF_\al, \cG_\al \in \Coh(U_\al)$ such that $\phi^*(\cG) \cong u_\al^*(\cG_\al)$ and $\phi^*(\cF) \cong u_\al^*(\cF_\al)$ by \cite[Cor. 4.5.1.10]{LurSAG} and flatness of the $u_\al$. The bottom right arrow of (\ref{eq:cHomcohpullbasechange3}) is then also the second factor of
	\begin{equation*}
		\begin{tikzpicture}
			[baseline=(current  bounding  box.center),thick,>=\arrtip]
			\newcommand*{\ha}{4.5}; \newcommand*{\hb}{5.3};
			\node (aa) at (0,0) {$h'^* u_\al^* \cHom(\cF_\al,\cG_\al)$};
			\node (ab) at (\ha,0) {$h'^*  \cHom(u_\al^*(\cF),u_\al^*(\cG))$};
			\node (ac) at (\ha+\hb,0) {$  \cHom(h'^*u_\al^*(\cF),h'^*u_\al^*(\cG)).$};
			\draw[->] (aa) to node[above] {$ $} (ab);
			\draw[->] (ab) to node[above] {$  $} (ac);
		\end{tikzpicture} 
	\end{equation*}
	But $u_\al\circ h'$ is of finite Tor-dimension by Proposition \ref{prop:finitetorapproxgstk}, hence the first factor and the composition are isomorphisms by \cite[\igpropcHomftdbasechange]{CWig}. 
\end{proof}

Now suppose $X$ and $Y$ are coherent, ind-tamely presented ind-geometric stacks, $h: X \to Y$ is a morphism with \stable coherent pullback, and $\cF \in \Coh(Y)$. We will generally assume $X \times X$ and $Y \times Y$ are coherent as well, which is automatic if $X$ and $Y$ are admissible. In this case $e_{\cF,Y}$ and $e_{h^*(\cF),X}$ are functors between compactly generated categories and preserve compact objects, hence $\cHom(\cF,-)$ and $\cHom(h^*(\cF),-)$ are continuous (also, $(h \times h)_*$ is well-defined). By (\ref{eq:eFXup*compat}) we have an isomorphism $e_{\cF,Y}^R (h \times h)_* \cong h_* e_{h^*(\cF),X}^R$ of functors $\IndCoh(X \times X) \to \IndCoh(Y)$, and combining with $(h \times h)_* \Delta_{X*} \cong \Delta_{Y*} h_*$ we obtain an isomorphism $\cHom(\cF, h_*(-)) \cong h_* \cHom(h^*(\cF),-)$. We then have the associated Beck-Chevalley map \begin{equation}\label{eq:hompullback3}
	h^* \cHom (\cF,\cG) \rightarrow \cHom(h^* (\cF), h^*(\cG))
\end{equation}
for any $\cG \in \IndCoh(Y)$. In the geometric case one can check that if we restrict to left bounded $\cG$, (\ref{eq:hompullback3}) is identified with (\ref{eq:hompullback2}) under the equivalences $\IndCoh(-)^+ \cong \QCoh(-)^+$.

\begin{Proposition}\label{prop:eFXRupper*cohpullgeomcase}
Let $X$ and $Y$ be coherent, tamely presented truncated, geometric stacks such that $X \times X$ and $Y \times Y$ are coherent, let $h: X \to Y$ be a morphism with \stable coherent pullback, and let $\cF \in \Coh(Y)$. Then the Beck-Chevalley map $h^* \cHom(\cF, \cG) \to \cHom(h^*(\cF), h^*(\cG))$ is an isomorphism for all $\cG \in \IndCoh(Y)$. 
\end{Proposition}
\begin{Lemma}\label{lem:Homtrunccoh}
	Let $X$ be a locally coherent geometric stack and $\cF, \cG \in \Coh(X)$. Then $\tau^{\leq n} \cHom(\cF, \cG)$ is coherent for all $n$.
\end{Lemma}
\begin{proof}
	First suppose $A \in \CAlgk$ is coherent and $M, N \in \Coh_A$. For any $m$ there exists an exact triangle $P \to M \to Q$ such that $P$ is perfect and $Q \in \Mod_A^{\leq m}$ \cite[Cor. 2.7.2.2]{LurSAG}, yielding an exact triangle
	$$ \cHom(Q,N) \to \cHom(M,N) \to \cHom(P,N).$$
	If $N \in \Mod^{\geq i}$ then $\cHom(Q,N) \in \Mod^{\geq i - m}$, hence it follows from the associated long exact sequence that $H^j\cHom(M,N) \cong H^j\cHom(P,N)$ for $j < i - m -1$. Since $P$ is perfect $\cHom(P, N) \cong P^\vee \ot N$, which is coherent since $N$ is and since $P^\vee$ is perfect (\cite[Prop. 7.2.4.11, Prop. 7.2.4.23]{LurHA}). Since $A$ is coherent $H^j\cHom(M,N)$ is then finitely presented over $H^0(A)$ for $j < i - m -1$. But since $m$ was arbitrary and $\cHom(M,N)$ is bounded below, coherence of $A$ then also implies $\tau^{\leq n} \cHom(M,N)$ is coherent for all $n$. 	
	
	Now let $f: \Spec A \to X$ be a flat cover such that $A$ is coherent. By \cite[\igpropeFXRpropsone]{CWig} we can conflate $\cF$ and $\cG$ with their images in $\QCoh(X)$. It suffices to show $f^*\tau^{\leq n} \cHom(\cF, \cG)$ is coherent. We have $f^*\tau^{\leq n} \cHom(\cF, \cG) \cong \tau^{\leq n} \cHom(f^*(\cF), f^*(\cG))$ by flatness and \cite[\igpropcHomftdbasechange]{CWig}, hence the claim follows from the previous paragraph.  
\end{proof}

\begin{proof} [Proof of Proposition \ref{prop:eFXRupper*cohpullgeomcase}]

As observed above, the hypotheses on $X$ and $Y$ imply that all functors appearing in the statement are continuous. By compact generation of $\IndCoh(Y)$ it then suffices to assume $\cG \in \Coh(Y')$. 

Proposition \ref{prop:cHomcohpullbasechange} implies that the map $h^* \cHom(\cF, \Psi_Y(\cG)) \to \cHom(h^*(\cF), h^*(\Psi_Y(\cG)))$ in $\QCoh(X)$ is an isomorphism, and a more elaborate version of Lemma \ref{lem:BCandPsi} (also using \cite[\igpropeFXRlowerstar]{CWig}) shows that $\Psi_X h^* \cHom(\cF, \cG) \to \Psi_X \cHom(h^*(\cF), h^*(\cG))$ is an isomorphism. Since $\cHom(h^*(\cF), h^*(\cG)$ is left bounded and $\Psi_{X}$ is conservative on $\IndCoh(X)^+$, it suffices to show $h^* \cHom(\cF, \cG)$ is left bounded. Alternatively, the proof of Proposition \ref{prop:cHomcohpullbasechange} extends to show $\phi'^* h^* \cHom(\cF, \cG) \to \phi'^* \cHom(h^*(\cF), h^*(\cG))$ is an isomorphism in $\IndCoh(W)$, and since $\phi'^*$ is conservative on $\IndCoh(X)^+$ we are again reduced to showing $h^* \cHom(\cF, \cG)$ is left bounded. 

For any $n$, Lemma \ref{lem:Homtrunccoh} implies $\tau^{\leq n} \cHom(\cF, \cG)$ is coherent, hence $h^* \tau^{\leq n} \cHom(\cF, \cG)$ is left bounded. Since the standard t-structure is right complete we have $\cHom(\cF, \cG) \cong \colim \tau^{\leq n} \cHom(\cF, \cG)$ by Lemma~\ref{lem:colimpres}. Since it is compatible with filtered colimits, and since $h^*$ is continuous, $h^* \cHom(\cF, \cG)$ is then left bounded if the sheaves $h^* \tau^{\leq n} \cHom(\cF, \cG)$ are uniformly left bounded. 

As in the proof of Proposition \ref{prop:cHomcohpullbasechange}, we fix a strictly tamely presented flat cover $\phi: U \cong \Spec A \to Y$, a diagram (\ref{eq:cHomcohpullbasechange0}), and sheaves $\cF_\al, \cG_\al \in \Coh(U_\al)$ such that $\phi^*(\cF) \cong u^*_\al(\cF_\al)$, $\phi^*(\cG) \cong u^*_\al(\cG_\al)$. Since $\phi'$ is faithfully flat, hence $\phi'^*$ conservative on $\IndCoh(X)^+$, it suffices to show the sheaves $\phi'^* h^* \tau^{\leq n} \cHom(\cF, \cG)$ are uniformly left bounded. But we have
\begin{align*}
	\phi'^* h^* \tau^{\leq n} \cHom(\cF,\cG) &\cong
	h'^* \phi^* \tau^{\leq n} \cHom(\cF,\cG)\\ &\cong 
	h'^* \tau^{\leq n} \cHom(\phi^*(\cF),\phi^*(\cG))\\ &\cong 
	h'^* \tau^{\leq n} \cHom(u^*_\al(\cF_\al),u^*_\al(\cG_\al))\\ &\cong 
	h'^* u^*_\al \tau^{\leq n} \cHom(\cF_\al,\cG_\al),
\end{align*}
where the second and fourth isomorphisms use \cite[\igcorcHomftdbasechangeind]{CWig} and flatness of $\phi'$ and $u'_\al$. The claim now follows since $u_\al \circ h'$ is of finite Tor-dimension by Proposition~\ref{prop:finitetorapproxgstk}.  
\end{proof}

We next consider the extension to the ind-geometric setting. 

\begin{Proposition}\label{prop:eFXRupper*cohpullindgeomcase}
	Let $X$ and $Y$ be coherent, ind-tamely presented ind-geometric stacks such that $X \times X$ and $Y \times Y$ are coherent, let $h: X \to Y$ be a morphism with \stable coherent pullback, and let $\cF \in \Coh(Y)$. Then the Beck-Chevalley map $h^* \cHom(\cF, \cG) \to \cHom(h^*(\cF), h^*(\cG))$ is an isomorphism for all $\cG \in \IndCoh(Y)$. 
\end{Proposition}
\begin{proof}
	Write $\cF \cong i_{\al*}(\cF_\al)$ for some reasonable geometric substack $i_{\al}: Y_\al \to Y$ and $\cF_\al \in \Coh(Y_\al)$. Consider the following diagram, where $p_1$ denotes projection onto the first factor. 
	\begin{equation}\label{eq:eFXRupper*cohpullgeomcase1}
		\begin{tikzpicture}
			[baseline=(current  bounding  box.center),thick,>=\arrtip]
			\newcommand*{\ha}{3.7}; \newcommand*{\hb}{3.7};
			\newcommand*{\va}{-1.5};
			\newcommand*{\vb}{-1.5};
			\node (aa) at (0,0) {$Y$};
			\node (ab) at (\ha,0) {$Y_\al$};
			\node (ba) at (0,\va) {$Y \times Y$};
			\node (bb) at (\ha,\va) {$Y \times Y_\al$};
			\node (bc) at (\ha+\hb,\va) {$Y_\al \times Y_\al$};
			\node (cb) at (\ha,\va+\vb) {$Y$};
			\node (cc) at (\ha+\hb,\va+\vb) {$Y_\al$};
			\draw[<-] (aa) to node[above] {$i_\al $} (ab);
			\draw[->] (ab) to node[above right] {$\Delta_{Y_\al} $} (bc);
			\draw[<-] (ba) to node[above] {$id_Y \times i_{\al} $} (bb);
			\draw[<-] (bb) to node[below] {$i_{\al} \times \id_{Y_\al} $} (bc);
			\draw[->] (aa) to node[left] {$\Delta_Y $} (ba);
			\draw[->] (ab) to node[right] {$\delta_{Y_\al}$} (bb);
			
			\draw[<-] (cb) to node[above] {$i_{\al} $} (cc);
			\draw[->] (ba) to node[below left] {$p_1 $} (cb);
			\draw[->] (bb) to node[right] {$p_1 $} (cb);
			\draw[->] (bc) to node[right] {$p_1 $} (cc);
		\end{tikzpicture}
	\end{equation} 
Recall from \cite[\igpropindhomrelationcoherent]{CWig} that the Beck-Chevalley maps
$$ i_{\al*} \cHom(\cF_{\al}, i^!_{\al}(\cG)) \cong i_{\al*} e_{\cF_\al,Y_\al}^R \Delta_{Y_\al *} i^!_{\al}(\cG)  \to e_{\cF_{\al},Y}^R \delta_{Y_\al *} i^!_{\al}(\cG) \to e_{\cF,Y}^R \Delta_{Y*} \cong \cHom(\cF,\cG)$$
are isomorphisms. Note we implicitly use the isomorphism $e_{\cF,Y}^R \cong   e_{\cF_\al,Y}^R (\id_Y \times i_{\al})^!$ of (\ref{eq:eFXRupper!com}), and the fact that $Y \times Y_\al$ and $Y_\al \times Y_\al$ are coherent by \cite[\iglemaffcgen]{CWig}. We have similar isomorphisms associated to the analogue of (\ref{eq:eFXRupper*cohpullgeomcase1}) involving $X$ and $X_\al$. These are compatible in the sense that they are part of the coherence data for a larger diagram involving $h_*$ and~$h_{\al*}$ (i.e. if (\ref{eq:eFXRupper*cohpullgeomcase1}) is of shape~$K$, we have a diagram in $\PrL$ of shape $K \times \Delta^1$). 
Taking adjoints of $h_*$ and~$h_{\al*}$, we obtain a diagram
		\begin{equation*}
		\begin{tikzpicture}
			[baseline=(current  bounding  box.center),thick,>=\arrtip]
			\newcommand*{\ha}{5.0}; \newcommand*{\hb}{5.4}; \newcommand*{\hc}{4.1}; \newcommand*{\hd}{4.1};
			\newcommand*{\va}{-1.5};
			\node (aa) at (0,0) {$h^* i_{\al*} \cHom(\cF_{\al}, i^!_{\al}(\cG))$};
			\node (ab) at (\ha,0) {$h^* \cHom(\cF,\cG)$};
			\node (ac) at (\ha+\hb,0) {$ \cHom(h^*(\cF),h^*(\cG))$};
			\node (ba) at (0,\va) {$i'_{\al*} h^*_{\al} \cHom(\cF_\al, i^!_{\al}(\cG)) $};
			\node (bb) at (\ha,\va) {$i'_{\al*} \cHom(h^*_{\al}(\cF_{\al}), h^*_{\al} i^!_{\al} (\cG)) $};
			\node (bc) at (\ha+\hb,\va) {$i'_{\al*} \cHom(h^*_{\al}(\cF_{\al}),  i'^!_{\al} h^* (\cG)),$};
			\draw[->] (aa) to node[above] {$\sim $} (ab);
			\draw[->] (ab) to node[above] {$ $} (ac);
			\draw[->] (ba) to node[above] {$ $} (bb);
			\draw[->] (bb) to node[above] {$ $} (bc);
			\draw[->] (aa) to node[below,rotate=90] {$ $} (ba);
			\draw[<-] (ac) to node[below,rotate=90] {$\sim $} (bc);
		\end{tikzpicture} 
	\end{equation*}
containg the map in the statement in the top right. The claim follows since the left arrow is an isomorphism by Proposition \ref{prop:coherentup*low*adj}, the bottom left is by Proposition \ref{prop:eFXRupper*cohpullgeomcase}, and the bottom right is by Proposition \ref{prop:up!up*genICcaseIG}. 
\end{proof}

In \cite{CW2} we will use the following extension of Proposition \ref{prop:eFXRupper*cohpullindgeomcase}. To motivate the condition on the morphism $Y' \to Z$, observe that \emph{any} reasonable presentation of $Y'$ satisfies the requirement if $Z$ is (fiber) smooth, or if $Z$ is weakly smooth and $Y' \to Z$ is ind-tamely presented (Proposition \ref{prop:wprosmoothcohdiaga}). More generally, $Y' \to Z$ thus satisfies the desired condition if it is an ind-tamely presented base change of an ind-tamely presented morphism with weakly smooth target. 

\begin{Proposition}\label{prop:eFXRupper*cohpullindcase}
	Let the following be a diagram of ind-geometric stacks with both squares Cartesian, and where $Z$ is a reasonable ind-geometric stack.
		\begin{equation*}
		\begin{tikzpicture}
			[baseline=(current  bounding  box.center),thick,>=\arrtip]
			\newcommand*{\ha}{3}; \newcommand*{\hb}{3};
			\newcommand*{\va}{-1.5};
			\node (aa) at (0,0) {$X'$};
			\node (ab) at (\ha,0) {$X \times Z$};
			\node (ac) at (\ha+\hb,0) {$X$};
			\node (ba) at (0,\va) {$Y'$};
			\node (bb) at (\ha,\va) {$Y \times Z$};
			\node (bc) at (\ha+\hb,\va) {$Y$};
			\draw[->] (aa) to node[above] {$\delta_X $} (ab);
			\draw[->] (ab) to node[above] {$ $} (ac);
			\draw[->] (ba) to node[above] {$\delta_Y $} (bb);
			\draw[->] (bb) to node[above] {$ $} (bc);
			\draw[->] (aa) to node[left] {$h' $} (ba);
			\draw[->] (ab) to node[right] {$ $} (bb);
			\draw[->] (ac) to node[right] {$h $} (bc);
		\end{tikzpicture}
	\end{equation*} 
Suppose that all stacks in the diagram are coherent and ind-tamely presented, and that $X' \times X'$ and $Y' \times Y'$ are as well. 
Suppose also that $h$ has \stable coherent pullback and that the composition of the bottom arrows is ind-proper and almost ind-finitely presented. 
Finally, suppose that $Y'$ can be written as a filtered colimit $Y' \cong \colim Y'_\al$ of reasonable ind-geometric stacks along almost ind-finitely presented ind-closed immersions such that the induced map $Y'_\al \to Z$ has \stable coherent pullback for all $\al$. 
Then for any $\cF \in \Coh(Z)$ and $\cG \in \IndCoh(Y')$, the composition 
	\begin{equation*}
		h^* e_{\cF,Y}^R \delta_{Y*} (\cG) \to  e_{\cF,X}^R (h \times \id_Z)^* \delta_{Y*} (\cG) \to e_{\cF,X}^R  \delta_{X*} h'^* (\cG)
	\end{equation*}
	of Beck-Chevalley maps is an isomorphism. 
\end{Proposition}

\begin{proof}
We can factor $\delta_Y$ as the top row of the following diagram 
	\begin{equation}\label{eq:eFXRupper*cohpullindcase1}
		\begin{tikzpicture}
			[baseline=(current  bounding  box.center),thick,>=\arrtip]
			\newcommand*{\ha}{2.8}; \newcommand*{\hb}{3.7}; \newcommand*{\hc}{3.7};
			\newcommand*{\va}{-1.5};
			\newcommand*{\vb}{-1.5};
			\node (aa) at (0,0) {$Y'$};
			\node (ab) at (\ha,0) {$Y' \times Y'$};
			\node (ac) at (\ha+\hb,0) {$Y' \times Z$};
			\node (ad) at (\ha+\hb+\hc,0) {$Y \times Z$};
			\node (bc) at (\ha+\hb,\va) {$Y'$};
			\node (bd) at (\ha+\hb+\hc,\va) {$Y$.};
			\draw[->] (aa) to node[above] {$\Delta_{Y'} $} (ab);
			\draw[->] (ab) to node[above] {$\id_{Y'} \times g $} (ac);
			\draw[->] (ac) to node[above] {$f \times \id_Z $} (ad);
			\draw[->] (ab) to node[below left] {$ $} (bc);
			\draw[->] (ac) to node[right] {$ $} (bc);
			\draw[->] (ad) to node[right] {$ $} (bd);
			\draw[->] (bc) to node[above] {$f$} (bd);
		\end{tikzpicture}
	\end{equation} 
Here $f: Y' \to Y$ and $g: Y' \to Z$ are the compositions of $\delta_Y$ and the projections, and we define $f': X' \to X$ and $g': X' \to Z$ similarly. By hypothesis $f$ is ind-proper and almost ind-finitely presented, hence $h'$ has \stable coherent pullback. 

Suppose first that $g$ itself has \stable coherent pullback. By hypothesis $f$ is ind-proper and almost ind-finitely presented, hence $h'$ has \stable coherent pullback, hence so does $g'$ since $g' \cong g \circ h'$. We consider the composition
\begin{equation*}
	f_* \cHom(g^*(\cF), -) \cong f_* e_{\cF,Y'}^R (\id_{Y'} \times g)_* \Delta_{Y' *} \to e_{\cF,Y}^R \delta_{Y*},
\end{equation*}
where the first factor is associated to the triangle in (\ref{eq:eFXRupper*cohpullindcase1}) by (\ref{eq:eFXup*compat}). The composition is an isomorphism since the second factor is by \cite[\igpropeFXRlowerstar]{CWig}. We have a similar composition associated to the analogue of (\ref{eq:eFXRupper*cohpullindcase1}) whose top row factors $\delta_X$. These are compatible in the sense that they are part of the coherence data for a larger diagram involving $h_*$ and $h'_*$. Taking adjoints of $h_*$ and $h'_*$, we obtain a diagram
\begin{equation*}
	\begin{tikzpicture}
		[baseline=(current  bounding  box.center),thick,>=\arrtip]
		\newcommand*{\ha}{4.5}; \newcommand*{\hb}{5.0};
		\newcommand*{\va}{-1.5};
		\node (aa) at (0,0) {$h^* e_{\cF,Y}^R \delta_{Y*}$};
		\node (ab) at (\ha,0) {$e_{\cF,X}^R (h \times \id_Z)^* \delta_{Y*}$};
		\node (ac) at (\ha+\hb,0) {$ e_{\cF,X}^R \delta_{X*} h^*$};
		\node (ba) at (0,\va) {$ h^* f_* \cHom(g^*(\cF),-)$};
		\node (bb) at (\ha,\va) {$f'_* h'^* \cHom(g^*(\cF),-)$};
		\node (bc) at (\ha+\hb,\va) {$f'_*  \cHom(g'^*(\cF),h'^*(-))$};
		\draw[->] (aa) to node[above] {$ $} (ab);
		\draw[->] (ab) to node[above] {$  $} (ac);
		\draw[->] (ba) to node[above] {$ $} (bb);
		\draw[->] (bb) to node[above] {$ $} (bc);
		\draw[<-] (aa) to node[below,rotate=90] {$\sim $} (ba);
		\draw[<-] (ac) to node[below,rotate=90] {$\sim $} (bc);
	\end{tikzpicture} 
\end{equation*}
whose top row gives the composition in the statement.  
The bottom left arrow is an isomorphism by Proposition \ref{prop:coherentup*low*adj} and the bottom right is by Proposition \ref{prop:eFXRupper*cohpullgeomcase}, hence the top composition is an isomorphism as well. 

For the general case, first note that our hypotheses imply that all functors appearing in the statement are continuous. By compact generation of $\IndCoh(Y')$ it then suffices to assume $\cG \in \Coh(Y')$. Since $\Coh(Y') \cong \colim \Coh(Y'_\al)$ (Proposition \ref{prop:cohonindgstks}), we may write $\cG \cong i_{\al*}(\cG_\al)$ for some $\al$ and some $\cG_\al \in \Coh(Y'_\al)$. Write $i'_\al: X'_\al \to X'$ and $h_\al: X'_\al \to Y'_\al$ for the base changes of $i_\al$ and $h$, noting that $X'_\al$ and $Y'_\al$ are coherent by \cite[\iglemaffcgen]{CWig}. By the previous paragraph the composition
	\begin{equation*}
	h^* e_{\cF,Y}^R \delta_{Y*}i_{\al*} (\cG_\al)  \to e_{\cF,X}^R  \delta_{X*} h'^* i_{\al*} (\cG_\al) \to e_{\cF,X}^R  \delta_{X*} i'_{\al*} h^*_{\al}(\cG_\al)
\end{equation*}
is an isomorphism, and the claim follows since the last factor is an isomorphism by Proposition~\ref{prop:coherentup*low*adj}. 
\end{proof}

\section{Kernels and integral transforms}\label{sec:kernels}

Given proper maps $Y_1 \to Z$ and $Y_2 \to Z$ of smooth varieties, the integral transform associated to a kernel $\cK \in \Coh(Y_1 \times_Z Y_2)$ is the functor 
\begin{align}\label{eq:inttransformula}
	\Phi_{\cK}: \Coh(Y_1) & \to \Coh(Y_2) \\
	\cF & \mapsto \pi_{2*}(\pi_1^*(\cF) \ot \cK).
\end{align}
This construction provides a dictionary between categorical and geometric aspects of coherent sheaf theory. In particular, composition of functors can be described by convolution of kernels, and the adjoint functors $\Phi_{\cK}^R$, $\Phi_{\cK}^L$ can be described by adjoint kernels $\cK^R, \cK^L \in \Coh(Y_2 \times_Z Y_1)$. In this section we describe the (ind-)tamely presented extension of these constructions and results. 

Let $Y_1 \to Z$ and $Y_2 \to Z$ now be morphisms of stacks satisfying the following hypotheses, where $X_{12}:= Y_1 \times_Z Y_2$ and where $\pi_1: X_{12} \to Y_1$, $\pi_2: X_{12} \to Y_2$ are the projections. 
\begin{equation}\label{eq:convhypos}
	\begin{aligned}
		(1) \quad & \text{$X_{12}$ is an admissible, ind-tamely presented ind-geometric stack, } \\
		(2) \quad & \text{$Y_1$ and $Y_2$ are weakly smooth geometric stacks, and} \\
		(3) \quad & \text{$\pi_1$ and $\pi_2$ are ind-proper and almost ind-finitely presented.}
	\end{aligned}
\end{equation}
Let $\delta_{12}: X_{12} \to Y_1 \times X_{12}$ denote the base change of $\Delta_{Y_1}$ along $\pi_1 \times \id_{Y_1}$. Since $Y_1$ is weakly smooth and $\pi_1$ is ind-tamely presened $\delta_{12}$ has coherent pullback. Since $\pi_2$ is ind-proper and almost ind-finitely presented, we obtain a functor
\begin{align*}
	\Phi_{\cK}: \Coh(Y_1) & \to \Coh(Y_2) \\
	\cF & \mapsto \pi_{2*}\delta_{12}^*(\cF \boxtimes \cK).
\end{align*}
Since $\delta_{12} \cong (\pi_1 \times \id_{X_{12}})$, this specializes to (\ref{eq:inttransformula}) when $X_{12}$ is a variety, hence the internal tensor product $\pi_1^*(\cF) \otimes \cK$ is defined. 

Now suppose $Y_3 \to Z$ is another map such that $X_{23} := Y_2 \times_Z Y_3$ satisfies the above hypotheses. The convolution of $\cK$ with  $\cK' \in \Coh(X_{23})$ is defined by
\begin{equation}\label{eq:kernel2}
	\cK' \conv \cK := \pi_{13*}(\delta_{123}^*(\cK \boxtimes \cK')),
\end{equation}
where $\delta_{123}$ and $\pi_{13}$ are defined by the Cartesian squares 
\begin{equation}\label{eq:conv}
	\begin{tikzpicture}[baseline=(current  bounding  box.center),thick,>=\arrtip]
		\newcommand*{\hp}{3.5}; \newcommand*{\hpp}{7}; 
		\newcommand*{\vb}{-1.5}; 
		\node (a) at (0,0) {$X_{12} \times X_{23}$};
		\node (a') at (\hp,0) {$X_{12} \times_{Y_2} X_{23}$};
		\node (a'') at (\hpp,0) {$X_{13}$};
		\node (b) at (0,\vb) {$Y_2 \times X_{23}$};
		\node (b') at (\hp,\vb) {$X_{23}$};
		\node (b'') at (\hpp,\vb) {$Y_3.$};
		
		\draw[<-] (a) to node[above] {$\delta_{123} $} (a');
		\draw[->] (a') to node[above] {$\pi_{13} $} (a'');
		\draw[<-] (b) to node[above] {$\delta_{23} $} (b');
		\draw[->] (b') to node[above] {$\pi_3 $} (b'');
		
		\draw[->] (a) to node[left] {$\pi_2 \times \id_{X_{23}} $} (b);
		\draw[->] (a') to node[right] {$ $} (b');
		\draw[->] (a'') to node[right] {$\pi_3 $} (b'');
	\end{tikzpicture}
\end{equation}
The vertical and rightward arrows are ind-proper and almost ind-finitely presented by (3), while the leftward arrows are ind-tamely presented base changes of $\Delta_{Y_2}$, hence have coherent pullback by~(2). In particular $\cK' \conv \cK$ is again coherent. Moreover, the left square defines an isomorphism $\pi_{23*} \delta_{123}^* \cong \delta_{23}^* (\pi_2 \times \id_{X_{23}})_*$ of functors $\Coh(X_{12} \times X_{23}) \to \Coh(X_{23})$, which then induces the following isomorphism. 

\begin{Proposition}
Given $\cK \in \Coh(X_{12})$ and $\cK' \in \Coh(X_{13})$ as above, there is an isomorphism $\Phi_{\cK' \conv \cK} \cong \Phi_{\cK'} \circ \Phi_\cK$ of functors $\Coh(Y_1) \to \Coh(Y_3)$. 
\end{Proposition}

Note that the diagonal $e_1: Y_1 \to X_{11}$ is ind-proper and almost ind-finitely presented since $\pi_1$ is and since $\pi_1 \circ e_1$ is the identity \cite[\igpropindPtwoofthreeprops]{CWig}. One checks that the kernel $e_{1*}(\cO_{Y_1})$ represents the identity functor. A kernel $\cK$ is right adjointable if there exists $\cK^R \in \Coh(X_{21})$ and maps $u: e_{1*}(\cO_{Y_1}) \to \cK^R \conv \cK$ and $c: \cK \conv \cK^R \to e_{2*}(\cO_{Y_2})$ such that the compositions
\begin{equation}\label{eq:kerneladj} \cK \xrightarrow{\id \conv u} \cK \conv \cK^R \conv \cK \xrightarrow{\id \conv c} \cK \quad\quad \cK^R \xrightarrow{\id \conv u} \cK^R \conv \cK \conv \cK^R \xrightarrow{\id \conv c} \cK^R 
\end{equation}
are homotopic to the identity maps. Likewise, $\cK$ is left adjointable if it is the right adjoint of another kernel $\cK^L$. It follows formally that adjointability of kernels implies adjointability of the induced functors, i.e. $\Phi_{\cK^R}$ and $\Phi_{\cK^L}$ are adjoints of $\Phi_{\cK}$. A priori, however, adjointability of kernels is stronger, as it asks not only that the adjoint functors are represented by kernels, but that the unit/counit natural transformations are represented by morphisms of kernels as in (\ref{eq:kerneladj}). 

Under the hypotheses (\ref{eq:convhypos}) we then have the following result. The core of the proof is deferred to \cite{CW2}, but we emphasize that the relevant proofs there depend crucially on Propositions~\ref{prop:up!up*genICcaseIG} and \ref{prop:eFXRupper*cohpullindcase}. 

\begin{Proposition}\label{prop:kernel3}
	Any kernel $\cK \in \Coh(X_{12})$ is left and right adjointable, with adjoints given by $\cK^L := \cHom(\cK, \pi_1^! (\O_{Y_1}))$ and $\cK^R := \cHom(\cK, \pi_2^! (\O_{Y_2}))$. 
\end{Proposition}
\begin{proof}
Let $Y:= Y_1 \sqcup Y_2$ and $X: = Y \times_Z Y$. As $X_{12}$ is a component of $X$ (or a union of components), we can regard $\cK$ as an object of $\Coh(X)$. As a monoidal category under convolution, $\Coh(X)$ is rigid by \cite[\crpropcohrigidity]{CW2}. The left and right duals of $\cK$ in $\Coh(X)$ are given by the above formulas by \cite[\crpropcohdualformula]{CW2}. But the notion of adjointability of kernels is the same as the monoidal notion of dualizability in this case (up to composing with the natural maps between $e_*(\cO_Y)$ and its summands $e_{1*}(\cO_{Y_1})$ and $e_{2*}(\cO_{Y_2})$), hence the claim follows. 
\end{proof}

\bibliographystyle{amsalpha}
\bibliography{bibKoszul}

\end{document}